\documentclass[%
aip,
amsmath,amssymb,
preprint,%
]{revtex4-1}

\usepackage{graphicx}
\usepackage{dcolumn}
\usepackage{bm}

\usepackage{multirow}%
\usepackage{amsmath,amssymb,amsfonts}%
\usepackage{amsthm}%
\usepackage{mathrsfs}%
\usepackage[title]{appendix}%
\usepackage{xcolor}%
\usepackage{textcomp}%
\usepackage{booktabs}%
\usepackage{algorithm}%
\usepackage{algorithmicx}%
\usepackage{algpseudocode}%
\usepackage{listings}%
\usepackage{mhchem}%
\usepackage{float}
\usepackage{subfigure}

\usepackage[utf8]{inputenc}
\usepackage[T1]{fontenc}
\usepackage{mathptmx}
\usepackage{etoolbox}

\theoremstyle{thmstyleone}%
\newtheorem{theorem}{Theorem}[section]
\newtheorem{lemma}[theorem]{Lemma}%
\newtheorem{proposition}[theorem]{Proposition}%
\newtheorem{corollary}[theorem]{Corollary}%

\theoremstyle{thmstyletwo}%
\newtheorem{remark}{Remark}[section]%

\theoremstyle{thmstylethree}%
\newtheorem{definition}{Definition}[section]%

\makeatletter
\def\@email#1#2{%
	\endgroup
	\patchcmd{\titleblock@produce}
	{\frontmatter@RRAPformat}
	{\frontmatter@RRAPformat{\produce@RRAP{*#1\href{mailto:#2}{#2}}}\frontmatter@RRAPformat}
	{}{}
}%
\makeatother
\begin{document}
	
	\title[Optimal Fluctuations for Nonlinear Chemical Reaction Systems with General Rate Law]{Optimal Fluctuations for Nonlinear Chemical Reaction Systems with General Rate Law}
	\author{Feng Zhao}
	\author{Jinjie Zhu}%
	\affiliation{State Key Laboratory of Mechanics and Control for Aerospace Structures, College of Aerospace Engineering, Nanjing University of Aeronautics and Astronautics, 29 Yudao Street, Nanjing 210016, China}%
	
	\author{Yang Li}
	\affiliation{School of Automation, Nanjing University of Science and Technology, 200 Xiaolingwei Street, Nanjing 210094, China}
	
	\author{Xianbin Liu}
	\altaffiliation[\textbf{Authors to whom correspondence should be addressed}]{}
	\email{xbliu@nuaa.edu.cn}
	\author{Dongping Jin}
	\altaffiliation[\textbf{Authors to whom correspondence should be addressed}]{}
	\email{jindp@nuaa.edu.cn}
	
	\affiliation{State Key Laboratory of Mechanics and Control for Aerospace Structures, College of Aerospace Engineering, Nanjing University of Aeronautics and Astronautics, 29 Yudao Street, Nanjing 210016, China}%

	\date{\today}
	
	\begin{abstract}
		This paper investigates optimal fluctuations for chemical reaction systems with $N$ species, $M$ reactions, and general rate law. In the limit of large volume, large fluctuations for such models occur with overwhelming probability in the vicinity of the so-called optimal path, which is a basic consequence of the Freidlin-Wentzell theory, and is vital in biochemistry as it unveils the almost deterministic mechanism concealed behind rare noisy phenomena such as escapes from the attractive domain of a stable state and transitions between different metastable states. In this study, an alternative description for optimal fluctuations is proposed in both non-stationary and stationary settings by means of a quantity called prehistory probability in the same setting, respectively. The evolution law of each of them is derived, showing their relationship with the time reversal of a specified family of probability distributions respectively. The law of large numbers and the central limit theorem for the reversed processes are then proved. In doing so, the prehistorical approach to optimal fluctuations for Langevin dynamics is naturally generalized to the present case, thereby suggesting a strong connection between optimal fluctuations and the time reversal of the chemical reaction model. 
	\end{abstract}
	
	\maketitle
	
	\section{Introduction}\label{sec01}
	Macroscopic variables characterizing the dynamics of multitudinous realistic systems in physics, chemistry, and biology typically fluctuate continuously due to environmental or intrinsic randomness. They predominantly oscillate stochastically in the vicinity of a stable state, with large deviations occurring occasionally. Such large fluctuations, although rare, are responsible for numerous phenomena in diverse scientific disciplines, including epigenetic switching in genetic networks \cite{Roma_2005}, atomic migration in crystals \cite{Schuss_1978,Schuss_1979}, unidirectional motion and energy transduction for molecular motors \cite{Zhang_2012}, phase transitions \cite{Schlogl_1972,Hao_Ge_2010}, stochastic and coherence resonance \cite{Freidlin_2000,Freidlin_2001,Muratov_2005} in bistable systems, and so on.
	
	In the last few decades, a great deal of effort has been devoted to the study of large fluctuations for nonlinear systems driven by Gaussian noise (with a flat \cite{Maier_1992,Chinarov_1993,Holcman_2018} or non-flat \cite{Bray_1989,Dykman_1990,Dykman_1993b} spectrum). For these models, there are, broadly speaking, two extensively used techniques  to approach significant fluctuations. The rigorous definition of the optimal path is contingent upon the large deviation principle \cite{Freidlin_2012}, or equivalently, the path integral formulation \cite{Graham_1977,McKane_1990a,McKane_1990b}. Both of them assert that the probability of an event is exponentially dominated by the minimum of the so-called Freidlin-Wentzell action functional, and that the rare large fluctuations, if occurring, are more likely to be close to the minimizer of the functional. This determination of the almost deterministic behavior hidden in a rare stochastic event was then proved to be a key to questions involving the behavior of these processes over infinite time intervals, such as the estimates on the stationary distribution \cite{Freidlin_2012}, the exit time \cite{Naeh_1990}, and the exit point distribution \cite{Maier_1997,Luchinsky_1999}, and so on. Building on these findings, several novel phenomena specific to systems far from equilibrium were also discovered, including the singularities in the patterns of extreme paths \cite{Maier_1993,Dykman_1994b,Dykman_1996,Smelyanskiy_1997}, the non-differentiability of the quasi-potential \cite{Maier_1993,Dykman_1996}, and the coexistence of multiple optimal paths \cite{Dykman_1996}.
	
	The statistical description of optimal fluctuations is attributed to Dykman et al.\cite{Dykman_1992}, who first proposed a quantity termed prehistory probability in the framework of Langevin dynamics, and demonstrated that it not only has the property of pinpointing the location of the optimal path, but also provides the statistical information of nearby trajectories. Such a quantity has recently been shown to be intimately associated with the time reversal of a specified family of probability distributions, connecting the optimal fluctuation with a reversed process \cite{Zhao_2022}. The advantages have been numerically substantiated in the investigation of escapes from chaotic attractors \cite{ANISHCHENKO_2001} and the coexistence of multiple optimal paths \cite{Dykman_1996}. 
	
	This paper concerns large fluctuations for nonlinear chemical reaction systems consisting of $N$ chemical species and $M$ reactions
	\begin{equation}\label{eq010101}
		\nu^+_{i1}S_1+\nu^+_{i2}S_2+\cdots+\nu^+_{iN}S_N\ce{<=>[\mathit{r_{+i}}][\mathit{r_{-i}}]}\nu^-_{i1}S_1+\nu^-_{i2}S_2+\cdots+\nu^-_{iN}S_N,
	\end{equation}
	in which $1\leq i\leq M$. $\nu_{ij}=\nu^-_{ij}-\nu^+_{ij}$ are called stoichiometric coefficients, which are always integers measuring the change in the number of the $j$th species when the $i$th forward or backward reaction occurs. Let $n_j(t)$ be the number of molecules of the $j$th species at the moment $t$. We make the following assumptions.
	
	(a) The system under consideration is confined to a volume $V$, well-stirred and in thermal equilibrium at some constant temperature. To the extent that this happens, it is possible to disregard the positions and velocities of the individual molecules, and instead focus exclusively on events that result in a change to the population vector $\boldsymbol{n}=(n_1,n_2,\cdots,n_N)^\top\in\mathbb{N}^N$ (or the concentration vector $\boldsymbol{x}=\boldsymbol{n}/V\in{V}^{-1}\mathbb{N}^N$) of the chemical species \cite{Gillespie_2007}. This simplifies the problem considerably. 
	
	(b) Each reaction is elementary and microscopically reversible, with the forward rate $r_{+i}(\boldsymbol{n},V):\mathbb{N}^N\times\mathbb{R}_{+}\to\mathbb{R}_{+}\triangleq[0,+\infty)$ and the backward rate $r_{-i}(\boldsymbol{n},V):\mathbb{N}^N\times\mathbb{R}_{+}\to\mathbb{R}_{+}$, which depict the number of occurrences of the $i$th forward and backward reactions per unit time respectively. In addition, there exist functions $R_{\pm i}(\boldsymbol{x}):\mathbb{R}_{+}^N\to\mathbb{R}_{+}$ such that $\lim_{V\to\infty}V^{-1}r_{\pm i}(\boldsymbol{n},V)=R_{\pm i}(\boldsymbol{x})$ for any $\boldsymbol{x}\in\mathbb{R}_{+}^N$ at the macroscopic limit $V\to\infty$, $\boldsymbol{n}/V\to\boldsymbol{x}$ \cite{Horn_1972,Hao_Ge_2017}. In particular, in Delbr\"uck-Gillespie's description of chemical kinetics, $r_{\pm i}(\boldsymbol{n},V)=k_{\pm i}V\prod_{j=1}^{N}{\frac{n_j !}{(n_{j}-\nu^{\pm}_{ij})!V^{\nu^{\pm}_{ij}}}}$ \cite{Kurtz_1972a,Gillespie_1977,Gillespie_2007}. It follows that $R_{\pm i}(\boldsymbol{x})=k_{\pm i}\prod_{j=1}^{N}{x_j^{\nu^{\pm}_{ij}}}$, which restores Waage-Guldberg's law of mass action for macroscopic chemical kinetics \cite{Hao_Ge_2016a}.
	
	At finite volume $V$, Eq. (\ref{eq010101}) defines a homogeneous Markov process with values in $V^{-1}\mathbb{N}^N$, whose trajectories are right continuous piecewise constant functions that can be expressed in terms of the random-time-changed Poisson form \cite{Kurtz_2015,Anderson_2015}
	\begin{equation}\label{eq010102}
			\boldsymbol{x}_{V}(t)=\boldsymbol{x}_{V}(0)
			+V^{-1}\sum_{i=1}^{M}{\boldsymbol{\nu}_{i}\left\{Y_{+i}\left(\int_{0}^{t}{r_{+i}(V\boldsymbol{x}_{V}(s),V)\mathrm{d}s}\right)-Y_{-i}\left(\int_{0}^{t}{r_{-i}(V\boldsymbol{x}_{V}(s),V)\mathrm{d}s}\right)\right\}},
	\end{equation}
	where $\boldsymbol{\nu}_{i}=(\nu_{i1},\nu_{i2},\cdots,\nu_{iN})^\top$, and $Y_{\pm i}(u)$ are $2M$ independent, standard Poisson processes.
	
	Based on the large deviation principle for Markov jump processes \cite{Freidlin_2012}, the analogous concept of the optimal path in chemical kinetics was proposed in [\onlinecite{Dykman_1994a}], and successfully applied to the estimation of the behavior of $\boldsymbol{x}_V(t)$ over very long time intervals, including the asymptotics of the stationary distribution \cite{Shwartz_1995}, the mean exit time \cite{Knessl_1984,Shwartz_1995}, and, more recently, the exit location distribution \cite{Pakdaman_2010}. However, to the best of our knowledge, the relationship between optimal paths and the time reversal of the associated process (\ref{eq010102}) has not been established. The objective of this paper is to extend the prehistorical description of optimal fluctuations for Langevin systems to the domain of chemical reaction models. The program's structure is outlined below: Several limit theorems and the associated results are presented in Sec. \ref{sec02}. In Sec. \ref{sec03}, it is demonstrated that the time reversal of a given family of probability distributions can be characterized by a Markov jump process of the same type as Eq. (\ref{eq010102}). This naturally interprets the prehistory probabilities in non-stationary and stationary settings as the conditional probability of a specific reversed stochastic process in the same setting respectively. The prehistorical description of optimal fluctuations on both finite and infinite time intervals is then established in Secs. \ref{sec04} and \ref{sec05}, respectively, by means of the law of large numbers and the central limit theorem of the reversed processes. Numerical examples are exhibited in Sec. \ref{sec06}. The conclusions of this study are set out in Sec. \ref{sec07}. The proofs and algorithms are relegated to the appendix.
	
	\section{Preliminaries}\label{sec02}
	\subsection{Kurtz's Limit Theorems}\label{sec0201}
	As $V\to\infty$, stochastic chemical reaction models in the form of Eq. (\ref{eq010101}) are "almost deterministic", which is a consequence of the type of the law of large numbers (cf. [\onlinecite[Theorem 2.2]{Kurtz_1978}]), and can be formulated as follows.
	\begin{theorem}\label{theo0201}
		Assume that
		\begin{itemize}
			\item[(a)] there exists a constant $\Gamma_0$ such that for any $1\leq i\leq M$, $\boldsymbol{x}\in{V}^{-1}\mathbb{N}_+^N$ and sufficiently large $V$,
			\begin{equation*}
				\begin{aligned}
					\vert V^{-1}r_{\pm i}(V\boldsymbol{x},V)\vert\leq \Gamma_0,\\
					\vert V^{-1}r_{\pm i}(V\boldsymbol{x},V)-R_{\pm i}(\boldsymbol{x})\vert\leq\frac{\Gamma_0}{V};
				\end{aligned}
			\end{equation*} 
			\item[(b)] $\boldsymbol{F}(\boldsymbol{x})\triangleq\sum_{i=1}^{M}\boldsymbol{\nu}_i(R_{+i}(\boldsymbol{x})-R_{-i}(\boldsymbol{x}))$ is Lipschitz continuous.
		\end{itemize}
		Then if $\lim_{V\to\infty}\boldsymbol{x}_V(0)=\boldsymbol{x}_{\infty}(0)$ a.s. (almost surely), for any $T>0$,
		\begin{equation}\label{eq020101}
			\lim_{V\to\infty}\sup_{t\in [0,T]}\vert\boldsymbol{x}_V(t)-\boldsymbol{x}_\infty(t)\vert=0,\;\text{a.s.},
		\end{equation}
		in which $\boldsymbol{x}_\infty(t)$ is the unique solution of the deterministic equation
		\begin{equation}\label{eq020102}
			\boldsymbol{x}_{\infty}(t)=\boldsymbol{x}_\infty(0)+\int_{0}^t{\boldsymbol{F}(\boldsymbol{x}_\infty(s))\,\mathrm{d}s}.
		\end{equation}
		
		Furthermore, if $\vert\boldsymbol{x}_V(0)-\boldsymbol{x}_{\infty}(0)\vert\leq{O}\left({1}/{\sqrt{V}}\right)$ a.s., then
		\begin{equation}\label{eq020103}
			\sup_{t\in [0,T]}\vert\boldsymbol{x}_V(t)-\boldsymbol{x}_\infty(t)\vert\leq{O}\left(\frac{1}{\sqrt{V}}\right),\;\text{a.s.}.
		\end{equation}
	\end{theorem}
	
	More precisely, in some neighborhood of $\{\boldsymbol{x}_\infty(t),t\in [0,T]\}$ (in the space $\mathcal{D}([0,T];\mathbb{R}^N)$ with a uniform-convergence topology), local fluctuations for the chemical reaction model can be approximated as a diffusion process, which was rigorously proved by Kurtz (cf. [\onlinecite[Theorem 3.3]{Kurtz_1978}]) by means of a Poisson representation of the form (\ref{eq010102}). One can also refer to the Kramers-Moyal expansion of the chemical master equation in order to achieve the same diffusion approximation \cite{Kampen_1961,Knessl_1984}.
	\begin{theorem}\label{theo0202}
		Assume that the conditions in Theorem \ref{theo0201} hold, and further suppose that for each $i$, $R_{\pm i}(\boldsymbol{x})$ are Lipschitz continuous. Let $\boldsymbol{y}_V(t)$ be a diffusion process that satisfies
		\begin{equation}\label{eq020104}
			\begin{aligned}
				\boldsymbol{y}_V(t)=&\boldsymbol{y}_V(0)+\int_{0}^{t}\boldsymbol{F}(\boldsymbol{y}_V(s))\mathrm{d}s\\
				&+\frac{1}{\sqrt{V}}\sum_{i=1}^{M}\boldsymbol{\nu}_{i}\left(\int_0^t\sqrt{R_{+i}(\boldsymbol{y}_V(s))}\mathrm{d}w_{+i}(s)-\int_0^t\sqrt{R_{-i}(\boldsymbol{y}_V(s))}\mathrm{d}w_{-i}(s)\right),
			\end{aligned}
		\end{equation}
		where $w_{\pm i}(t)$ are $2M$ independent, standard Wiener processes. 
		
		If $\vert\boldsymbol{x}_V(0)-\boldsymbol{y}_{V}(0)\vert\leq{O}\left({1}/{{V}}\right)$ a.s., then for any $T>0$,
		\begin{equation}\label{eq020105}
			\sup_{t\in [0,T]}\vert\boldsymbol{x}_V(t)-\boldsymbol{y}_V(t)\vert\leq{O}\left(\frac{\ln{V}}{V}\right),\;\text{a.s.}.
		\end{equation}
	\end{theorem}
	
	Moreover, the following result regarding the central limit theorem for the chemical reaction model shows that local deviations of order ${O}(1/\sqrt{V})$ are approximately Gaussian (cf. [\onlinecite[Theorem 4.4]{Kurtz_1978}]).
	\begin{theorem}\label{theo0203}
		Assume that the conditions in Theorems \ref{theo0201} and \ref{theo0202} hold, and also suppose that
		\begin{itemize}
			\item[(a)] there exists a constant $L_0$ such that for any $\boldsymbol{x},\boldsymbol{y}\in \mathbb{R}_+^N$,
			\begin{equation*}
				\sum_{i=1}^{M}\left\vert\boldsymbol{\nu}_i\right\vert^2\left(\left\vert \sqrt{R_{+i}(\boldsymbol{x})}-\sqrt{R_{+i}(\boldsymbol{y})}\right\vert^2+\left\vert \sqrt{R_{-i}(\boldsymbol{x})}-\sqrt{R_{-i}(\boldsymbol{y})}\right\vert^2\right)\leq L_0\left\vert\boldsymbol{x}-\boldsymbol{y}\right\vert^2;
			\end{equation*}
			\item[(b)] $\boldsymbol{F}(\boldsymbol{x})$ is bounded, and has bounded and continuous partial derivatives up to order $2$ inclusive.
		\end{itemize}
		
		Let $\boldsymbol{z}(t)$ be the Gaussian process defined by
		\begin{equation}\label{eq020106}
			\begin{aligned}
				\boldsymbol{z}(t)=&\int_{0}^{t}\nabla\boldsymbol{F}(\boldsymbol{x}_{\infty}(s))\cdot\boldsymbol{z}(s)\mathrm{d}s\\
				&+\sum_{i=1}^{M}\boldsymbol{\nu}_i\left(\int_{0}^{t}\sqrt{R_{+i}(\boldsymbol{x}_\infty(s))}\mathrm{d}w_{+i}(s)-\int_{0}^{t}\sqrt{R_{-i}(\boldsymbol{x}_\infty(s))}\mathrm{d}w_{-i}(s)\right).
			\end{aligned}
		\end{equation}
		
		If $\vert\boldsymbol{x}_V(0)-\boldsymbol{x}_{\infty}(0)\vert\leq{O}\left({1}/{{V}}\right)$ a.s., then for any $T>0$,
		\begin{equation}\label{eq020107}
			\sup_{t\in [0,T]}\left\vert \sqrt{V}\left(\boldsymbol{x}_V(t)-\boldsymbol{x}_\infty(t)\right)-\boldsymbol{z}(t)\right\vert\leq{O}\left(\frac{\ln{V}}{\sqrt{V}}\right),\;\text{a.s.}.
		\end{equation}
	\end{theorem}
	\subsection{Large Deviation Principle}\label{sec0202}
	Diffusion approximation to the chemical reaction systems has been demonstrated to be applicable in a variety of situations, including stochastic simulations \cite{Gillespie_2000} and estimates of the stationary distribution in the vicinity of a stable state \cite{Dykman_1993}. Nonetheless, it has been shown that the approximation becomes invalid if the major contribution of the event under consideration is dominated by large fluctuations \cite{Knessl_1984,Pakdaman_2010,Freidlin_2012}. As demonstrated in the seminal works of Freidlin and Wentzell \cite{Freidlin_2012} (cf. also [\onlinecite{Shwartz_1995}]), the probabilities of rare events for $\boldsymbol{x}_V(t)$ can be described by a rate function (or an action functional) that does not coincide with that of the approximated diffusion process $\boldsymbol{y}_V(t)$. Consequently, the large deviation principle for the original chemical reaction model, as opposed to its diffusion approximation, is imperative for the rigorous definition of the optimal path.
	
	For arbitrary $T>0$, denote by $\mathcal{D}([0,T];\mathbb{R}^N)$ the space containing all the functions of the variable $t\in [0,T]$ with values in $\mathbb{R}^N$ that are right continuous with left limits. Let $\Lambda$ be a collection of strictly increasing real functions $\lambda$ on $[0,T]$, such that $\lambda(0)=0$, $\lambda(T)=T$, and
	\begin{equation*}
		\gamma(\lambda)\triangleq\sup_{0\leq s\leq t\leq T}{\left\vert\ln\frac{\lambda(s)-\lambda(t)}{s-t}\right\vert}<\infty.
	\end{equation*}
	Define a metric on $\mathcal{D}([0,T];\mathbb{R}^N)$ by
	\begin{equation*}
		\rho\left(\left\{\boldsymbol{x}(t):t\in [0,T]\right\},\left\{\boldsymbol{y}(t):t\in [0,T]\right\}\right)\triangleq\inf_{\lambda\in\Lambda}\left\{\max\left(\gamma(\lambda),\sup_{t\in[0,T]}\left\vert\boldsymbol{x}(t)-\boldsymbol{y}(\lambda(t))\right\vert\right)\right\}.
	\end{equation*}
	It follows that $\left(\mathcal{D}([0,T];\mathbb{R}^N),\rho\right)$ is a Polish space (a complete separable metric space) called the Skorohod space (cf. [\onlinecite[Theorem A.55]{Shwartz_1995}]).
	
	Denote 
	\begin{equation*}
		H(\boldsymbol{x},\boldsymbol{\alpha})\triangleq\sum_{i=1}^{M}\left(R_{+i}(\boldsymbol{x})\left(e^{\boldsymbol{\nu}_i\cdot\boldsymbol{\alpha}}-1\right)+R_{-i}(\boldsymbol{x})\left(e^{-\boldsymbol{\nu}_i\cdot\boldsymbol{\alpha}}-1\right)\right),
	\end{equation*}
	and 
	\begin{equation*}
		L(\boldsymbol{x},\boldsymbol{\beta})\triangleq\sup_{\boldsymbol{\alpha}\in\mathbb{R}^N}\left(\boldsymbol{\alpha}\cdot\boldsymbol{\beta}-H(\boldsymbol{x},\boldsymbol{\alpha})\right),
	\end{equation*}
	then the Freidlin-Wentzell-type large deviation principle (cf. [\onlinecite[Theorem 5.1; Proposition 5.49]{Shwartz_1995}]) for stochastic chemical reaction models can be elaborated as follows, with the corresponding rate function on $\mathcal{D}([0,T];\mathbb{R}^N)$ defined by
	\begin{equation}\label{eq020201}
		I_{[0,T]}\left(\left\{\boldsymbol{\phi}(t):t\in[0,T]\right\}\right)\triangleq\begin{cases}
			\int_{0}^{T}L(\boldsymbol{\phi}(s),\dot{\boldsymbol{\phi}}(s))\mathrm{d}s, &\text{if $\boldsymbol{\phi}(t)$ is absolutely continuous}, \\ 
			\infty,&\text{otherwise}.
		\end{cases}
	\end{equation}
	
	\begin{theorem}\label{theo0204}
		Assume that for each $i$, $\ln{R_{\pm i}(\boldsymbol{x})}$ are bounded and Lipschitz continuous, and $V^{-1}r_{\pm i}(V\boldsymbol{x},V)$ converge to $R_{\pm i}(\boldsymbol{x})$ uniformly with respect to $\boldsymbol{x}\in\mathbb{R}_+^N$ as $V\to\infty$. Then 
		\begin{itemize}
			\item[(a)] $I_{[0,T]}$ is a good rate function on $\left(\mathcal{D}([0,T];\mathbb{R}^N),\rho\right)$, i.e., $I_{[0,T]}$ is lower semi-continuous on $\left(\mathcal{D}([0,T];\mathbb{R}^N),\rho\right)$, and the set $\cup_{\boldsymbol{x}_{0}\in{K}}\Phi_{\boldsymbol{x}_{0},[0,T]}(s)$ is compact for any compact subset $K\subset\mathbb{R}_{+}^N$, where $\Phi_{\boldsymbol{x}_{0},[0,T]}(s)$ is defined for each $\boldsymbol{x}_{0}\in\mathbb{R}_{+}^N$ and $s>0$ by
			\begin{equation*}
				\Phi_{\boldsymbol{x}_{0},[0,T]}(s)\triangleq\{\boldsymbol{\phi}(t):t\in[0,T],\boldsymbol{\phi}(0)=\boldsymbol{x}_{0},I_{[0,T]}\left(\{\boldsymbol{\phi}(t):t\in[0,T]\}\right)\leq{s}\};
			\end{equation*}
			\item[(b)] for any open set $\mathcal{G}\subset\mathcal{D}([0,T];\mathbb{R}^N)$, and uniformly for any $\boldsymbol{x}_0$ in each compact subset of $\mathbb{R}_+^N$,
			\begin{equation}\label{eq020202}
				\begin{aligned}
					\liminf_{V\to\infty}V^{-1}&\ln P_{\boldsymbol{x}_0}\left(\left\{\boldsymbol{x}_V(t):t\in[0,T]\right\}\in \mathcal{G}\right)\geq\\
					&-\inf\left\{I_{[0,T]}\left(\left\{\boldsymbol{\phi}(t):t\in[0,T]\right\}\right):\boldsymbol{\phi}(0)=\boldsymbol{x}_0,\left\{\boldsymbol{\phi}(t):t\in[0,T]\right\}\in \mathcal{G}\right\};
				\end{aligned}
			\end{equation}
			\item[(c)] for any closed set $\mathcal{F}\subset\mathcal{D}([0,T];\mathbb{R}^N)$, and $\boldsymbol{x}_0\in\mathbb{R}_+^N$
			\begin{equation}\label{eq020203}
				\begin{aligned}
					\limsup_{V\to\infty}V^{-1}&\ln P_{\boldsymbol{x}_0}\left(\left\{\boldsymbol{x}_V(t):t\in[0,T]\right\}\in \mathcal{F}\right)\leq\\
					&-\inf\left\{I_{[0,T]}\left(\left\{\boldsymbol{\phi}(t):t\in[0,T]\right\}\right):\boldsymbol{\phi}(0)=\boldsymbol{x}_0,\left\{\boldsymbol{\phi}(t):t\in[0,T]\right\}\in \mathcal{F}\right\}.
				\end{aligned}
			\end{equation}
		\end{itemize}
		Here, $P_{\boldsymbol{x}_0}(\cdot)$ means the probability of the process $\boldsymbol{x}_V(t)$ conditioned on $\boldsymbol{x}_V(0)=\boldsymbol{x}_0$.
	\end{theorem}
	\subsection{Optimal Fluctuations on Finite Time Intervals}\label{sec0203}
	Define a mapping by $\boldsymbol{\psi}(\left\{\boldsymbol{\phi}(t):t\in[0,T]\right\})\triangleq\boldsymbol{\phi}(T)$. It can be shown easily that $\boldsymbol{\psi}$ is a continuous mapping from $\mathcal{D}([0,T];\mathbb{R}^N)$ to $\mathbb{R}^N$. Utilizing the contraction principle (cf. [\onlinecite[Chapter 3, Theorem 3.1]{Freidlin_2012}]) and Borovkov's description of the large deviation principle (cf. [\onlinecite[Chapter 3, Theorem 3.4]{Freidlin_2012}]), one can straightforwardly prove the following proposition.
	\begin{proposition}\label{theo0205}
		Assume that the conditions in Theorem \ref{theo0204} hold. For each $T>0$ and $\boldsymbol{x}_0, \boldsymbol{x}_T\in\mathbb{R}_+^N$, define
		\begin{equation}\label{eq020301}
			S(\boldsymbol{x}_T,T\vert\boldsymbol{x}_0)\triangleq\inf_{\boldsymbol{\phi}(0)=\boldsymbol{x}_0,\,\boldsymbol{\phi}(T)=\boldsymbol{x}_T}I_{[0,T]}\left(\left\{\boldsymbol{\phi}(t):t\in[0,T]\right\}\right).
		\end{equation}
		Then for each $\boldsymbol{x}_0\in\mathbb{R}_+^N$ and each $D\subset\mathbb{R}_+^N$ such that 
		\begin{equation*}
			\inf_{\boldsymbol{x}_T\in\bar{D}}S(\boldsymbol{x}_T,T\vert\boldsymbol{x}_0)=\inf_{\boldsymbol{x}_T\in{D}^o}S(\boldsymbol{x}_T,T\vert\boldsymbol{x}_0),
		\end{equation*}
		we have
		\begin{equation}\label{eq020302}
			\lim_{V\to\infty}V^{-1}\ln P_{\boldsymbol{x}_0}\left(\boldsymbol{x}_V(T)\in D\right)=-\inf_{\boldsymbol{x}_T\in{D}}S(\boldsymbol{x}_T,T\vert\boldsymbol{x}_0),
		\end{equation}
		in which $\bar{D}$ and $D^o$ are the closure and the interior of $D$, respectively.
	\end{proposition}
	
	Now, we can define the concept of a non-stationary optimal path in the following way.
	\begin{definition}\label{def0201}
		For each $T>0$ and $\boldsymbol{x}_0, \boldsymbol{x}_T\in\mathbb{R}_+^N$, a path $\left\{\boldsymbol{\phi}(t):t\in[0,T]\right\}$ is said to be a non-stationary optimal path (NOP) that connects $\boldsymbol{x}_0$ with $\boldsymbol{x}_T$ in the time span $T$ if it is a minimizer of Eq. (\ref{eq020301}), i.e., a path that satisfies $\boldsymbol{\phi}(0)=\boldsymbol{x}_0$, $\boldsymbol{\phi}(T)=\boldsymbol{x}_T$ and
		\begin{equation*}
			S(\boldsymbol{x}_T,T\vert\boldsymbol{x}_0)=I_{[0,T]}\left(\left\{\boldsymbol{\phi}(t):t\in[0,T]\right\}\right)<\infty.
		\end{equation*}
		We denote it by $\{\boldsymbol{\phi}_{\text{NOP}}(t;\boldsymbol{x}_T,T;\boldsymbol{x}_0):t\in[0,T]\}$ if it exists.
	\end{definition}
	
	Denote the stoichiometric matrix by $\boldsymbol{\nu}^\top\triangleq[\boldsymbol{\nu}_1,\boldsymbol{\nu}_2,\cdots,\boldsymbol{\nu}_M]$. Each vector in the left null space $\boldsymbol{\nu}^{-1}(\boldsymbol{0})$ of $\boldsymbol{\nu}^\top$ sets a conservation law for the chemical reaction model, i.e., if $\boldsymbol{\eta}\in\mathbb{R}^N$ satisfies $\boldsymbol{\eta}\cdot\boldsymbol{\nu}_i=\boldsymbol{0}$ for each $1\leq i\leq M$, then
	\begin{equation*}
		\frac{\mathrm{d}\left(\boldsymbol{\eta}\cdot\boldsymbol{x}_V(t)\right)}{\mathrm{d}t}=0.
	\end{equation*}
	Therefore, the image space $\boldsymbol{\nu}^\top(\mathbb{R}^M)$ is the increment space of the reaction scheme (\ref{eq010101}) in the sense that $\boldsymbol{x}_V(t)-\boldsymbol{x}_V(0)$, $\boldsymbol{x}_\infty(t)-\boldsymbol{x}_\infty(0)$ and $\boldsymbol{y}_V(t)-\boldsymbol{y}_V(0)$ all belong to $\boldsymbol{\nu}^\top(\mathbb{R}^M)$ for each $\boldsymbol{x}_V(0),\boldsymbol{x}_\infty(0),\boldsymbol{y}_V(0)\in\mathbb{R}_+^N$ and $t>0$. The following proposition gives several properties that $S(\boldsymbol{x}_T,T\vert\boldsymbol{x}_0)$ and $\{\boldsymbol{\phi}_{\text{NOP}}(t;\boldsymbol{x}_T,T;\boldsymbol{x}_0):t\in[0,T]\}$ obey in the increment space. The proof is left in Appendix \ref{secA1}.
	\begin{proposition}\label{theo0206}
		Assume that the conditions in Theorem \ref{theo0204} hold. Then
		\begin{itemize}
			\item[(a)] if $T>0$, then $S(\boldsymbol{x}_T,T\vert\boldsymbol{x}_0)<\infty$ for each $\boldsymbol{x}_0, \boldsymbol{x}_T\in\mathbb{R}_+^N$ such that $\boldsymbol{x}_T-\boldsymbol{x}_0\in\boldsymbol{\nu}^\top(\mathbb{R}^M)$, and $S(\boldsymbol{x}_T,T\vert\boldsymbol{x}_0)=\infty$ for the remaining $\boldsymbol{x}_0, \boldsymbol{x}_T\in\mathbb{R}_+^N$;
			\item[(b)] if $T>0$, and $\boldsymbol{x}_0, \boldsymbol{x}_T\in\mathbb{R}_+^N$ satisfying $\boldsymbol{x}_T-\boldsymbol{x}_0\in\boldsymbol{\nu}^\top(\mathbb{R}^M)$, then there is at least one (possibly not unique) NOP that connects $\boldsymbol{x}_0$ with $\boldsymbol{x}_T$ in the time span $T$; for each $T>0$, $S(\boldsymbol{x}_T,T\vert\boldsymbol{x}_0)\to\infty$ as $\vert\boldsymbol{x}_T-\boldsymbol{x}_0\vert\to\infty$; for each $\boldsymbol{x}_0,\,\boldsymbol{x}_T\in\mathbb{R}^N_+$ that satisfy $\boldsymbol{x}_0\neq\boldsymbol{x}_T$, $ S(\boldsymbol{x}_T,T\vert\boldsymbol{x}_0)\to\infty$ as $T\to{0}$;
			\item[(c)] if $\{\boldsymbol{\phi}(t):t\in[0,T],\boldsymbol{\phi}(0)=\boldsymbol{x}_0,\boldsymbol{\phi}(T)=\boldsymbol{x}_T\}$ is a (or the unique) NOP that connects $\boldsymbol{x}_0$ with $\boldsymbol{x}_T$ in the time span $T$, then for any $t_1,t_2\in[0,T]$ such that $t_1<t_2$, $\{\boldsymbol{\phi}(t+t_1):t\in[0,t_2-t_1]\}$ is also a (the unique) NOP that connects $\boldsymbol{\phi}(t_1)$ with $\boldsymbol{\phi}(t_2)$ in the time span $t_2-t_1$.
		\end{itemize}
		
		If we further assume that for each $i$, ${R_{\pm i}(\boldsymbol{x})}$ are $\mathcal{C}^{k+1}$ (i.e., functions possessing continuous partial derivatives up to order $k+1$ inclusive) with some integer $k\geq 1$, and $\text{Rank}(\boldsymbol{\nu})=N$, then we have the following consequences.
		\begin{itemize}
			\item[(d)] For each NOP $\{\boldsymbol{\phi}_{\text{NOP}}(t;\boldsymbol{x}_T,T;\boldsymbol{x}_0):t\in[0,T]\}$, the functions
			\begin{equation*}
				\begin{aligned}
					&\boldsymbol{x}(t)=\boldsymbol{\phi}_{\text{NOP}}(t;\boldsymbol{x}_T,T;\boldsymbol{x}_0),\\
					&\boldsymbol{\alpha}(t)=\nabla_{\boldsymbol{\beta}}L\left(\boldsymbol{\phi}_{\text{NOP}}(t;\boldsymbol{x}_T,T;\boldsymbol{x}_0),\dot{\boldsymbol{\phi}}_{\text{NOP}}(t;\boldsymbol{x}_T,T;\boldsymbol{x}_0)\right),
				\end{aligned}
			\end{equation*}
			are also $\mathcal{C}^{k+1}$ and satisfy the following Hamilton's system of equations
			\begin{equation}\label{eq020303}
				\begin{aligned}
					&\dot{\boldsymbol{x}}(t)=\nabla_{\boldsymbol{\alpha}}H(\boldsymbol{x}(t),\boldsymbol{\alpha}(t)),\\
					&\dot{\boldsymbol{\alpha}}(t)=-\nabla_{\boldsymbol{x}}H(\boldsymbol{x}(t),\boldsymbol{\alpha}(t)),
				\end{aligned}
			\end{equation}
			with the constraints
			\begin{equation*}
				\boldsymbol{x}(0)=\boldsymbol{x}_0,\;\boldsymbol{x}(T)=\boldsymbol{x}_T.
			\end{equation*}
			Consequently, for each $t\in[0,T]$,
			\begin{equation}\label{eq020304}
				S(\boldsymbol{x}(t),t\vert\boldsymbol{x}_0)=\int_{0}^{t}\left[\dot{\boldsymbol{x}}(u)\cdot\boldsymbol{\alpha}(u)-H(\boldsymbol{x}(u),\boldsymbol{\alpha}(u))\right]\mathrm{d}u,
			\end{equation}
			and
			\begin{equation}\label{eq020305}
				S(\boldsymbol{x}_T,T-t\vert\boldsymbol{x}(t))=\int_{t}^{T}\left[\dot{\boldsymbol{x}}(u)\cdot\boldsymbol{\alpha}(u)-H(\boldsymbol{x}(u),\boldsymbol{\alpha}(u))\right]\mathrm{d}u.
			\end{equation}
			\item[(e)] Suppose that the NOP $\{\boldsymbol{\phi}_{\text{NOP}}(t;\boldsymbol{x}_T,T;\boldsymbol{x}_0):t\in[0,T]\}$ is a proper subarc of another NOP, i.e., there exist a constant $\gamma_0>0$ and a NOP $\{\bar{\boldsymbol{\phi}}(t):t\in[0,T+2\gamma_0]\}$ such that $\boldsymbol{\phi}_{\text{NOP}}(t;\boldsymbol{x}_T,T;\boldsymbol{x}_0)=\bar{\boldsymbol{\phi}}(t+\gamma_0)$ for $t\in[0,T]$. Let $B_{\delta_0}(\boldsymbol{\phi}_{\text{NOP}}(t;\boldsymbol{x}_T,T;\boldsymbol{x}_0))$ be the open $\delta_0$-neighborhood of $\boldsymbol{\phi}_{\text{NOP}}(t;\boldsymbol{x}_T,T;\boldsymbol{x}_0)$ in $\mathbb{R}_+^N$, and for any $t_1,t_2\in[0,T]$ with $t_1<t_2$, denote
			\begin{equation*}
				\Omega_{\delta_0,[t_1,t_2]}\triangleq\left\{(\boldsymbol{y},t)\in\mathbb{R}_+^N\times[t_1,t_2]:\boldsymbol{y}\in B_{\delta_0}(\boldsymbol{\phi}_{\text{NOP}}(t;\boldsymbol{x}_T,T;\boldsymbol{x}_0))\right\}.
			\end{equation*}
			Then for each $T^*\in(0,T)$ there exists a constant $\delta_0>0$ such that $S(\boldsymbol{x},t\vert\boldsymbol{x}_0)$ and $S(\boldsymbol{x}_T,T-t\vert\boldsymbol{x})$ are $\mathcal{C}^{k+1}$ with respect to the arguments $(\boldsymbol{x},t)$ at points of the sets $\Omega_{\delta_0,[T-T^*,T]}$ and $\Omega_{\delta_0,[0,T^*]}$ respectively.  In this case, they satisfy the Hamilton-Jacobi equations
			\begin{equation}\label{eq020306}
				\frac{\partial}{\partial{t}}S(\boldsymbol{x},t\vert\boldsymbol{x}_0)+H(\boldsymbol{x},\nabla_{\boldsymbol{x}}S(\boldsymbol{x},t\vert\boldsymbol{x}_0))=0,\;\;(\boldsymbol{x},t)\in\Omega_{\delta_0,[T-T^*,T]},
			\end{equation}
			and
			\begin{equation}\label{eq020307}
				\frac{\partial}{\partial{t}}S(\boldsymbol{x}_T,T-t\vert\boldsymbol{x})-H(\boldsymbol{x},-\nabla_{\boldsymbol{x}}{S}(\boldsymbol{x}_T,T-t\vert\boldsymbol{x}))=0,\;\;(\boldsymbol{x},t)\in\Omega_{\delta_0,[0,T^*]},
			\end{equation}
			respectively. The NOP is the unique solution of the following equation
			\begin{equation}\label{eq020308}
				\dot{\boldsymbol{x}}(t)=\nabla_{\boldsymbol{\alpha}}H(\boldsymbol{x}(t),\nabla_{\boldsymbol{x}}S(\boldsymbol{x}(t),t\vert\boldsymbol{x}_0)),\quad t\in[T-T^*,T],
			\end{equation}
			with the terminal condition $\boldsymbol{x}(T)=\boldsymbol{x}_T$, or
			\begin{equation}\label{eq020309}
				\dot{\boldsymbol{x}}(t)=\nabla_{\boldsymbol{\alpha}}H(\boldsymbol{x}(t),-\nabla_{\boldsymbol{x}}S(\boldsymbol{x}_T,T-t\vert\boldsymbol{x}(t))),\quad t\in[0,T^*],
			\end{equation}
			with the initial condition $\boldsymbol{x}(0)=\boldsymbol{x}_0$. Moreover, denote by $\boldsymbol{x}^*(\boldsymbol{x}_0,V)\in V^{-1}\mathbb{N}^N$ the nearest point to $\boldsymbol{x}_0$. Then, for any $(\boldsymbol{x},t)\in \Omega_{\delta_0,[T-T^*,T]}$ and sufficiently small $\varepsilon>0$,
			\begin{equation}\label{eq020310}
				\lim_{V\to\infty}V^{-1}\ln P_{\boldsymbol{x}^*(\boldsymbol{x}_0,V)}\left(\boldsymbol{x}_V(t)\in B_{\varepsilon}(\boldsymbol{x})\right)=-\inf_{\boldsymbol{y}\in B_{\varepsilon}(\boldsymbol{x})}S(\boldsymbol{y},t\vert \boldsymbol{x}_0).
			\end{equation}
		\end{itemize}
	\end{proposition}
	
	\begin{remark}
		(a) Note that the NOP is defined for each $\boldsymbol{x}_0,\boldsymbol{x}_T\in\mathbb{R}_+^N$, while chemical reaction models require $\boldsymbol{x}_0,\boldsymbol{x}_T\in V^{-1}\mathbb{N}^N$. The purpose of $\boldsymbol{x}^*(\boldsymbol{x}_0,V)$ defined here is to fill this gap.
		
		(b) $S(\boldsymbol{x},t\vert\boldsymbol{y})$ is not continuous at $t=0$ since $S(\boldsymbol{x},0\vert\boldsymbol{y})=0$ if $\boldsymbol{x}=\boldsymbol{y}$ and $S(\boldsymbol{x},0\vert\boldsymbol{y})=\infty$ otherwise. The parameter $T^*$ is selected to circumvent this non-smoothness.
		
		(c) The assumptions in part (e) of this proposition are slightly stronger than the uniqueness requirement of $\{\boldsymbol{\phi}_{\text{NOP}}(t;\boldsymbol{x}_T,T;\boldsymbol{x}_0):t\in[0,T]\}$. That is, if the assumptions in part (e) are valid, then $\{\boldsymbol{\phi}_{\text{NOP}}(t;\boldsymbol{x}_T,T;\boldsymbol{x}_0):t\in[0,T]\}$ is the unique NOP that connects $\boldsymbol{x}_0$ with $\boldsymbol{x}_T$ in the time span $T$. See Appendix \ref{secA1} for proof details.
		
		(d) Part (e) of this proposition provides a sufficient condition to represent the segments $\{\boldsymbol{\phi}_{\text{NOP}}(t;\boldsymbol{x}_T,T;\boldsymbol{x}_0):t\in[T-T^*,T]\}$ and $\{\boldsymbol{\phi}_{\text{NOP}}(t;\boldsymbol{x}_T,T;\boldsymbol{x}_0):t\in[0,T^*]\}$ as the unique solution of Eqs. (\ref{eq020308}) and (\ref{eq020309}), respectively. This result is a prerequisite for the subsequent limit theorems in Sec. \ref{sec04}.
	\end{remark}
	
	The following lemma will be used in Sec. \ref{sec04}. The proof is also left in Appendix \ref{secA1}.
	\begin{lemma}\label{theo0207}
		Assume that the conditions in Theorem \ref{theo0204} and part (e) of Proposition \ref{theo0206} hold. For any $T^*\in(0,T)$ and sufficiently small $\delta_0>0$, define $\Sigma_{\delta_0,[T-T^*,T]}\triangleq\cup_{t\in[T-T^*,T]}B_{\delta_0}(\boldsymbol{\phi}_{\text{NOP}}(t))$. Assume that for any $T^*\in(0,T)$, there is a constant $\delta_0$, so that for each sufficiently small $\varepsilon$, the pre-factor
		\begin{equation*}
			k^{\varepsilon,V}(\boldsymbol{x},t\vert\boldsymbol{x}_0)\triangleq P_{\boldsymbol{x}^*(\boldsymbol{x}_0,V)}\left(\boldsymbol{x}_V(t)\in B_{\varepsilon}(\boldsymbol{x})\right)\exp\left(V\inf_{\boldsymbol{y}\in B_{\varepsilon}(\boldsymbol{x})}S(\boldsymbol{y},t\vert \boldsymbol{x}_0)\right),
		\end{equation*}
		is continuous in $\Sigma_{\delta_0,[T-T^*,T]}\times[T-T^*,T]$, and there exists a continuous function $f^{\varepsilon,V}$ of $V$ such that 
		\begin{equation*}
			\lim_{V\to\infty,\,\varepsilon\to 0}\frac{k^{\varepsilon,V}(\boldsymbol{x},t\vert\boldsymbol{x}_0)}{f^{\varepsilon,V}}=K(\boldsymbol{x},t\vert\boldsymbol{x}_0),
		\end{equation*}
		exists, uniformly for $(\boldsymbol{x},t)\in\Sigma_{\delta_0,[T-T^*,T]}\times[T-T^*,T]$. We further assume that both $K(\boldsymbol{x},t\vert\boldsymbol{x}_0)$ and $S(\boldsymbol{x},t\vert\boldsymbol{x}_0)$ are at least twice continuous differentiable in $\Sigma_{\delta_0,[T-T^*,T]}\times[T-T^*,T]$ and $K(\boldsymbol{x},t\vert\boldsymbol{x}_0)>0$. Then for any positive function $\bar{\varepsilon}(V)$ such that $V\bar{\varepsilon}(V)>\frac{1}{2}\sqrt{N}\max_{1\leq i\leq M}\vert \boldsymbol{\nu}_i\vert$ and $\lim_{V\to\infty}V\bar{\varepsilon}^2(V)=0$, we have
		\begin{equation}\label{eq020311}
			\lim_{V\to\infty}\frac{P_{\boldsymbol{x}^*(\boldsymbol{x}_0,V)}\left(\boldsymbol{x}_V(t)\in B_{\bar{\varepsilon}(V)}(\boldsymbol{x}\pm V^{-1}\boldsymbol{\nu}_i)\right)}{P_{\boldsymbol{x}^*(\boldsymbol{x}_0,V)}\left(\boldsymbol{x}_V(t)\in B_{\bar{\varepsilon}(V)}(\boldsymbol{x})\right)}=e^{\mp \boldsymbol{\nu}_i\cdot\nabla_{\boldsymbol{x}}S(\boldsymbol{x},t\vert\boldsymbol{x}_0)},
		\end{equation} 
		uniformly for $(\boldsymbol{x},t)\in\Sigma_{\delta_0,[T-T^*,T]}\times[T-T^*,T]$ and $1\leq i\leq M$.
		
		Moreover, denote by $\boldsymbol{x}^{**}(\boldsymbol{x},V)\in\left\{\boldsymbol{y}\in\mathbb{R}_+^N:\boldsymbol{y}=\boldsymbol{x}^{*}(\boldsymbol{x}_0,V)+V^{-1}\boldsymbol{\nu}^\top\boldsymbol{k},\,\boldsymbol{k}\in\mathbb{Z}^M\right\}$ the nearest point to $\boldsymbol{x}$, then
		\begin{equation}\label{eq020312}
			\lim_{V\to\infty}\frac{p_V(\boldsymbol{x}^{**}(\boldsymbol{x},V)\pm V^{-1}\boldsymbol{\nu}_i,t\vert\boldsymbol{x}^*(\boldsymbol{x}_0,V))}{p_V(\boldsymbol{x}^{**}(\boldsymbol{x},V),t\vert\boldsymbol{x}^*(\boldsymbol{x}_0,V))}=e^{\mp \boldsymbol{\nu}_i\cdot\nabla_{\boldsymbol{x}}S(\boldsymbol{x},t\vert\boldsymbol{x}_0)},
		\end{equation}
		uniformly for $(\boldsymbol{x},t)\in\Sigma_{\delta_0,[T-T^*,T]}\times[T-T^*,T]$ and $1\leq i\leq M$.
	\end{lemma}
	\begin{remark}
		(a) The time interval is restricted to $[T-T^*,T]$ due to the singular nature of  $p_V(\boldsymbol{x},t\vert\boldsymbol{x}^*(\boldsymbol{x}_0,V))$ at $t=0$. 
		
		(b) The domain is restricted to $\Sigma_{\delta_0,[T-T^*,T]}$, as this constitutes the fundamental requirement of the limit theorems presented in Sec. \ref{sec04}.
		
		(c) Note that $\boldsymbol{x}_V(t)-\boldsymbol{x}_V(0)=V^{-1}\boldsymbol{\nu}^\top\boldsymbol{k}$ for some $\boldsymbol{k}\in\mathbb{Z}^M$. The choice of $\boldsymbol{x}^{**}(\boldsymbol{x},V)$ is made with the intention of ensuring that both the numerator and the denominator in Eq. (\ref{eq020312}) are positive. 
		
		(d) The condition $V\bar{\varepsilon}(V)>\frac{1}{2}\sqrt{N}\max_{1\leq i\leq M}\vert \boldsymbol{\nu}_i\vert$ is sufficient for us to ensure that there exists at least one point $\boldsymbol{x}^{**}(\boldsymbol{x},V)$ in the neighborhood $B_{\bar{\varepsilon}(V)}(\boldsymbol{x})$ of each $\boldsymbol{x}$.
	\end{remark}
	
	\subsection{Optimal Fluctuations on Infinite Time Intervals}
	The large deviation principle also provides a tool for us to estimate the probabilities of improbable events involving the behavior of the stochastic chemical reaction model over very long (possibly infinite) time intervals \cite{Freidlin_2012}, such as the behavior of escapes from a domain for sufficiently large $V$ \cite{Knessl_1984,Shwartz_1995} and the limit behavior of the stationary distribution as $V\to\infty$ \cite{Hu_Gang_1987,Shwartz_1995,Anderson_2015,Hao_Ge_2016b}. The following theorem concerning the limit distribution in the case that the corresponding deterministic system possesses a unique global attractive equilibrium was proved in [\onlinecite[Theorem 6.89]{Shwartz_1995}].
	\begin{theorem}\label{theo0208}
		Assume that the conditions in Theorem \ref{theo0204} hold, $\text{Rank}(\boldsymbol{\nu})=N$, and
		\begin{itemize}
			\item[(a)] $\boldsymbol{x}_{\text{eq}}\in(0,+\infty)^N$ is the unique global attractive equilibrium of Eq. (\ref{eq020102});
			\item[(b)] for each $V$, the stochastic chemical reaction model is positive recurrent such that there exists a unique stationary distribution $\pi_V$ for the process $\boldsymbol{x}_V(t)$.
		\end{itemize} 
		Then for each sufficiently small $\varepsilon$,
		\begin{equation}\label{eq020401}
			\lim_{V\to\infty}\pi_V(B_{\varepsilon}(\boldsymbol{x}_{eq}))=1,
		\end{equation}
		where $B_{\varepsilon}(\boldsymbol{x}_{eq})$ is the open $\varepsilon$-neighborhood of $\boldsymbol{x}_{eq}$ in $\mathbb{R}_+^N$.
		
		Define
		\begin{equation}\label{eq020402}
			S(\boldsymbol{x})\triangleq\inf_{T>0}\inf_{\boldsymbol{\phi}(0)=\boldsymbol{x}_{eq},\boldsymbol{\phi}(T)=\boldsymbol{x}}I_{[0,T]}(\{\boldsymbol{\phi}(t):t\in[0,T]\}).
		\end{equation}
		For a bounded open set $D$ with smooth boundary, define 
		\begin{equation*}
			\mathcal{E}\triangleq\left\{\{\boldsymbol{\phi}(t):t\in [0,T]\}:\boldsymbol{\phi}(0)=\boldsymbol{x}_{eq},\boldsymbol{\phi}(T)\in\bar{D}\;\text{for some}\;T>0\right\},
		\end{equation*}
		and
		\begin{equation*}
			S(\bar{D})\triangleq\inf_{\boldsymbol{x}\in\bar{D}}S(\boldsymbol{x}).
		\end{equation*}
		We further assume that there is a neighborhood of $\boldsymbol{x}_{eq}$ (for example, $B_{\varepsilon_0}(\boldsymbol{x}_{eq})$) such that
		\begin{itemize}
			\item[(c)] $D\subset B_{\varepsilon_0}(\boldsymbol{x}_{eq})$ and $S(\boldsymbol{x})>S(\bar{D})+1$ whenever $\boldsymbol{x}$ is outside $B_{\varepsilon_0}(\boldsymbol{x}_{eq})$;
			\item[(d)] $\mathcal{E}$ is a continuity set, and every point in $\mathcal{E}$ is the limit of points in the interior of $\mathcal{E}$. That is, $\mathcal{E}\subset\overline{\mathcal{E}^o}$ and $S(D^o)=S(\bar{D})$;
			\item[(e)] for each $\eta>0$, there is a constant $T<\infty$ such that, uniformly over $\boldsymbol{x}_0\in B_{\varepsilon_0}(\boldsymbol{x}_{eq})$ with $\boldsymbol{x}_\infty(0)=\boldsymbol{x}_0$, $\vert\boldsymbol{x}_\infty(t)-\boldsymbol{x}_{eq}\vert<\eta$ for all $t>T$.
		\end{itemize}
		Then 
		\begin{equation}\label{eq020403}
			\lim_{V\to\infty}V^{-1}\ln\pi_V(D)=-\inf_{\boldsymbol{x}\in D}S(\boldsymbol{x}).
		\end{equation}
	\end{theorem}
	
	Now, the (stationary) optimal path can be defined similarly.
	\begin{definition}\label{def0202}
		For each $\boldsymbol{x}\in\mathbb{R}_+^N$, a path is said to be an optimal path (OP) (or a stationary optimal path) that connects $\boldsymbol{x}_{eq}$ with $\boldsymbol{x}$ if it is a minimizer of Eq. (\ref{eq020402}), i.e., a path that begins at $\boldsymbol{x}_{eq}$ and ends at $\boldsymbol{x}$ with a minimal value of the rate function $I$. We denote it by $\left\{\boldsymbol{\phi}_{\text{OP}}(t;\boldsymbol{x})\right\}$ if it exists.
	\end{definition}
	
	$S(\boldsymbol{x})$ and $\left\{\boldsymbol{\phi}_{\text{OP}}(t;\boldsymbol{x})\right\}$ possess properties that are partly similar to those of $S(\boldsymbol{x}_T,T\vert\boldsymbol{x}_0)$ and $\{\boldsymbol{\phi}_{\text{NOP}}(t;\boldsymbol{x}_T,T;\boldsymbol{x}_0)\}$\cite{Day_1985,Day_1986}. We state them as follows. The proof is left in Appendix \ref{secA2}.
	\begin{proposition}\label{theo0209}
		Assume that the conditions in Theorems \ref{theo0204} and \ref{theo0208} hold. Then
		\begin{itemize}
			\item[(a)] for each $\boldsymbol{x}\in\mathbb{R}_+^N$, $S(\boldsymbol{x})<\infty$;
			\item[(b)] $S(\boldsymbol{x})$ is a global Lipschitz continuous function;
			\item[(c)] if $\left\{\boldsymbol{\phi}(t):t\in[T_1,T_2],\boldsymbol{\phi}(T_1)=\boldsymbol{x}_{eq},\boldsymbol{\phi}(T_2)=\boldsymbol{x}\right\}$ ($-\infty\leq T_1<T_2\leq\infty$) is an (or the unique) OP that connects $\boldsymbol{x}_{eq}$ with $\boldsymbol{x}$, for any $T^*\in(T_1,T_2)$, $\left\{\boldsymbol{\phi}(t):t\in[T_1,T^*]\right\}$ is also an (the unique) OP that connects $\boldsymbol{x}_{eq}$ with $\boldsymbol{\phi}(T^*)$. In addition, for any $t_1,t_2\in(T_1,T_2)$ such that $t_1<t_2$, $\left\{\boldsymbol{\phi}(t+t_1):t\in[0,t_2-t_1]\right\}$ is a (the unique) NOP that connects $\boldsymbol{\phi}(t_1)$ with $\boldsymbol{\phi}(t_2)$ in all possible time span, i.e., 
			\begin{equation*}
				I_{[0,t_2-t_1]}(\{\boldsymbol{\phi}(t+t_1):t\in[0,t_2-t_1]\})=\inf_{T>0}S(\boldsymbol{\phi}(t_2),T\vert\boldsymbol{\phi}(t_1)).
			\end{equation*}
		\end{itemize}
		
		If we further assume that for each $i$, ${R_{\pm i}(\boldsymbol{x})}$ are $\mathcal{C}^{k+1}$ for some $k\geq 1$, then we have the following consequences.
		\begin{itemize}
			\item[(d)] For each $\boldsymbol{x}\in\mathbb{R}_+^N$, there is at least one (possible not unique) OP, which can be represented by $\left\{\boldsymbol{\phi}_{\text{OP}}(t;\boldsymbol{x}):t\in(-\infty,0],\lim_{t\to-\infty}\boldsymbol{\phi}_{\text{OP}}(t;\boldsymbol{x})=\boldsymbol{x}_{eq},\boldsymbol{\phi}_{\text{OP}}(0;\boldsymbol{x})=\boldsymbol{x}\right\}$.
			\item[(e)] For each OP $\left\{\boldsymbol{\phi}_{\text{OP}}(t;\boldsymbol{x}):t\in(-\infty,0]\right\}$, the functions
			\begin{equation*}
				\begin{aligned}
					&\boldsymbol{x}(t)=\boldsymbol{\phi}_{\text{OP}}(t;\boldsymbol{x}),\\
					&\boldsymbol{\alpha}(t)=\nabla_{\boldsymbol{\beta}}L\left(\boldsymbol{\phi}_{\text{OP}}(t;\boldsymbol{x}),\dot{\boldsymbol{\phi}}_{\text{OP}}(t;\boldsymbol{x})\right),
				\end{aligned}
			\end{equation*}
			are also $\mathcal{C}^{k+1}$ and satisfy Eq. (\ref{eq020303}) as well as
			\begin{equation}\label{eq020404}
				H(\boldsymbol{x}(t),\boldsymbol{\alpha}(t))\equiv 0.
			\end{equation}
			Similarly, for each $t\in(-\infty,0]$,
			\begin{equation}\label{eq020405}
				S(\boldsymbol{x}(t))=\int_{-\infty}^{t}\dot{\boldsymbol{x}}(s)\cdot\boldsymbol{\alpha}(s)\mathrm{d}s.
			\end{equation}
			\item[(f)] Suppose that $\left\{\boldsymbol{\phi}_{\text{OP}}(t;\boldsymbol{x}):t\in(-\infty,0]\right\}$ is a proper subarc of another OP, i.e., there exist a constant $\gamma_0>0$ and an OP $\{\bar{\boldsymbol{\phi}}(t):t\in(-\infty,0]\}$ such that $\boldsymbol{\phi}_{\text{OP}}(t;\boldsymbol{x})=\bar{\boldsymbol{\phi}}(t-\gamma_0)$ for $t\in(-\infty,0]$. Then $S(\boldsymbol{x})$ is $\mathcal{C}^{k+1}$ in the vicinity of each point $\boldsymbol{x}(t)$ for $t\in(-\infty,0]$. As a result, in some neighborhood of the OP, $S(\boldsymbol{x})$ satisfies the stationary Hamilton-Jacobi equation
			\begin{equation}\label{eq020406}
				H(\boldsymbol{x},\nabla_{\boldsymbol{x}}S(\boldsymbol{x}))=0,
			\end{equation}
			and the OP is the unique solution of the following equation
			\begin{equation}\label{eq020407}
				\dot{\boldsymbol{x}}(t)=\nabla_{\boldsymbol{\alpha}}H(\boldsymbol{x}(t),\nabla_{\boldsymbol{x}}S(\boldsymbol{x}(t))), \quad t\in(-\infty,0],
			\end{equation}
			with the constraints $\lim_{t\to-\infty}\boldsymbol{x}(t)=\boldsymbol{x}_{eq}$ and $\boldsymbol{x}(0)=\boldsymbol{x}$.
		\end{itemize}
	\end{proposition}
	\begin{remark}
		(c) The assumptions in part (f) of this proposition are slightly stronger than the uniqueness requirement of $\{\boldsymbol{\phi}_{\text{OP}}(t;\boldsymbol{x}):t\in(-\infty,0]\}$. That is, if the assumptions in part (e) hold, then $\{\boldsymbol{\phi}_{\text{OP}}(t;\boldsymbol{x}):t\in(-\infty,0]\}$ is the unique OP that connects $\boldsymbol{x}_{eq}$ with $\boldsymbol{x}$.
		
		(d) Part (f) of this proposition provides a sufficient condition to represent the OP as the unique solution of Eq. (\ref{eq020407}). This result constitutes a prerequisite for the ensuing limit theorems in Sec. \ref{sec05}.
	\end{remark}
	The following lemma is from [\onlinecite[Lemma 5]{Hao_Ge_2017}], and will be required in Sec. \ref{sec05}. 
	\begin{lemma}\label{theo0210}
		Assume that the conditions in Theorems \ref{theo0204}, \ref{theo0208} and part (f) of Proposition \ref{theo0209} hold. For any $T^*\in(-\infty,0)$ and sufficiently small $\delta_0>0$, define $\Sigma_{\delta_0,[T^*,0]}=\cup_{t\in[T^*,0]}B_{\delta_0}(\boldsymbol{\phi}_{\text{OP}}(t))$. Assume that for any $T^*\in(-\infty,0)$, there is a constant $\delta_0$, so that for each sufficiently small $\varepsilon$, the pre-factor
		\begin{equation*}
			k^{\varepsilon,V}(\boldsymbol{x})\triangleq\pi_V(B_{\varepsilon}(\boldsymbol{x}))\exp\left(V\inf_{\boldsymbol{y}\in B_{\varepsilon}(\boldsymbol{x})}S(\boldsymbol{y})\right),
		\end{equation*}
		is continuous in $\Sigma_{\delta_0,[T^*,0]}$, and there exists a positive function $f^{\varepsilon,V}$ of $V$ such that
		\begin{equation*}
			\lim_{V\to\infty,\;\varepsilon\to 0}\frac{k^{\varepsilon,V}(\boldsymbol{x})}{f^{\varepsilon,V}}=K(\boldsymbol{x}),
		\end{equation*}
		exists, uniformly for $\boldsymbol{x}\in\Sigma_{\delta_0,[T^*,0]}$. We also assume that both $K(\boldsymbol{x})$ and $S(\boldsymbol{x})$ are at least twice continuous differentiable in $\Sigma_{\delta_0,[T^*,0]}$, and $K(\boldsymbol{x})>0$. Then for any positive function $\bar{\varepsilon}(V)$ such that $V\bar{\varepsilon}(V)>\frac{1}{2}\sqrt{N}\max_{1\leq i\leq M}\vert \boldsymbol{\nu}_i\vert$, and $\lim_{V\to\infty}V\bar{\varepsilon}^2(V)=0$, we have
		\begin{equation}\label{eq020408}
			\lim_{V\to\infty}\frac{\pi_V(B_{\bar{\varepsilon}(V)}(\boldsymbol{x}\pm V^{-1}\boldsymbol{\nu}_i))}{\pi_V(B_{\bar{\varepsilon}(V)}(\boldsymbol{x}))}=e^{\mp\boldsymbol{\nu}_i\cdot\nabla_{\boldsymbol{x}}S(\boldsymbol{x})},
		\end{equation}
		uniformly for $\boldsymbol{x}\in\Sigma_{\delta_0,[T^*,0]}$ and $1\leq i\leq M$.
		
		Furthermore,
		\begin{equation}\label{eq020409}
			\lim_{V\to\infty}\frac{\pi_V(\boldsymbol{x}^{**}(\boldsymbol{x},V)\pm V^{-1}\boldsymbol{\nu}_i)}{\pi_V(\boldsymbol{x}^{**}(\boldsymbol{x},V))}=e^{\mp\boldsymbol{\nu}_i\cdot\nabla_{\boldsymbol{x}}S(\boldsymbol{x})},
		\end{equation}
		uniformly for $\boldsymbol{x}\in\Sigma_{\delta_0,[T^*,0]}$ and $1\leq i\leq M$.
	\end{lemma}
	\begin{remark}
		(a)	The domain is restricted to $\Sigma_{\delta_0,[T^*,0]}$, since this is fundamental for the validity of the limit theorems presented in Sec. \ref{sec05}.
		
		(b) Despite the fact that the notations employed in this section are identical to that utilised in the preceding one, it is possible to distinguish them according to whether the case is stationary or non-stationary.
	\end{remark}
	
	\section{Time Reversal of Nonlinear Chemical Reaction Models}\label{sec03}
	In this section, we fix the time interval as $[0,T]$. The infinitesimal generator $\mathscr{J}_{\boldsymbol{x}}$ is defined for each bounded continuous function $f(\boldsymbol{x})$ by the formula
	\begin{equation*}
		\begin{aligned}
			\mathscr{J}_{\boldsymbol{x}}f(\boldsymbol{x})\triangleq&\sum_{i=1}^{M}\left(f(\boldsymbol{x}+V^{-1}\boldsymbol{\nu}_i)-f(\boldsymbol{x})\right)r_{+i}(V\boldsymbol{x},V)\\
			&+\sum_{i=1}^{M}\left(f(\boldsymbol{x}-V^{-1}\boldsymbol{\nu}_i)-f(\boldsymbol{x})\right)r_{-i}(V\boldsymbol{x},V),
		\end{aligned}
	\end{equation*}
	with its adjoint operator being of the form
	\begin{equation*}
		\begin{aligned}
			\mathscr{J}^{*}_{\boldsymbol{x}}f(\boldsymbol{x})=&\sum_{i=1}^{M}\left(f(\boldsymbol{x}-V^{-1}\boldsymbol{\nu}_i)r_{+i}(V\boldsymbol{x}-\boldsymbol{\nu}_i,V)-f(\boldsymbol{x})r_{+i}(V\boldsymbol{x},V)\right)\\
			&+\sum_{i=1}^{M}\left(f(\boldsymbol{x}+V^{-1}\boldsymbol{\nu}_i)r_{-i}(V\boldsymbol{x}+\boldsymbol{\nu}_i,V)-f(\boldsymbol{x})r_{-i}(V\boldsymbol{x},V)\right).
		\end{aligned}
	\end{equation*}
	The time evolution for probability of $\boldsymbol{x}_V(t)$ can be described by the chemical master equation (or the Kolmogorov forward equation)
	\begin{equation}\label{eq030101}
		\begin{aligned}
			\frac{\mathrm{d}p_V(\boldsymbol{x},t)}{\mathrm{d}t}=\mathscr{J}^{*}_{\boldsymbol{x}}p_V(\boldsymbol{x},t),
		\end{aligned}
	\end{equation}
	in which $x\in V^{-1}\mathbb{N}^d$ and $t\in [0,T]$. In particular, the conditional probability $p_V(\boldsymbol{x},t\vert \boldsymbol{x}_0)$ satisfies the forward equation
	\begin{equation*}
		\frac{\mathrm{d}p_V(\boldsymbol{x},t\vert \boldsymbol{x}_0)}{\mathrm{d}t}=\mathscr{J}^{*}_{\boldsymbol{x}}p_V(\boldsymbol{x},t\vert \boldsymbol{x}_0),
	\end{equation*}
	with the initial condition
	\begin{equation*}
		p_V(\boldsymbol{x},0\vert \boldsymbol{x}_0)=\begin{cases}
			1,\quad \boldsymbol{x}=\boldsymbol{x}_0\in V^{-1}\mathbb{N}^N, \\
			0,\quad \text{otherwise},
		\end{cases}
	\end{equation*}
	and also the backward equation
	\begin{equation*}
		\begin{aligned}
			\frac{\mathrm{d}p_V(\boldsymbol{x},t\vert \boldsymbol{x}_0)}{\mathrm{d}t}=\mathscr{J}_{\boldsymbol{x}_0}p_V(\boldsymbol{x},t\vert \boldsymbol{x}_0),
		\end{aligned}
	\end{equation*}
	with the terminal condition 
	\begin{equation*}
		p_V(\boldsymbol{x},0\vert \boldsymbol{x}_0)=\begin{cases}
			1,\quad \boldsymbol{x}_0=\boldsymbol{x}\in V^{-1}\mathbb{N}^N, \\
			0,\quad \text{otherwise}.
		\end{cases}
	\end{equation*}
	\subsection{Time Reversal of a Given Family of Probabilities}\label{sec0301}
	For a given family of probabilities $\{p_V(\boldsymbol{x},t)\}_{t\in [0,T]}$ satisfying Eq. (\ref{eq030101}), we call $\{\bar{p}_V(\boldsymbol{x},t)=p_{V}(\boldsymbol{x},T-t)\}_{t\in [0,T]}$ the time reversal of $\{p_V(\boldsymbol{x},t)\}_{t\in [0,T]}$. The following proposition is easy to verify.
	\begin{proposition}
		Let $\mathscr{K}^{*}_{\boldsymbol{x},t}$ be an operator depending on the given family $\{p_V(\boldsymbol{x},t)\}_{t\in [0,T]}$, which is defined for each bounded continuous function $f(\boldsymbol{x},t)$ by
		\begin{equation*}
			\begin{aligned}
				\mathscr{K}^{*}_{\boldsymbol{x},t}f(\boldsymbol{x},t)\triangleq&\sum_{i=1}^{M}{\left(f(\boldsymbol{x}-V^{-1}\boldsymbol{\nu}_i,t)\bar{r}_{+i}(V\boldsymbol{x}-\boldsymbol{\nu}_i,V,t)-f(\boldsymbol{x},t)\bar{r}_{+i}(V\boldsymbol{x},V,t)\right)}\\
				&+\sum_{i=1}^{M}{\left(f(\boldsymbol{x}+V^{-1}\boldsymbol{\nu}_i,t)\bar{r}_{-i}(V\boldsymbol{x}+\boldsymbol{\nu}_i,V,t)-f(\boldsymbol{x},t)\bar{r}_{-i}(V\boldsymbol{x},V,t)\right)},
			\end{aligned}
		\end{equation*}
		where
		\begin{equation}\label{eq030201}
			\bar{r}_{\pm i}(V\boldsymbol{x},V,t)\triangleq\frac{p_V(\boldsymbol{x}\pm V^{-1}\boldsymbol{\nu}_i,T-t)r_{\mp i}(V\boldsymbol{x}\pm\boldsymbol{\nu}_i,V)}{p_V(\boldsymbol{x},T-t)}.
		\end{equation}
		Then, $\{\bar{p}_V(x,t)\}_{t\in [0,T]}$ satisfies the following master equation
		\begin{equation}\label{eq030202}
			\frac{\mathrm{d}\bar{p}_V(\boldsymbol{x},t)}{\mathrm{d}t}=\mathscr{K}^{*}_{\boldsymbol{x},t}\bar{p}_V(\boldsymbol{x},t).
		\end{equation}
	\end{proposition}
	\begin{remark}
		The convention $\frac{0}{0}=0$ is used here and in subsequent contents to ensure that the rate functions $\bar{r}_{\pm i}$ of the reversed process are well-defined. 
	\end{remark}
	It is well known that the time reversal of a Markov process remains a Markov process \cite{Haussmann_1986,Cattiaux_2023}. In the context of a continuous-time setting, a substantial class of Markov processes with jumps is also preserved under the time reversal \cite{Conforti_2022}. Now, let $\bar{\boldsymbol{x}}_{V}(t)$ be a non-homogeneous Markov jump process with its transition rates dependent on the given family $\{p_V(x,t)\}_{t\in [0,T]}$, which is defined by
	\begin{equation}\label{eq030203}
		\begin{aligned}
			\bar{\boldsymbol{x}}_{V}(t)=&\bar{\boldsymbol{x}}_{V}(0)\\
			&+V^{-1}\sum_{i=1}^{M}{\boldsymbol{\nu}_i\left\{Y_{+i}\left(\int_{0}^{t}\bar{r}_{+i}(V\bar{\boldsymbol{x}}_{V}(s),V,s)\mathrm{d}s\right)-Y_{-i}\left(\int_{0}^{t}\bar{r}_{-i}(V\bar{\boldsymbol{x}}_{V}(s),V,s)\mathrm{d}s\right)\right\}}.
		\end{aligned}
	\end{equation}	
	It follows from this proposition that
	\begin{corollary}
		For each $t\in[0,T]$, $\bar{p}_V(x,t)$ is the probability of the reversed process $\bar{\boldsymbol{x}}_{V}(t)$ if the initial distribution of $\bar{\boldsymbol{x}}_{V}(0)$ is given by $\bar{p}_V(x,0)$. In other words, the reversed process $\bar{\boldsymbol{x}}_{V}(t)$ defined here is employed to characterize the time evolution of the reversed family $\{\bar{p}_V(x,t)\}_{t\in[0,T]}$.
	\end{corollary}
	\subsection{Conditional Probability of the Reversed Process}\label{sec0302}
	For a fixed $\boldsymbol{x}_T$, define
	\begin{equation}\label{eq030301}
		q_{V}(\boldsymbol{x},t;\boldsymbol{x}_T,T)\triangleq\frac{p_V(\boldsymbol{x},t)p_V(\boldsymbol{x}_T,T-t\vert\boldsymbol{x})}{p_V(\boldsymbol{x}_T,T)},\quad t\in[0,T],
	\end{equation}
	and its time reversal
	\begin{equation}\label{eq030302}
		\bar{q}_{V}(\boldsymbol{x},t;\boldsymbol{x}_T,T)\triangleq q_{V}(\boldsymbol{x},T-t;\boldsymbol{x}_T,T)=\frac{p_V(\boldsymbol{x},T-t)p_V(\boldsymbol{x}_T,t\vert\boldsymbol{x})}{p_V(\boldsymbol{x}_T,T)},\quad t\in[0,T].
	\end{equation}
	Taking the derivative of these quantities with respect to $t$, substituting the forward and backward equations into them, and regrouping the terms, yield the following proposition.
	\begin{proposition}
		\begin{itemize}
			\item[(a)] $\bar{q}_{V}(\boldsymbol{x},t;\boldsymbol{x}_T,T)$ satisfies the following master equation
			\begin{equation}\label{eq030303}
				\frac{\mathrm{d}\bar{q}_{V}(\boldsymbol{x},t;\boldsymbol{x}_T,T)}{\mathrm{d}t}=\mathscr{K}^{*}_{\boldsymbol{x},t}\bar{q}_{V}(\boldsymbol{x},t;\boldsymbol{x}_T,T).
			\end{equation}
			\item[(b)] Let $\mathscr{L}^{*}_{\boldsymbol{x},t}$ be an operator depending on the given family $\{p_V(\boldsymbol{x}_T,T-t\vert\boldsymbol{x})\}_{t\in [0,T]}$, which is defined for each bounded continuous function $f(\boldsymbol{x},t)$ by
			\begin{equation*}
				\begin{aligned}
					\mathscr{L}^{*}_{\boldsymbol{x},t}f(\boldsymbol{x},t)\triangleq&\sum_{i=1}^{M}{\left(f(\boldsymbol{x}-V^{-1}\boldsymbol{\nu}_i,t)\tilde{r}_{+i}(V\boldsymbol{x}-\boldsymbol{\nu}_i,V,t)-f(\boldsymbol{x},t)\tilde{r}_{+i}(V\boldsymbol{x},V,t)\right)}\\
					&+\sum_{i=1}^{M}{\left(f(\boldsymbol{x}+V^{-1}\boldsymbol{\nu}_i,t)\tilde{r}_{-i}(V\boldsymbol{x}+\boldsymbol{\nu}_i,V,t)-f(\boldsymbol{x},t)\tilde{r}_{-i}(V\boldsymbol{x},V,t)\right)},
				\end{aligned}
			\end{equation*}
			in which 
			\begin{equation}\label{eq030304}
				\tilde{r}_{\pm i}(V\boldsymbol{x},V,t)\triangleq\frac{p_V(\boldsymbol{x}_T,T-t\vert\boldsymbol{x}\pm V^{-1}\boldsymbol{\nu}_i)r_{\pm i}(V\boldsymbol{x},V)}{p_V(\boldsymbol{x}_T,T-t\vert\boldsymbol{x})}.
			\end{equation}
			Then $q_{V}(\boldsymbol{x},t;\boldsymbol{x}_T,T)$ satisfies the following master equation
			\begin{equation}\label{eq030305}
				\frac{\mathrm{d}q_{V}(\boldsymbol{x},t;\boldsymbol{x}_T,T)}{\mathrm{d}t}=\mathscr{L}^{*}_{\boldsymbol{x},t}q_{V}(\boldsymbol{x},t;\boldsymbol{x}_T,T).
			\end{equation}
		\end{itemize}
	\end{proposition}
	
	We can interpret $q_{V}(\boldsymbol{x},t;\boldsymbol{x}_T,T)$ and $\bar{q}_{V}(\boldsymbol{x},t;\boldsymbol{x}_T,T)$ as follows.
	\begin{corollary}
		\begin{itemize}
			\item[(a)] For each $t\in[0,T]$, $\bar{q}_{V}(\boldsymbol{x},t;\boldsymbol{x}_T,T)$ is the probability of $\bar{\boldsymbol{x}}_{V}(t)$ if the initial distribution of $\bar{\boldsymbol{x}}_{V}(0)$ is given by
			\begin{equation*}
				\bar{q}_{V}(\boldsymbol{x},0;\boldsymbol{x}_T,T)=\begin{cases}
					1,\quad \boldsymbol{x}=\boldsymbol{x}_T\in V^{-1}\mathbb{N}^N, \\
					0,\quad \text{otherwise},
				\end{cases}
			\end{equation*}
			i.e., $\bar{\boldsymbol{x}}_{V}(0)=\boldsymbol{x}_T$ a.s..
			\item[(b)] $\bar{q}_{V}(\boldsymbol{x},t;\boldsymbol{x}_T,T)$ is the conditional probability of the reversed family $\{\bar{p}_V(\boldsymbol{x},t)\}_{t\in [0,T]}$ in the sense that
			\begin{equation*}
				\sum_{\boldsymbol{x}_T\in V^{-1}\mathbb{N}^N}\bar{p}_V(\boldsymbol{x}_T,0)\bar{q}_{V}(\boldsymbol{x},t;\boldsymbol{x}_T,T)=\bar{p}_V(\boldsymbol{x},t),
			\end{equation*}
			i.e., the Chapman–Kolmogorov-type equation holds.
		\end{itemize}
	\end{corollary}
	
	Let $\tilde{\boldsymbol{x}}_{V}(t)$ be another non-homogeneous Markov jump process, which is defined by
	\begin{equation}\label{eq030306}
		\begin{aligned}
			\tilde{\boldsymbol{x}}_{V}(t)=&\tilde{\boldsymbol{x}}_{V}(0)\\
			+&V^{-1}\sum_{i=1}^{M}{\boldsymbol{\nu}_i\left\{Y_{+i}\left(\int_{0}^{t}\tilde{r}_{+i}(V\tilde{\boldsymbol{x}}_{V}(s),V,s)\mathrm{d}s\right)-Y_{-i}\left(\int_{0}^{t}\tilde{r}_{-i}(V\tilde{\boldsymbol{x}}_{V}(s),V,s)\mathrm{d}s\right)\right\}}.
		\end{aligned}
	\end{equation}
	Then we can conclude that
	\begin{corollary}
		For each $t\in[0,T]$, $q_{V}(\boldsymbol{x},t;\boldsymbol{x}_T,T)$ is the probability of $\tilde{\boldsymbol{x}}_{V}(t)$ if the initial distribution is given by $q_{V}(\boldsymbol{x},0;\boldsymbol{x}_T,T)$. That is to say, the reversed process $\tilde{\boldsymbol{x}}_{V}(t)$, defined here, is used to characterize the time evolution of the family $\{q_{V}(\boldsymbol{x},t;\boldsymbol{x}_T,T)\}_{t\in[0,T]}$.
	\end{corollary}
	\begin{remark}
		An alternative description of the process $\tilde{\boldsymbol{x}}_{V}(t)$ can be provided as follows. In fact, based on the result in Sec. \ref{sec0301}, it was already known that $\{\bar{q}_{V}(\boldsymbol{x},t;\boldsymbol{x}_T,T)\}_{t\in [0,T]}$ is a family of probabilities of $\bar{\boldsymbol{x}}_{V}(t)$. Repeating the process above, it is evident that the reversed family $\{\bar{\bar{q}}_{V}(\boldsymbol{x},t;\boldsymbol{x}_T,T)=q_{V}(\boldsymbol{x},t;\boldsymbol{x}_T,T)\}_{t\in [0,T]}$ satisfies another master equation with its rate functions characterised by the following expressions:
		\begin{equation*}
			\begin{aligned}
				\bar{\bar{r}}_{\pm i}(V\boldsymbol{x},V,t)=&\frac{\bar{q}_{V}(\boldsymbol{x}\pm V^{-1}\boldsymbol{\nu}_i,T-t;\boldsymbol{x}_T,T)\bar{r}_{\mp i}(V\boldsymbol{x}\pm \boldsymbol{\nu}_i,V,T-t)}{\bar{q}_{V}(\boldsymbol{x},T-t;\boldsymbol{x}_T,T)}\\
				=&\frac{p_V(\boldsymbol{x}_T,T-t\vert\boldsymbol{x}\pm V^{-1}\boldsymbol{\nu}_i)r_{\pm i}(V\boldsymbol{x},V)}{p_V(\boldsymbol{x}_T,T-t\vert\boldsymbol{x})},
			\end{aligned}
		\end{equation*}
		which is exactly the rate functions defined in Eq. (\ref{eq030304}). Consequently, we can regard $\tilde{\boldsymbol{x}}_{V}(t)$ as a time reversal of the process $\bar{\boldsymbol{x}}_{V}(t)$, and, to a certain extent, a double time reversal of the original process $\boldsymbol{x}_{V}(t)$.
	\end{remark}
	\subsection{Stationary Prehistory Probability}\label{sec0303}
	In this section, we give a rigorous definition of the stationary prehistory probability and relate it to the time reversal of the stationary distribution $\pi_V$.
	\begin{definition}
		The stationary prehistory probability (SPP) $q^{\text{SPP}}_{V}(\boldsymbol{x},t;\boldsymbol{x}_T,T)$ is defined by
		\begin{equation}\label{eq030401}
			q^{\text{SPP}}_{V}(\boldsymbol{x},t;\boldsymbol{x}_T,T)\triangleq\frac{\pi_V(\boldsymbol{x})p_V(\boldsymbol{x}_T,T-t\vert\boldsymbol{x})}{\pi_V(\boldsymbol{x}_T)},\quad t\in[0,T],
		\end{equation}
		with its time reversal given by
		\begin{equation}\label{eq030402}
			\bar{q}^{\text{SPP}}_{V}(\boldsymbol{x},t;\boldsymbol{x}_T,T)\triangleq q^{\text{SPP}}_{V}(\boldsymbol{x},T-t;\boldsymbol{x}_T,T)=\frac{\pi_V(\boldsymbol{x})p_V(\boldsymbol{x}_T,t\vert\boldsymbol{x})}{\pi_V(\boldsymbol{x}_T)},\quad t\in[0,T].
		\end{equation}
	\end{definition}
	
	The ensuing results are analogous to those presented in Sec. \ref{sec0302}.
	\begin{corollary}
		\begin{itemize}
			\item[(a)] $q^{\text{SPP}}_{V}(\boldsymbol{x},t;\boldsymbol{x}_T,T)$ satisfies the master equation
			\begin{equation}\label{eq030403}
				\frac{\mathrm{d}q^{\text{SPP}}_{V}(\boldsymbol{x},t;\boldsymbol{x}_T,T)}{\mathrm{d}t}=\mathscr{L}^{*}_{\boldsymbol{x},t}q^{\text{SPP}}_{V}(\boldsymbol{x},t;\boldsymbol{x}_T,T).
			\end{equation}
			As a result, $q^{\text{SPP}}_{V}(\boldsymbol{x},t;\boldsymbol{x}_T,T)$ is the probability of $\tilde{\boldsymbol{x}}_V(t)$ if the initial distribution obeys $P(\tilde{\boldsymbol{x}}_V(0)=\boldsymbol{x})=q^{\text{SPP}}_{V}(\boldsymbol{x},0;\boldsymbol{x}_T,T)$.
			\item[(b)] $\bar{q}^{\text{SPP}}_{V}(\boldsymbol{x},t;\boldsymbol{x}_T,T)$ satisfies the following master equation
			\begin{equation}\label{eq030404}
				\frac{\mathrm{d}\bar{q}^{\text{SPP}}_{V}(\boldsymbol{x},t;\boldsymbol{x}_T,T)}{\mathrm{d}t}=\mathscr{M}^{*}_{\boldsymbol{x}}\bar{q}^{\text{SPP}}_{V}(\boldsymbol{x},t;\boldsymbol{x}_T,T),
			\end{equation}
			in which $\mathscr{M}^{*}_{\boldsymbol{x}}$ is defined for each bounded continuous function $f(\boldsymbol{x})$ by
			\begin{equation*}
				\begin{aligned}
					\mathscr{M}^{*}_{\boldsymbol{x}}f(\boldsymbol{x})\triangleq&\sum_{i=1}^{M}{\left(f(\boldsymbol{x}-V^{-1}\boldsymbol{\nu}_i)\bar{r}^{\text{SPP}}_{+i}(V\boldsymbol{x}-\boldsymbol{\nu}_i,V)-f(\boldsymbol{x})\bar{r}^{\text{SPP}}_{+i}(V\boldsymbol{x},V)\right)}\\
					&+\sum_{i=1}^{M}{\left(f(\boldsymbol{x}+V^{-1}\boldsymbol{\nu}_i)\bar{r}^{\text{SPP}}_{-i}(V\boldsymbol{x}+\boldsymbol{\nu}_i,V)-f(\boldsymbol{x})\bar{r}^{\text{SPP}}_{-i}(V\boldsymbol{x},V)\right)},
				\end{aligned}
			\end{equation*}
			with 
			\begin{equation}\label{eq030405}
				\bar{r}^{\text{SPP}}_{\pm i}(V\boldsymbol{x},V)\triangleq\frac{\pi_V(\boldsymbol{x}\pm V^{-1}\boldsymbol{\nu}_i)r_{\mp i}(V\boldsymbol{x}\pm\boldsymbol{\nu}_i,V)}{\pi_V(\boldsymbol{x})}.
			\end{equation}
			It follows that $\bar{q}^{\text{SPP}}_{V}(\boldsymbol{x},t;\boldsymbol{x}_T,T)$ is the conditional probability of the reversed family $\{\bar{p}_V(\boldsymbol{x},t)=\pi_V(\boldsymbol{x})\}_{t\in [0,T]}$, which corresponds to the following homogeneous Markov jump process
			\begin{equation}\label{eq030406}
				\begin{aligned}
					\bar{\boldsymbol{x}}^{\text{SPP}}_{V}(t)=&\bar{\boldsymbol{x}}^{\text{SPP}}_{V}(0)\\
					+&V^{-1}\sum_{i=1}^{M}{\boldsymbol{\nu}_i\left\{Y_{+i}\left(\int_{0}^{t}\bar{r}^{\text{SPP}}_{+i}(V\bar{\boldsymbol{x}}^{\text{SPP}}_{V}(s),V)\mathrm{d}s\right)-Y_{-i}\left(\int_{0}^{t}\bar{r}^{\text{SPP}}_{-i}(V\bar{\boldsymbol{x}}^{\text{SPP}}_{V}(s),V)\mathrm{d}s\right)\right\}},
				\end{aligned}
			\end{equation}	
			with its initial condition restricted to be $\bar{\boldsymbol{x}}^{\text{SPP}}_{V}(0)=\boldsymbol{x}_T$ a.s..
		\end{itemize}
	\end{corollary}
	
	\subsection{Non-stationary Prehistory Probability}\label{sec0304}
	In this section, we further show that the idea of stationary prehistory probability can also be extended to the non-stationary case, in which analogous properties remain valid.
	\begin{definition}
		The non-stationary prehistory probability (NPP) is defined by 
		\begin{equation}\label{eq030501}
			q^{\text{NPP}}_{V}(\boldsymbol{x},t;\boldsymbol{x}_T,T;\boldsymbol{x}_0)\triangleq\frac{p_V(\boldsymbol{x},t\vert\boldsymbol{x}_0)p_V(\boldsymbol{x}_T,T-t\vert\boldsymbol{x})}{p_V(\boldsymbol{x}_T,T\vert\boldsymbol{x}_0)},\; t\in[0,T],
		\end{equation}
		with its time reversal satisfying
		\begin{equation}\label{eq030502}
			\bar{q}^{\text{NPP}}_{V}(\boldsymbol{x},t;\boldsymbol{x}_T,T;\boldsymbol{x}_0)\triangleq q^{\text{NPP}}_{V}(\boldsymbol{x},T-t;\boldsymbol{x}_T,T;\boldsymbol{x}_0)=\frac{p_V(\boldsymbol{x},T-t\vert\boldsymbol{x}_0)p_V(\boldsymbol{x}_T,t\vert\boldsymbol{x})}{p_V(\boldsymbol{x}_T,T\vert\boldsymbol{x}_0)},\; t\in[0,T].
		\end{equation}
	\end{definition}
	
	Consequently, we may reach analogous conclusions, as outlined below.
	\begin{corollary}
		\begin{itemize}
			\item[(a)] $q^{\text{NPP}}_{V}(\boldsymbol{x},t;\boldsymbol{x}_T,T;\boldsymbol{x}_0)$ satisfies the master equation
			\begin{equation}\label{eq030503}
				\frac{\mathrm{d}q^{\text{NPP}}_{V}(\boldsymbol{x},t;\boldsymbol{x}_T,T;\boldsymbol{x}_0)}{\mathrm{d}t}=\mathscr{L}^{*}_{\boldsymbol{x},t}q^{\text{NPP}}_{V}(\boldsymbol{x},t;\boldsymbol{x}_T,T;\boldsymbol{x}_0).
			\end{equation}
			Therefore, $P(\tilde{\boldsymbol{x}}_V(t)=x)=q^{\text{NPP}}_{V}(\boldsymbol{x},t;\boldsymbol{x}_T,T;\boldsymbol{x}_0)$ if the initial distribution obeys
			\begin{equation*}
				P(\tilde{\boldsymbol{x}}_V(0)=\boldsymbol{x})=q^{\text{NPP}}_{V}(\boldsymbol{x},0;\boldsymbol{x}_T,T;\boldsymbol{x}_0)=\begin{cases}
					1,\quad \boldsymbol{x}=\boldsymbol{x}_0\in V^{-1}\mathbb{N}^N, \\
					0,\quad \text{otherwise},
				\end{cases}
			\end{equation*}
			i.e., $\tilde{\boldsymbol{x}}_V(0)=\boldsymbol{x}_0$ a.s..
			\item[(b)] $\bar{q}^{\text{NPP}}_{V}(\boldsymbol{x},t;\boldsymbol{x}_T,T;\boldsymbol{x}_0)$ is the conditional probability of the reversed family $\{\bar{p}_V(\boldsymbol{x},t)=p_V(\boldsymbol{x},T-t\vert\boldsymbol{x}_0)\}_{t\in [0,T]}$. It naturally satisfies the master equation
			\begin{equation}\label{eq030504}
				\frac{\mathrm{d}\bar{q}^{\text{NPP}}_{V}(\boldsymbol{x},t;\boldsymbol{x}_T,T;\boldsymbol{x}_0)}{\mathrm{d}t}=\mathscr{N}^{*}_{\boldsymbol{x},t}\bar{q}^{\text{NPP}}_{V}(\boldsymbol{x},t;\boldsymbol{x}_T,T;\boldsymbol{x}_0),
			\end{equation}
			in which
			\begin{equation*}
				\begin{aligned}
					\mathscr{N}^{*}_{\boldsymbol{x},t}f(\boldsymbol{x},t)\triangleq&\sum_{i=1}^{M}{\left(f(\boldsymbol{x}-V^{-1}\boldsymbol{\nu}_i,t)\bar{r}^{\text{NPP}}_{+i}(V\boldsymbol{x}-\boldsymbol{\nu}_i,V,t)-f(\boldsymbol{x},t)\bar{r}^{\text{NPP}}_{+i}(V\boldsymbol{x},V,t)\right)}\\
					&+\sum_{i=1}^{M}{\left(f(\boldsymbol{x}+V^{-1}\boldsymbol{\nu}_i,t)\bar{r}^{\text{NPP}}_{-i}(V\boldsymbol{x}+\boldsymbol{\nu}_i,V,t)-f(\boldsymbol{x},t)\bar{r}^{\text{NPP}}_{-i}(V\boldsymbol{x},V,t)\right)},
				\end{aligned}
			\end{equation*}
			with
			\begin{equation}\label{eq030505}
				\bar{r}^{\text{NPP}}_{\pm i}(V\boldsymbol{x},V,t)\triangleq\frac{p_V(\boldsymbol{x}\pm V^{-1}\boldsymbol{\nu}_i,T-t\vert\boldsymbol{x}_0)r_{\mp i}(V\boldsymbol{x}\pm\boldsymbol{\nu}_i,V)}{p_V(\boldsymbol{x},T-t\vert\boldsymbol{x}_0)}.
			\end{equation}
			Let $\bar{\boldsymbol{x}}^{\text{NPP}}_{V}(t)$ be a non-homogeneous Markov jump process, which is defined by
			\begin{equation}\label{eq030506}
				\begin{aligned}
					\bar{\boldsymbol{x}}^{\text{NPP}}_{V}(t)=&\bar{\boldsymbol{x}}^{\text{NPP}}_{V}(0)\\
					+V^{-1}\sum_{i=1}^{M}&{\boldsymbol{\nu}_i\left\{Y_{+i}\left(\int_{0}^{t}\bar{r}^{\text{NPP}}_{+i}(V\bar{\boldsymbol{x}}^{\text{NPP}}_{V}(s),V,s)\mathrm{d}s\right)-Y_{-i}\left(\int_{0}^{t}\bar{r}^{\text{NPP}}_{-i}(V\bar{\boldsymbol{x}}^{\text{NPP}}_{V}(s),V,s)\mathrm{d}s\right)\right\}},
				\end{aligned}
			\end{equation}
			then we have $P(\bar{\boldsymbol{x}}^{\text{NPP}}_{V}(t)=\boldsymbol{x})=\bar{q}^{\text{NPP}}_{V}(\boldsymbol{x},t;\boldsymbol{x}_T,T;\boldsymbol{x}_0)$ if $\bar{\boldsymbol{x}}^{\text{NPP}}_{V}(0)=\boldsymbol{x}_T$ a.s..
		\end{itemize}
	\end{corollary}
	\begin{remark}
		Notice that both $q_{V}(\boldsymbol{x},t;\boldsymbol{x}_T,T)$ and $q^{\text{NPP}}_{V}(\boldsymbol{x},t;\boldsymbol{x}_T,T;\boldsymbol{x}_0)$ satisfy a master equation with the same operator $\mathscr{L}^{*}_{\boldsymbol{x},t}$. We conclude here that $q^{\text{NPP}}_{V}(\boldsymbol{x},t;\boldsymbol{x}_T,T;\boldsymbol{x}_0)$ is, in fact, the conditional probability of the family $\{q_{V}(\boldsymbol{x},t;\boldsymbol{x}_T,T)\}_{t\in[0,T]}$ as the Chapman–Kolmogorov-type equation 
		\begin{equation*}
			\sum_{\boldsymbol{x}_0\in V^{-1}\mathbb{N}^N}q_{V}(\boldsymbol{x}_0,0;\boldsymbol{x}_T,T)q^{\text{NPP}}_{V}(\boldsymbol{x},t;\boldsymbol{x}_T,T;\boldsymbol{x}_0)=q_{V}(\boldsymbol{x},t;\boldsymbol{x}_T,T),
		\end{equation*}
		is valid. 
	\end{remark}
	
	\section{Prehistorical Description of Optimal Fluctuations on Finite Time Intervals}\label{sec04}
	Assume that the conditions in part (e) of Proposition \ref{theo0206} hold. Let $\{\boldsymbol{\phi}_{\text{NOP}}(t;\boldsymbol{x}_T,T;\boldsymbol{x}_0):t\in [0,T]\}$ be the unique NOP connecting $\boldsymbol{x}_0$ and $\boldsymbol{x}_T$ in the time span $T$, and define $\hat{\boldsymbol{x}}_{\infty}(t)=\boldsymbol{\phi}_{\text{NOP}}(T-t;\boldsymbol{x}_T,T;\boldsymbol{x}_0)$. It follows from Eq. (\ref{eq020308}) that $\hat{\boldsymbol{x}}_{\infty}(t)$ satisfies
	\begin{equation}\label{eq0401}
		\hat{\boldsymbol{x}}_{\infty}(0)=\boldsymbol{x}_T,
	\end{equation}
	\begin{equation}\label{eq0402}
		\dot{\hat{\boldsymbol{x}}}_{\infty}(t)=\boldsymbol{G}(\hat{\boldsymbol{x}}_{\infty}(t),t),
	\end{equation}
	where
	\begin{equation*}
		\boldsymbol{G}(\boldsymbol{x},t)\triangleq-\sum_{i=1}^{M}\boldsymbol{\nu}_i\left(R_{+i}(\boldsymbol{x})e^{\boldsymbol{\nu}_i\cdot\nabla_{\boldsymbol{x}}S(\boldsymbol{x},T-t\vert\boldsymbol{x}_0)}-R_{-i}(\boldsymbol{x})e^{-\boldsymbol{\nu}_i\cdot\nabla_{\boldsymbol{x}}S(\boldsymbol{x},T-t\vert\boldsymbol{x}_0)}\right).
	\end{equation*}
	
	Let $\hat{\boldsymbol{x}}_{V}(t)$ be the Markov jump process defined by
	\begin{equation}\label{eq0403}
		\begin{aligned}
			\hat{\boldsymbol{x}}_{V}(t)=&\boldsymbol{x}^{**}(\boldsymbol{x}_T,V)\\
			&+V^{-1}\sum_{i=1}^{M}{\boldsymbol{\nu}_i\left\{Y_{+i}\left(\int_{0}^{t}\hat{r}_{+i}(V\hat{\boldsymbol{x}}_{V}(s),V,s)\mathrm{d}s\right)-Y_{-i}\left(\int_{0}^{t}\hat{r}_{-i}(V\hat{\boldsymbol{x}}_{V}(s),V,s)\mathrm{d}s\right)\right\}},
		\end{aligned}
	\end{equation}
	in which
	\begin{equation}\label{eq0404}
		\hat{r}_{\pm i}(V\boldsymbol{x},V,t)\triangleq\frac{p_V(\boldsymbol{x}\pm V^{-1}\boldsymbol{\nu}_i,T-t\vert\boldsymbol{x}^*(\boldsymbol{x}_0,V))r_{\mp i}(V\boldsymbol{x}\pm\boldsymbol{\nu}_i,V)}{p_V(\boldsymbol{x},T-t\vert\boldsymbol{x}^*(\boldsymbol{x}_0,V))}.
	\end{equation}
	Comparing it with Eq. (\ref{eq030506}), we know that
	\begin{equation*}
		P(\hat{\boldsymbol{x}}_{V}(t)=\boldsymbol{x})=\bar{q}^{\text{NPP}}_{V}(\boldsymbol{x},t;\boldsymbol{x}^{**}(\boldsymbol{x}_T,V),T;\boldsymbol{x}^{*}(\boldsymbol{x}_0,V)).
	\end{equation*}
	
	Denote
	\begin{equation*}
		\boldsymbol{G}_V(\boldsymbol{x},t)\triangleq{V}^{-1}\sum_{i=1}^{M}\boldsymbol{\nu}_i\left(\hat{r}_{+i}(V\boldsymbol{x},V,t)-\hat{r}_{-i}(V\boldsymbol{x},V,t)\right).
	\end{equation*}
	For any $T^*\in(0,T)$ and sufficiently small $\delta_0>0$, let $\hat{\Sigma}_{\delta_0,[0,T^*]}\triangleq\cup_{t\in[0,T^*]}B_{\delta_0}(\hat{\boldsymbol{x}}_{\infty}(t))$. Obviously, $\hat{\Sigma}_{\delta_0,[0,T^*]}=\Sigma_{\delta_0,[T-T^*,T]}$. The subsequent proposition of the type of the law of large numbers can be substantiated. The proof is provided in Appendix \ref{secA3}.
	
	\begin{proposition}\label{theo0401}
		Assume the conditions in Theorem \ref{theo0204}, part (e) of Proposition \ref{theo0206} and Lemma \ref{theo0207} hold. Then for any $T^*\in(0,T)$ there exists a constant $\delta_0$ so that for each $\delta<\delta_0$,
		\begin{equation}\label{eq0405}
			\lim_{V\to\infty}P\left(\sup_{t\in[0,T^*]}\left\vert\hat{\boldsymbol{x}}_{V}(t)-\hat{\boldsymbol{x}}_{\infty}(t)\right\vert>\delta\right)=0.
		\end{equation}
	\end{proposition}
	Define
	\begin{equation*}
		{\boldsymbol{J}}_{V}(\boldsymbol{x},t)\triangleq{V}^{-1}\sum_{i=1}^{M}\boldsymbol{\nu}_i\otimes\boldsymbol{\nu}_i\left(\hat{r}_{+i}(V\boldsymbol{x},V,t)+\hat{r}_{-i}(V\boldsymbol{x},V,t)\right),
	\end{equation*}
	and
	\begin{equation*}
		{\boldsymbol{J}}(\boldsymbol{x},t)\triangleq\sum_{i=1}^{M}\boldsymbol{\nu}_i\otimes\boldsymbol{\nu}_i\left(R_{+i}(\boldsymbol{x})e^{\boldsymbol{\nu}_i\cdot\nabla_{\boldsymbol{x}}S(\boldsymbol{x},T-t\vert\boldsymbol{x}_0)}+R_{-i}(\boldsymbol{x})e^{-\boldsymbol{\nu}_i\cdot\nabla_{\boldsymbol{x}}S(\boldsymbol{x},T-t\vert\boldsymbol{x}_0)}\right).
	\end{equation*}
	The standard deviation of the type of the central limit theorem can also be estimated. We state it as a proposition, with the subsequent proof being postponed to the Appendix \ref{secA4}.
	
	\begin{proposition}\label{theo0402}
		In addition to the hypotheses formulated in Proposition \ref{theo0401}, we further suppose that for any $T^*\in(0,T)$ there exists a constant $\delta_0$ so that 
		\begin{equation}\label{eq0406}
			\lim_{V\to\infty}\sqrt{V}\vert{\boldsymbol{G}}_V(\boldsymbol{x},t)-{\boldsymbol{G}}(\boldsymbol{x},t)\vert=0,
		\end{equation}
		uniformly for $(\boldsymbol{x},t)\in\hat{\Sigma}_{\delta_0,[0,T^*]}\times[0,T^*]$.
		
		Then, for any $T^*\in(0,T)$, $\boldsymbol{\mu}_V(t)\triangleq\sqrt{V}(\hat{\boldsymbol{x}}_V(t)-\hat{\boldsymbol{x}}_{\infty}(t))$ converges weakly on the interval $[0,T^*]$ to a diffusion process $\boldsymbol{\mu}_{\infty}(t)$, with $\boldsymbol{\mu}_{\infty}(0)=0$ and its characteristic function $\Phi(\boldsymbol{\theta},t)$ satisfying
		\begin{equation}\label{eq0407}
			\frac{\partial}{\partial t}\Phi(\boldsymbol{\theta},t)=-\frac{1}{2}\left(\boldsymbol{\theta}\cdot{\boldsymbol{J}}(\hat{\boldsymbol{x}}_{\infty}(t),t)\cdot\boldsymbol{\theta}\right)\Phi(\boldsymbol{\theta},t)+\boldsymbol{\theta}\cdot\nabla_{\boldsymbol{x}}\boldsymbol{G}(\hat{\boldsymbol{x}}_{\infty}(t),t)\cdot\nabla_{\boldsymbol{\theta}}\Phi(\boldsymbol{\theta},t),\quad t\in[0,T^*].
		\end{equation}
	\end{proposition}
	\begin{remark}
		In fact, $\boldsymbol{\mu}_{\infty}(t)$ is a Gaussian process that obeys the following stochastic differential equation
		\begin{equation*}
			\mathrm{d}\boldsymbol{\mu}_{\infty}(t)=\nabla_{\boldsymbol{x}}\boldsymbol{G}(\hat{\boldsymbol{x}}_{\infty}(t),t)\cdot\boldsymbol{\mu}_{\infty}(t)\mathrm{d}t+\boldsymbol{\sigma}(\hat{\boldsymbol{x}}_{\infty}(t),t)\cdot\mathrm{d}\boldsymbol{w}(t),
		\end{equation*}
		in which $\boldsymbol{\sigma}(\boldsymbol{x},t)\cdot\boldsymbol{\sigma}^{\top}(\boldsymbol{x},t)=\boldsymbol{J}(\boldsymbol{x},t)$.
	\end{remark}
	\begin{corollary}
		Under the assumptions of Propositions \ref{theo0401} and \ref{theo0402}, we can conclude that for any $T^*\in(0,T)$,
		\begin{itemize}
			\item[(a)] $\bar{q}^{\text{NPP}}_{V}(\boldsymbol{x},t;\boldsymbol{x}^{**}(\boldsymbol{x}_T,V),T;\boldsymbol{x}^{*}(\boldsymbol{x}_0,V))$ will focus on $\hat{\boldsymbol{x}}_{\infty}(t)=\boldsymbol{\phi}_{\text{NOP}}(T-t;\boldsymbol{x}_T,T;\boldsymbol{x}_0)$ as $V\to\infty$, and the focusing effect holds uniformly for $t\in[0,T^*]$;
			\item[(b)] for each $t\in[0,T^*]$ and sufficiently large $V$,  $\bar{q}^{\text{NPP}}_{V}(\boldsymbol{x},t;\boldsymbol{x}^{**}(\boldsymbol{x}_T,V),T;\boldsymbol{x}^{*}(\boldsymbol{x}_0,V))$ exhibits approximate conformity to a Gaussian distribution in the vicinity of $\hat{\boldsymbol{x}}_{\infty}(t)$, i.e., for $\boldsymbol{x}$ in the ${O}({1}/{\sqrt{V}})$ neighborhood of $\hat{\boldsymbol{x}}_{\infty}(t)$, 
			\begin{equation*}
				\bar{q}^{\text{NPP}}_{V}(\boldsymbol{x},t;\boldsymbol{x}^{**}(\boldsymbol{x}_T,V),T;\boldsymbol{x}^{*}(\boldsymbol{x}_0,V))\simeq\exp\left\{-\frac{V}{2}(\boldsymbol{x}-\hat{\boldsymbol{x}}_{\infty}(t))\cdot\bar{\boldsymbol{\kappa}}^{-1}(t)\cdot(\boldsymbol{x}-\hat{\boldsymbol{x}}_{\infty}(t))\right\},
			\end{equation*}
			in which $\bar{\boldsymbol{\kappa}}(t)$ satisfies the following Lyapunov matrix differential equation
			\begin{equation*}
				\dot{\bar{\boldsymbol{\kappa}}}(t)=\nabla_{\boldsymbol{x}}\boldsymbol{G}(\hat{\boldsymbol{x}}_{\infty}(t),t)\cdot\bar{\boldsymbol{\kappa}}(t)+\bar{\boldsymbol{\kappa}}(t)\cdot\nabla_{\boldsymbol{x}}\boldsymbol{G}^{\top}(\hat{\boldsymbol{x}}_{\infty}(t),t)+\boldsymbol{J}(\hat{\boldsymbol{x}}_{\infty}(t),t),
			\end{equation*}
			for $t\in[0,T^*]$, with the constraint $\bar{\boldsymbol{\kappa}}(0)=\boldsymbol{0}$.
		\end{itemize}
	\end{corollary}
	
	Note that the law of large numbers and the central limit theorem presented above are restricted to the interval $[0,T^*]$. To achieve the prehistorical description for the entire interval $[0,T]$, limit theorems for stochastic processes of the form (\ref{eq030306}) are also required. 
	
	Assume the conditions in part (e) of Proposition \ref{theo0206} hold. Define $\breve{\boldsymbol{x}}_{\infty}(t)=\boldsymbol{\phi}_{\text{NOP}}(t;\boldsymbol{x}_T,T;\boldsymbol{x}_0)$. Eq. (\ref{eq020309}) implies that $\breve{\boldsymbol{x}}_{\infty}(t)$ satisfies
	\begin{equation}\label{eq0408}
		\breve{\boldsymbol{x}}_{\infty}(0)=\boldsymbol{x}_0,
	\end{equation}
	\begin{equation}\label{eq0409}
		\dot{\breve{\boldsymbol{x}}}_{\infty}(t)=\boldsymbol{Q}(\breve{\boldsymbol{x}}_{\infty}(t),t),
	\end{equation}
	where
	\begin{equation*}
		\boldsymbol{Q}(\boldsymbol{x},t)\triangleq\sum_{i=1}^{M}\boldsymbol{\nu}_i\left(R_{+i}(\boldsymbol{x})e^{-\boldsymbol{\nu}_i\cdot\nabla_{\boldsymbol{x}}S(\boldsymbol{x}_T,T-t\vert\boldsymbol{x})}-R_{-i}(\boldsymbol{x})e^{\boldsymbol{\nu}_i\cdot\nabla_{\boldsymbol{x}}S(\boldsymbol{x}_T,T-t\vert\boldsymbol{x})}\right).
	\end{equation*}
	
	Let $\breve{\boldsymbol{x}}_{V}(t)$ be the Markov jump process defined by
	\begin{equation}\label{eq0410}
		\begin{aligned}
			\breve{\boldsymbol{x}}_{V}(t)=&\boldsymbol{x}^{*}(\boldsymbol{x}_0,V)\\
			&+V^{-1}\sum_{i=1}^{M}{\boldsymbol{\nu}_i\left\{Y_{+i}\left(\int_{0}^{t}\breve{r}_{+i}(V\breve{\boldsymbol{x}}_{V}(s),V,s)\mathrm{d}s\right)-Y_{-i}\left(\int_{0}^{t}\breve{r}_{-i}(V\breve{\boldsymbol{x}}_{V}(s),V,s)\mathrm{d}s\right)\right\}},
		\end{aligned}
	\end{equation}
	in which
	\begin{equation}\label{eq0411}
		\breve{r}_{\pm i}(V\boldsymbol{x},V,t)\triangleq\frac{p_V(\boldsymbol{x}^{**}(\boldsymbol{x}_T,V),T-t\vert\boldsymbol{x}\pm V^{-1}\boldsymbol{\nu}_i)r_{\pm i}(V\boldsymbol{x},V)}{p_V(\boldsymbol{x}^{**}(\boldsymbol{x}_T,V),T-t\vert\boldsymbol{x})}.
	\end{equation}
	Obviously, $P(\breve{\boldsymbol{x}}_{V}(t)=\boldsymbol{x})=q^{\text{NPP}}_{V}(\boldsymbol{x},t;\boldsymbol{x}^{**}(\boldsymbol{x}_T,V),T;\boldsymbol{x}^{*}(\boldsymbol{x}_0,V))$. 
	
	Let 
	\begin{equation*}
		\boldsymbol{Q}_{V}(\boldsymbol{x},t)\triangleq{V}^{-1}\sum_{i=1}^{M}\boldsymbol{\nu}_i\left(\breve{r}_{+i}(V\boldsymbol{x},V,t)-\breve{r}_{-i}(V\boldsymbol{x},V,t)\right),
	\end{equation*}
	\begin{equation*}
		\boldsymbol{W}_{V}(\boldsymbol{x},t)\triangleq{V}^{-1}\sum_{i=1}^{M}\boldsymbol{\nu}_i\otimes\boldsymbol{\nu}_i\left(\breve{r}_{+i}(V\boldsymbol{x},V,t)+\breve{r}_{-i}(V\boldsymbol{x},V,t)\right),
	\end{equation*}
	and
	\begin{equation*}
		\boldsymbol{W}(\boldsymbol{x},t)\triangleq\sum_{i=1}^{M}\boldsymbol{\nu}_i\otimes\boldsymbol{\nu}_i\left(R_{+i}(\boldsymbol{x})e^{-\boldsymbol{\nu}_i\cdot\nabla_{\boldsymbol{x}}S(\boldsymbol{x}_T,T-t\vert\boldsymbol{x})}+R_{-i}(\boldsymbol{x})e^{\boldsymbol{\nu}_i\cdot\nabla_{\boldsymbol{x}}S(\boldsymbol{x}_T,T-t\vert\boldsymbol{x})}\right).
	\end{equation*}
	For any $T^*\in(0,T)$ and sufficiently small $\delta_0>0$, let ${\Sigma}_{\delta_0,[0,T^*]}\triangleq\cup_{t\in[0,T^*]}B_{\delta_0}(\breve{\boldsymbol{x}}_{\infty}(t))$.
	The following result is in the nature of the law of large numbers and the central limit theorem. The proof is largely analogous to those of Propositions \ref{theo0401} and \ref{theo0402}. We will not repeat it.
	
	\begin{proposition}\label{theo0403}
		In addition to the conditions in Theorem \ref{theo0204} and part (e) of Proposition \ref{theo0206}, we assume that for any $T^{*}\in(0,T)$, there is a constant $\delta_0$ so that
		\begin{itemize}
			\item[(a)] $S(\boldsymbol{x}_T,T-t\vert\boldsymbol{x})$ is at least twice continuous differentiable in $\Sigma_{\delta_0,[0,T^*]}\times[0,T^*]$;
			\item[(b)] 
			$\lim_{V\to\infty}\left\vert\frac{p_V(\boldsymbol{x}^{**}(\boldsymbol{x}_T,V),T-t\vert\boldsymbol{x}^{**}(\boldsymbol{x},V)\pm V^{-1}\boldsymbol{\nu}_i)}{p_V(\boldsymbol{x}^{**}(\boldsymbol{x}_T,V),T-t\vert\boldsymbol{x}^{**}(\boldsymbol{x},V))}-e^{\mp \boldsymbol{\nu}_i\cdot\nabla_{\boldsymbol{x}}S(\boldsymbol{x}_T,T-t\vert\boldsymbol{x})}\right\vert=0$,
			uniformly for $(\boldsymbol{x},t)\in\Sigma_{\delta_0,[0,T^*]}\times[0,T^*]$ and $1\leq{i}\leq{M}$.
		\end{itemize}
		Then for any $T^{*}\in(0,T)$, there is a constant $\delta_0$ so that for each $\delta<\delta_0$,
		\begin{equation}\label{eq0412}
			\lim_{V\to\infty}P\left\{\sup_{t\in[0,T^*]}\left\vert\breve{\boldsymbol{x}}_{V}(t)-\breve{\boldsymbol{x}}_{\infty}(t)\right\vert>\delta\right\}=0.
		\end{equation}
		
		Furthermore, if we also assume that
		\begin{itemize}
			\item[(c)] 
			$\lim_{V\to\infty}\sqrt{V}\vert\boldsymbol{Q}_V(\boldsymbol{x},t)-\boldsymbol{Q}(\boldsymbol{x},t)\vert=0$,
			uniformly for $(\boldsymbol{x},t)\in\Sigma_{\delta_0,[0,T^*]}\times[0,T^*]$.
		\end{itemize}
		Then for any $T^{*}\in(0,T)$, $\boldsymbol{\upsilon}_{V}(t)\triangleq\sqrt{V}(\breve{\boldsymbol{x}}_V(t)-\breve{\boldsymbol{x}}_{\infty}(t))$ converges weakly on the interval $[0,T^*]$ to a diffusion process $\boldsymbol{\upsilon}_{\infty}(t)$, with $\boldsymbol{\upsilon}_{\infty}(0)=0$ and its characteristic function $\Psi(\boldsymbol{\theta},t)$ satisfying
		\begin{equation}\label{eq0413}
			\frac{\partial}{\partial{t}}\Psi(\boldsymbol{\theta},t)=-\frac{1}{2}\left(\boldsymbol{\theta}\cdot\boldsymbol{W}(\breve{\boldsymbol{x}}_{\infty}(t),t)\cdot\boldsymbol{\theta}\right)\Psi(\boldsymbol{\theta},t)+\boldsymbol{\theta}\cdot\nabla_{\boldsymbol{x}}\boldsymbol{Q}(\breve{\boldsymbol{x}}_{\infty}(t),t)\cdot\nabla_{\boldsymbol{\theta}}\Psi(\boldsymbol{\theta},t)\quad t\in[0,T^*].
		\end{equation}
	\end{proposition}
	\begin{corollary} Under the assumptions of Propositions \ref{theo0401}, \ref{theo0402} and \ref{theo0403}, we have that
		\begin{itemize}
			\item[(a)] the non-stationary prehistory probability ${q}^{\text{NPP}}_{V}(\boldsymbol{x},t;\boldsymbol{x}^{**}(\boldsymbol{x}_T,V),T;\boldsymbol{x}^{*}(\boldsymbol{x}_0,V))$ will focus on $\boldsymbol{\phi}_{\text{NOP}}(t;\boldsymbol{x}_T,T;\boldsymbol{x}_0)$, uniformly for $t\in[0,T]$, as $V\to\infty$; 
			\item[(b)] for each $t\in[0,T]$, sufficiently large $V$, and each $\boldsymbol{x}$ in the ${O}({1}/{\sqrt{V}})$ neighborhood of $\boldsymbol{\phi}_{\text{NOP}}(t;\boldsymbol{x}_T,T;\boldsymbol{x}_0)$, 
			\begin{equation*}
				\begin{aligned}
					{q}^{\text{NPP}}_{V}&(\boldsymbol{x},t;\boldsymbol{x}^{**}(\boldsymbol{x}_T,V),T;\boldsymbol{x}^{*}(\boldsymbol{x}_0,V))\simeq\\
					&\exp\left\{-\frac{V}{2}(\boldsymbol{x}-\boldsymbol{\phi}_{\text{NOP}}(t;\boldsymbol{x}_T,T;\boldsymbol{x}_0))\cdot\boldsymbol{\kappa}^{-1}(t)\cdot(\boldsymbol{x}-\boldsymbol{\phi}_{\text{NOP}}(t;\boldsymbol{x}_T,T;\boldsymbol{x}_0))\right\},
				\end{aligned}
			\end{equation*}
			where ${\boldsymbol{\kappa}}(t)=\bar{\boldsymbol{\kappa}}(T-t)$ satisfies the following Lyapunov matrix differential equation
			\begin{equation*}
				\dot{\boldsymbol{\kappa}}(t)=\nabla_{\boldsymbol{x}}\boldsymbol{Q}(\breve{\boldsymbol{x}}_{\infty}(t),t)\cdot\boldsymbol{\kappa}(t)+\boldsymbol{\kappa}(t)\cdot\nabla_{\boldsymbol{x}}\boldsymbol{Q}^{\top}(\breve{\boldsymbol{x}}_{\infty}(t),t)+\boldsymbol{W}(\breve{\boldsymbol{x}}_{\infty}(t),t),
			\end{equation*}
			for $t\in[0,T]$, with the constraints ${\boldsymbol{\kappa}}(0)=\boldsymbol{0}$ and ${\boldsymbol{\kappa}}(T)=\boldsymbol{0}$.
		\end{itemize}
		This is the complete prehistorical description of optimal fluctuations in the non-stationary setting.
	\end{corollary}
	
	\section{Prehistorical Description of Optimal Fluctuations on Infinite Time Intervals}\label{sec05}
	Assume that the conditions in part (f) of Proposition \ref{theo0209} hold. Let $\{\boldsymbol{\phi}_{\text{OP}}(t;\boldsymbol{x}_{T}):t\in(-\infty,0]\}$ be the unique (stationary) optimal path connecting $\boldsymbol{x}_{\text{eq}}$ and $\boldsymbol{x}_T$, and define $\hat{\boldsymbol{x}}_{\infty}(t)=\boldsymbol{\phi}_{\text{OP}}(-t;\boldsymbol{x}_T)$. Eq. (\ref{eq020407}) gives
	\begin{equation}\label{eq0501}
		\hat{\boldsymbol{x}}_{\infty}(0)=\boldsymbol{x}_T,
	\end{equation}
	\begin{equation}\label{eq0502}
		\dot{\hat{\boldsymbol{x}}}_{\infty}(t)=\boldsymbol{G}(\hat{\boldsymbol{x}}_{\infty}(t)),
	\end{equation}
	in which 
	\begin{equation*}
		\boldsymbol{G}(\boldsymbol{x})\triangleq-\sum_{i=1}^{M}\boldsymbol{\nu}_i\left(R_{+i}(\boldsymbol{x})e^{\boldsymbol{\nu}_i\cdot\nabla_{\boldsymbol{x}}S(\boldsymbol{x})}-R_{-i}(\boldsymbol{x})e^{-\boldsymbol{\nu}_i\cdot\nabla_{\boldsymbol{x}}S(\boldsymbol{x})}\right).
	\end{equation*}
	
	Define a homogeneous Markov jump process $\hat{\boldsymbol{x}}_{V}(t)$, which is given by
	\begin{equation}\label{eq0503}
		\begin{aligned}
			\hat{\boldsymbol{x}}_{V}(t)=&\boldsymbol{x}^{**}(\boldsymbol{x}_T,V)\\
			&+V^{-1}\sum_{i=1}^{M}{\boldsymbol{\nu}_i\left\{Y_{+i}\left(\int_{0}^{t}\hat{r}_{+i}(V\hat{\boldsymbol{x}}_{V}(s),V)\mathrm{d}s\right)-Y_{-i}\left(\int_{0}^{t}\hat{r}_{-i}(V\hat{\boldsymbol{x}}_{V}(s),V)\mathrm{d}s\right)\right\}},
		\end{aligned}
	\end{equation}
	with 
	\begin{equation}\label{eq0504}
		\hat{r}_{\pm i}(V\boldsymbol{x},V)\triangleq\frac{\pi_V(\boldsymbol{x}\pm V^{-1}\boldsymbol{\nu}_i)r_{\mp i}(V\boldsymbol{x}\pm\boldsymbol{\nu}_i,V)}{\pi_V(\boldsymbol{x})}.
	\end{equation}
	Clearly, in this case, $P(\hat{\boldsymbol{x}}_{V}(t)=\boldsymbol{x})=\bar{q}^{\text{SPP}}_{V}(\boldsymbol{x},t;\boldsymbol{x}^{**}(\boldsymbol{x}_T,V),T)$.
	
	Define
	\begin{equation*}
		\boldsymbol{G}_V(\boldsymbol{x})\triangleq{V}^{-1}\sum_{i=1}^{M}\boldsymbol{\nu}_i\left(\hat{r}_{+i}(V\boldsymbol{x},V)-\hat{r}_{-i}(V\boldsymbol{x},V)\right),
	\end{equation*}
	\begin{equation*}
		\boldsymbol{J}_V(\boldsymbol{x})\triangleq{V}^{-1}\sum_{i=1}^{M}\boldsymbol{\nu}_{i}\otimes\boldsymbol{\nu}_{i}\left(\hat{r}_{+i}(V\boldsymbol{x},V)+\hat{r}_{-i}(V\boldsymbol{x},V)\right),
	\end{equation*}
	and
	\begin{equation*}
		\boldsymbol{J}(\boldsymbol{x})\triangleq\sum_{i=1}^{M}\boldsymbol{\nu}_{i}\otimes\boldsymbol{\nu}_{i}\left(R_{+i}(\boldsymbol{x})e^{\boldsymbol{\nu}_i\cdot\nabla_{\boldsymbol{x}}S(\boldsymbol{x})}+R_{-i}(\boldsymbol{x})e^{-\boldsymbol{\nu}_i\cdot\nabla_{\boldsymbol{x}}S(\boldsymbol{x})}\right).
	\end{equation*}
	For any $T^*\in(0,\infty)$ and sufficiently small $\delta_0>0$, let $\hat{\Sigma}_{\delta_0,[0,T^*]}=\cup_{t\in[0,T^*]}B_{\delta_0}(\hat{\boldsymbol{x}}_{\infty}(t))$. Obviously, in this case, $\hat{\Sigma}_{\delta_0,[0,T^*]}=\Sigma_{\delta_0,[-T^*,0]}$.
	We state the following analogs of Propositions \ref{theo0401} and \ref{theo0402}. The proofs are essentially the same, and will not be repeated.

	\begin{proposition}\label{theo0501}
		Assume the conditions in Theorems \ref{theo0204}, \ref{theo0208}, part (f) of Proposition \ref{theo0209} and Lemma \ref{theo0210} hold. Then for any $T^*\in(0,\infty)$ there is a constant $\delta_0$ so that for each $\delta<\delta_0$, we have 
		\begin{equation}\label{eq0505}
			\lim_{V\to\infty}P\left\{\sup_{t\in[0,T^*]}\left\vert\hat{\boldsymbol{x}}_{V}(t)-\hat{\boldsymbol{x}}_{\infty}(t)\right\vert>\delta\right\}=0.
		\end{equation}
	\end{proposition}
	
	\begin{proposition}\label{theo0502}
		In addition to the hypotheses formulated in Proposition \ref{theo0501}, we
		further suppose that for any $T^*\in(0,\infty)$, there is a constant $\delta_0$ so that
		\begin{equation*}
			\lim_{V\to\infty}\sqrt{V}\vert\boldsymbol{G}_V(\boldsymbol{x})-\boldsymbol{G}(\boldsymbol{x})\vert=0,
		\end{equation*}
		uniformly for $\boldsymbol{x}\in\hat{\Sigma}_{\delta_0,[0,T^*]}$.
		
		Then for any $T^*\in(0,\infty)$, $\boldsymbol{\mu}_V(t)\triangleq\sqrt{V}(\hat{\boldsymbol{x}}_V(t)-\hat{\boldsymbol{x}}_{\infty}(t))$ converges weakly on the interval $[0,T^*]$ to a diffusion process $\boldsymbol{\mu}_{\infty}(t)$, with $\boldsymbol{\mu}_{\infty}(0)=0$ and its characteristic function $\Phi(\boldsymbol{\theta},t)$ satisfying
		\begin{equation}\label{eq0506}
			\frac{\partial}{\partial{t}}\Phi(\boldsymbol{\theta},t)=-\frac{1}{2}\left(\boldsymbol{\theta}\cdot\boldsymbol{J}(\hat{\boldsymbol{x}}_{\infty}(t))\cdot\boldsymbol{\theta}\right)\Phi(\boldsymbol{\theta},t)+\boldsymbol{\theta}\cdot\nabla_{\boldsymbol{x}}\boldsymbol{G}(\hat{\boldsymbol{x}}_{\infty}(t))\cdot\nabla_{\boldsymbol{\theta}}\Phi(\boldsymbol{\theta},t),\quad t\in[0,T^*].
		\end{equation}
	\end{proposition}
	Although the definition of the stationary prehistory probability is contingent on a specified parameter $T$, it can be observed that the process $\hat{\boldsymbol{x}}_V(t)$ (as well as $\bar{\boldsymbol{x}}_V^{\text{SPP}}$) itself is not influenced by this quantity. The aforementioned theorems are applicable to any $T^*\in(0,\infty)$ as long as the pertinent conditions are met. At this point, if we let $T^*=T$, then the subsequent conclusions can be derived.
	\begin{corollary}
		Under the assumptions of Propositions \ref{theo0501} and \ref{theo0502}, we can conclude that for any $T>0$,
		\begin{itemize}
			\item[(a)] the stationary prehistory probability ${q}^{\text{SPP}}_{V}(\boldsymbol{x},t;\boldsymbol{x}^{**}(\boldsymbol{x}_T,V),T)$ will focus on $\boldsymbol{\phi}_{\text{OP}}(t-T;\boldsymbol{x}_T)$, uniformly for $t\in[0,T]$, as $V\to\infty$;
			\item[(b)] for each $t\in[0,T]$, sufficiently large $V$, and each $\boldsymbol{x}$ in the ${O}({1}/{\sqrt{V}})$ neighborhood of $\boldsymbol{\phi}_{\text{OP}}(t-T;\boldsymbol{x}_T)$, 
			\begin{equation*}
				\begin{aligned}
					{q}^{\text{SPP}}_{V}(\boldsymbol{x},t;&\boldsymbol{x}^{**}(\boldsymbol{x}_T,V),T)\simeq\\
					&\exp\left\{-\frac{V}{2}(\boldsymbol{x}-\boldsymbol{\phi}_{\text{OP}}(t-T;\boldsymbol{x}_T))\cdot\bar{\boldsymbol{\kappa}}^{-1}(T-t)\cdot(\boldsymbol{x}-\boldsymbol{\phi}_{\text{OP}}(t-T;\boldsymbol{x}_T))\right\},
				\end{aligned}
			\end{equation*}
			in which $\bar{\boldsymbol{\kappa}}(t)$ satisfies the Lyapunov matrix differential equation
			\begin{equation*}
				\dot{\bar{\boldsymbol{\kappa}}}(t)=\nabla_{\boldsymbol{x}}\boldsymbol{G}(\hat{\boldsymbol{x}}_{\infty}(t))\cdot\bar{\boldsymbol{\kappa}}(t)+\bar{\boldsymbol{\kappa}}(t)\cdot\nabla_{\boldsymbol{x}}\boldsymbol{G}^{\top}(\hat{\boldsymbol{x}}_{\infty}(t))+\boldsymbol{J}(\hat{\boldsymbol{x}}_{\infty}(t)),
			\end{equation*}
			for $t\in[0,T]$, with the initial condition $\dot{\bar{\boldsymbol{\kappa}}}(0)=\boldsymbol{0}$. 
		\end{itemize}
		This restores the complete prehistorical description of optimal fluctuations in the stationary setting.
	\end{corollary}
	\begin{remark}
		In the vicinity of the equilibrium $\boldsymbol{x}_{\text{eq}}$, $S(\boldsymbol{x})$ can be locally approximated by the quadratic form $S(\boldsymbol{x})\simeq 1/2(\boldsymbol{x}-\boldsymbol{x}_{\text{eq}})\cdot\boldsymbol{\iota}\cdot(\boldsymbol{x}-\boldsymbol{x}_{\text{eq}})$, where the symmetric matrix $\boldsymbol{\iota}$ satisfies the following Riccati algebraic equation 
		\begin{equation*}
			\nabla_{\boldsymbol{x}}\boldsymbol{F}^{\top}(\boldsymbol{x}_{\text{eq}})\cdot\boldsymbol{\iota}+\boldsymbol{\iota}\cdot\nabla_{\boldsymbol{x}}\boldsymbol{F}(\boldsymbol{x}_{\text{eq}})+\boldsymbol{\iota}\cdot\boldsymbol{J}(\boldsymbol{x}_{\text{eq}})\cdot\boldsymbol{\iota}=0.
		\end{equation*}
		(See also the fluctuation-dissipation theorem for stochastic chemical reaction models for details \cite{Hao_Ge_2017}.)
		
		Combined with the fact $\lim_{t\to\infty}\hat{\boldsymbol{x}}_{\infty}(t)=\boldsymbol{x}_{\text{eq}}$, we know
		\begin{equation*}
			\lim_{t\to\infty}\bar{\boldsymbol{\kappa}}(t)=\bar{\boldsymbol{\kappa}}_{\infty}=\boldsymbol{\iota}^{-1},
		\end{equation*}
		where $\bar{\boldsymbol{\kappa}}_{\infty}$ obeys the Lyapunov algebraic equation
		\begin{equation*}
			\nabla_{\boldsymbol{x}}\boldsymbol{F}(\boldsymbol{x}_{\text{eq}})\cdot\bar{\boldsymbol{\kappa}}_{\infty}+\bar{\boldsymbol{\kappa}}_{\infty}\cdot\nabla_{\boldsymbol{x}}\boldsymbol{F}^{\top}(\boldsymbol{x}_{\text{eq}})+\boldsymbol{J}(\boldsymbol{x}_{\text{eq}})=0.
		\end{equation*}
	\end{remark}
	\begin{remark}
		One can also use the techniques outlined in chapters 4 and 6 of [\onlinecite{Freidlin_2012}] to generalize the results of this section to cases where the corresponding deterministic system possesses either other types of attractive invariant sets (such as limit cycles), or multiple coexisting attractors.
	\end{remark}
	\section{Numerical Examples}\label{sec06}
	\subsection{A Chemical Monostable System}
	We consider the case of a reaction $A\ce{<=>[\mathit{r_{+}}][\mathit{r_{-}}]}S$ in which the concentration of $A$ is constant. In Delbr\"uck-Gillespie's exposition on the subject of chemical kinetics,
	\begin{equation*}
		r_{+}(n,V)=k_{+}n_{A},\quad r_{-}(n,V)=k_{-}n,\quad \nu=1,
	\end{equation*}
	with
	\begin{equation*}
		R_{+}(x)=k_{+}a,\quad R_{-}(x)=k_{-}x,
	\end{equation*}
	where $n_A$ and $a$ represent the number of molecules and the concentration of $A$, respectively. For simplicity, the parameter is set to $k_{+}a=1$ and $k_{-}=1$. The deterministic equation (\ref{eq020102}) yields a unique stable fixed point $x_{eq}=1$. Now we utilize the algorithms presented in Appendix \ref{secA6} to show how the non-stationary and stationary prehistory probabilities can be employed to approximate the NOP and the OP, respectively.
	
	\begin{figure*}
		\centering
		\subfigure{\begin{minipage}[b]{.3\linewidth}
				
				\includegraphics[scale=0.2]{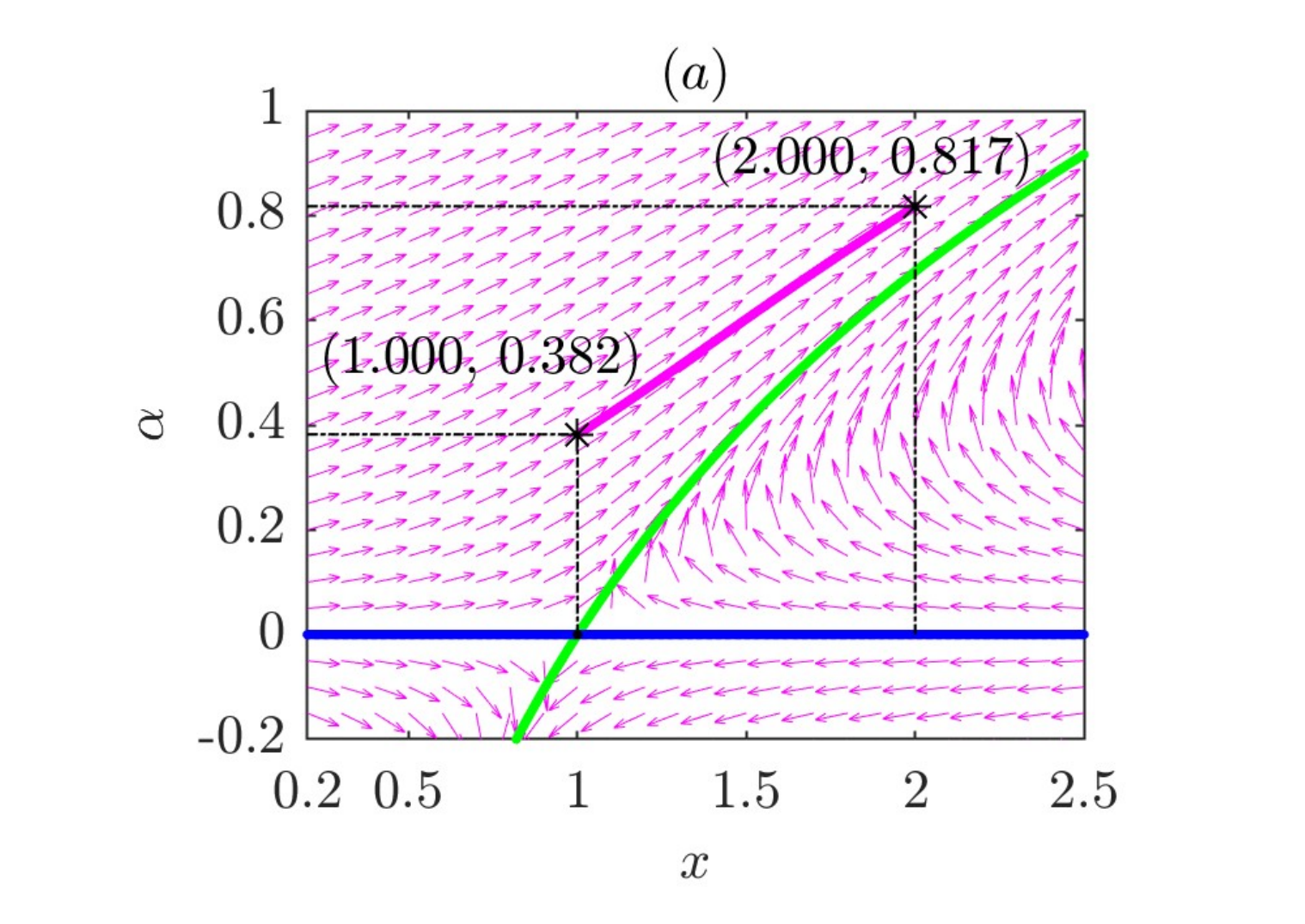}
			\end{minipage}
			\label{fig01a}
		}
		\subfigure{\begin{minipage}[b]{.3\linewidth}
				\includegraphics[scale=0.2]{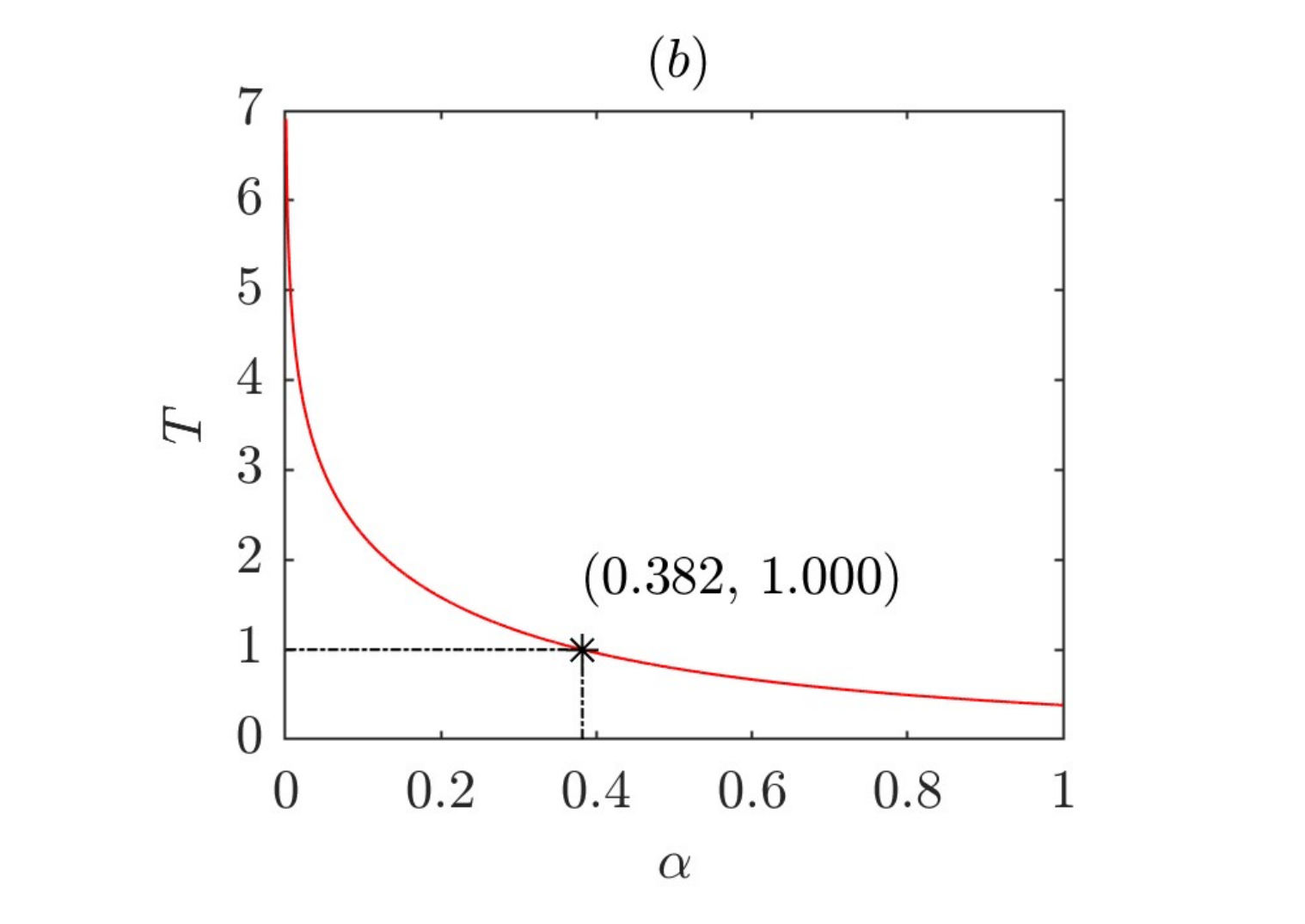}
			\end{minipage}
			\label{fig01b}
		}
		\caption{(a) The Hamiltonian vector field, the stable (blue) and unstable (green) manifolds of the fixed point $(1,0)$. The unique (magenta) solution to the constrained Hamiltonian problem with $x_0=1$, $x_T=2$, $T=1$. (b) The profiles of $T(\alpha)$ versus $\alpha$ for $x_0=1$, $x_T=2$. The monotonicity indicates that each $T$ corresponds to a unique NOP.}
		\label{fig01}
	\end{figure*}
	\begin{figure*}
		\centering
		\subfigure{\begin{minipage}[b]{.3\linewidth}
				\centering
				\includegraphics[scale=0.2]{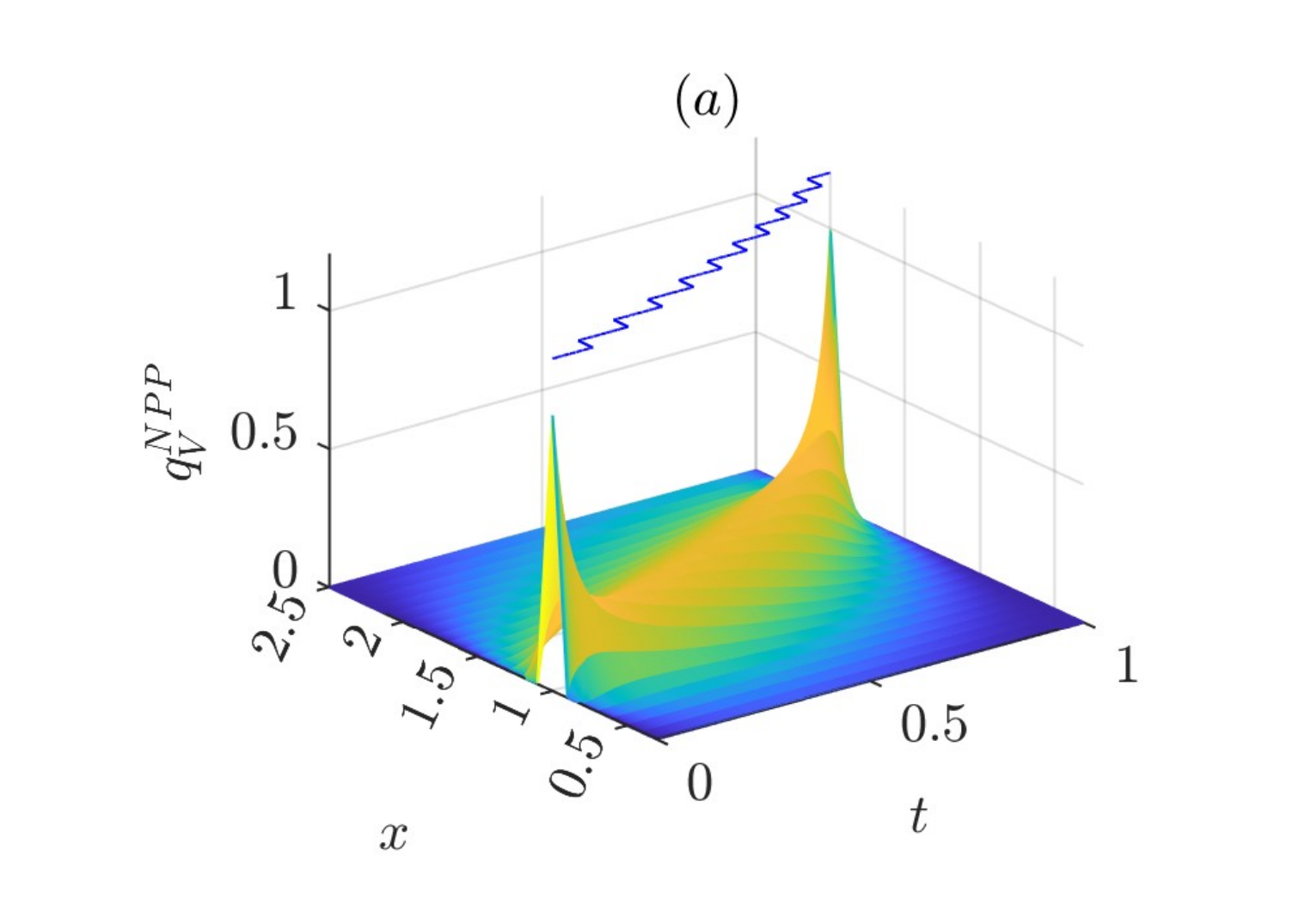}
			\end{minipage}
			\label{fig02a}
		}
		\subfigure{\begin{minipage}[b]{.3\linewidth}
				\centering
				\includegraphics[scale=0.2]{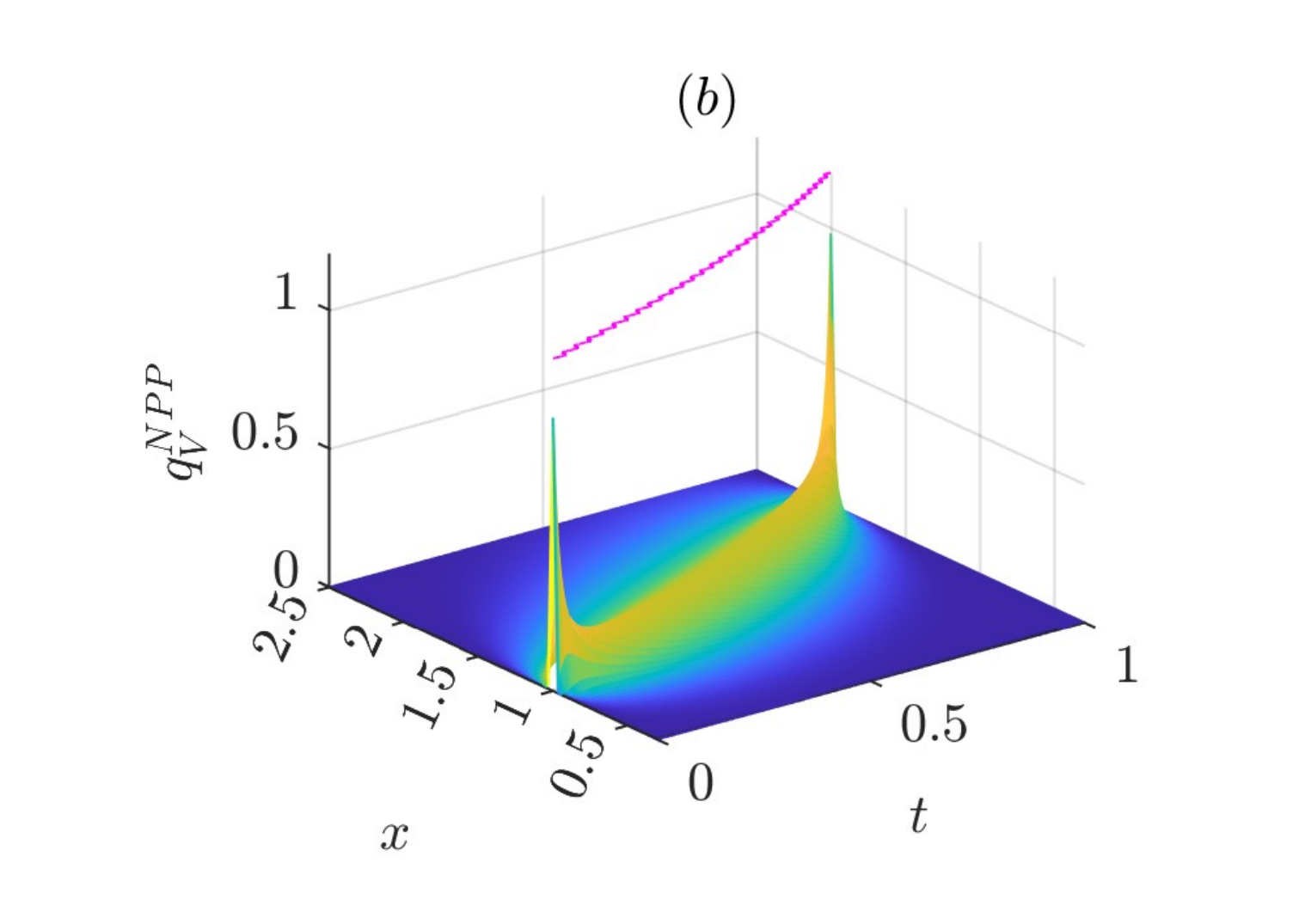}
			\end{minipage}
			\label{fig02b}
		}
		
		\subfigure{\begin{minipage}[b]{.3\linewidth}
				\centering
				\includegraphics[scale=0.2]{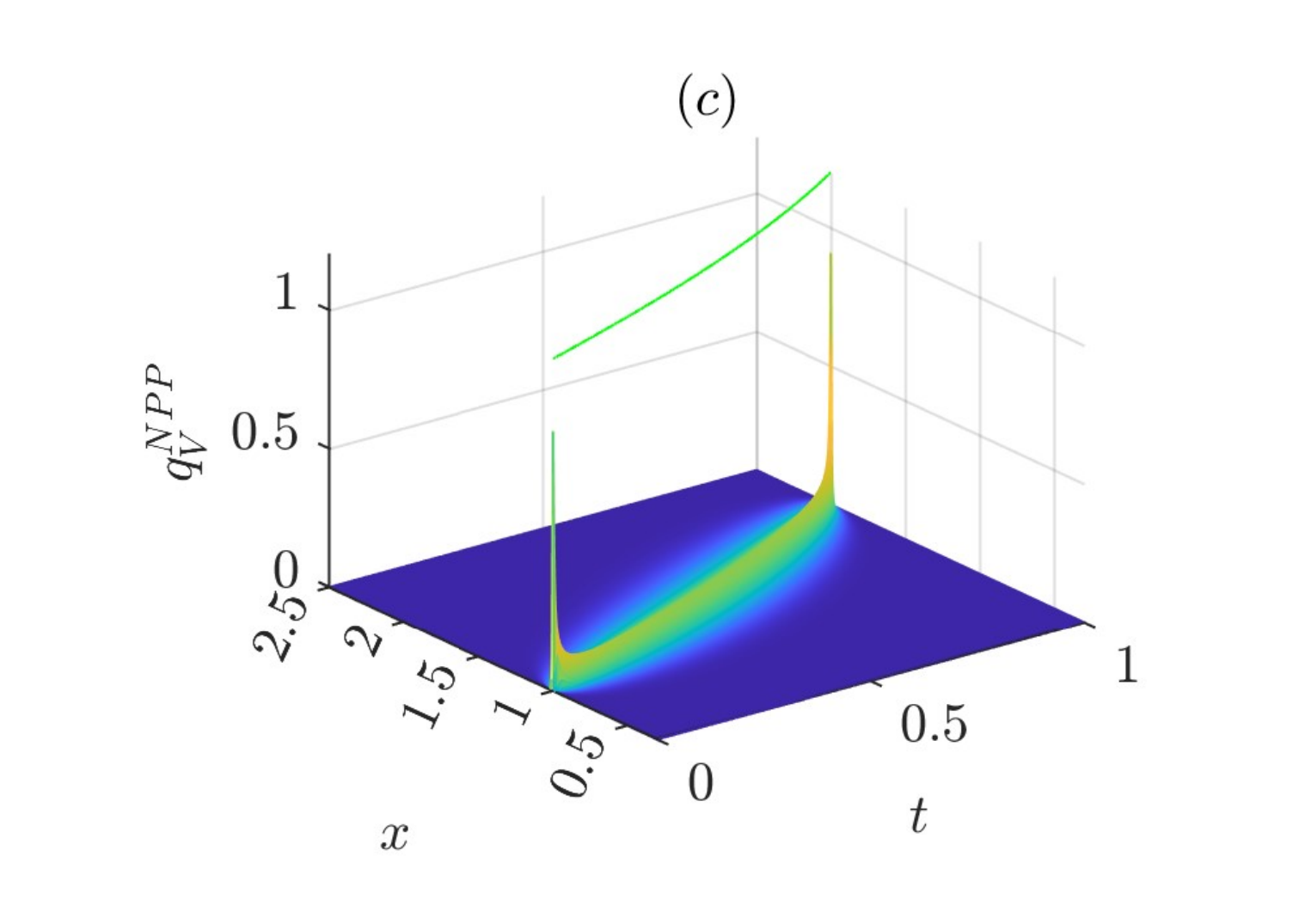}
			\end{minipage}
			\label{fig02c}
		}
		\subfigure{\begin{minipage}[b]{.3\linewidth}
				\centering
				\includegraphics[scale=0.2]{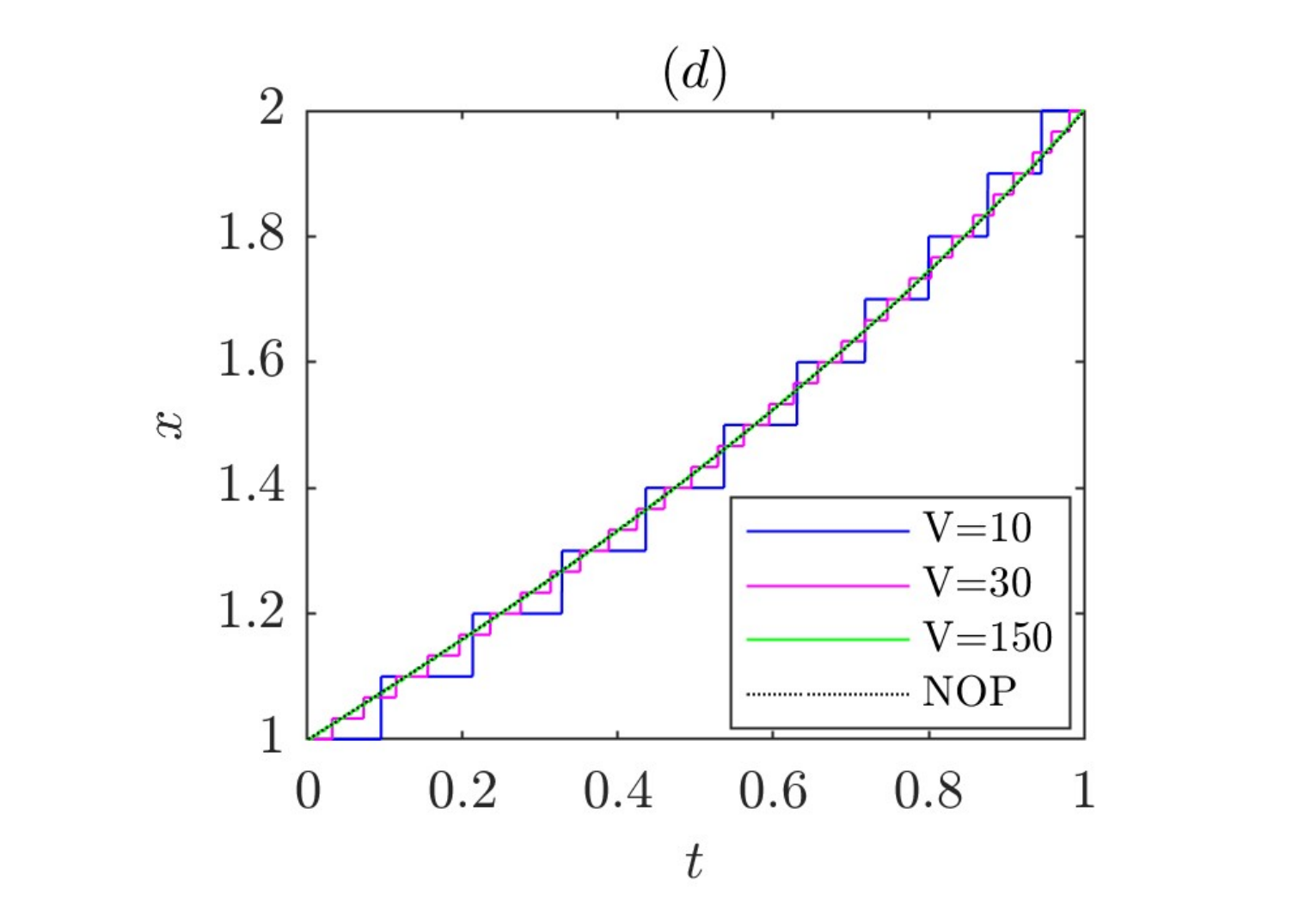}
			\end{minipage}
			\label{fig02d}
		}
		\caption{The non-stationary prehistory probabilities and their peak trajectories for (a) $V=10$, (b) $V=30$, (c) $V=150$. (d) The convergence of these peak trajectories to the NOP. The parameters are set as $x_0=1$, $x_T=2$, and $T=1$.}
		\label{fig02}
	\end{figure*}
	\begin{figure*}
		\centering
		\subfigure{\begin{minipage}[b]{.3\linewidth}
				\centering
				\includegraphics[scale=0.2]{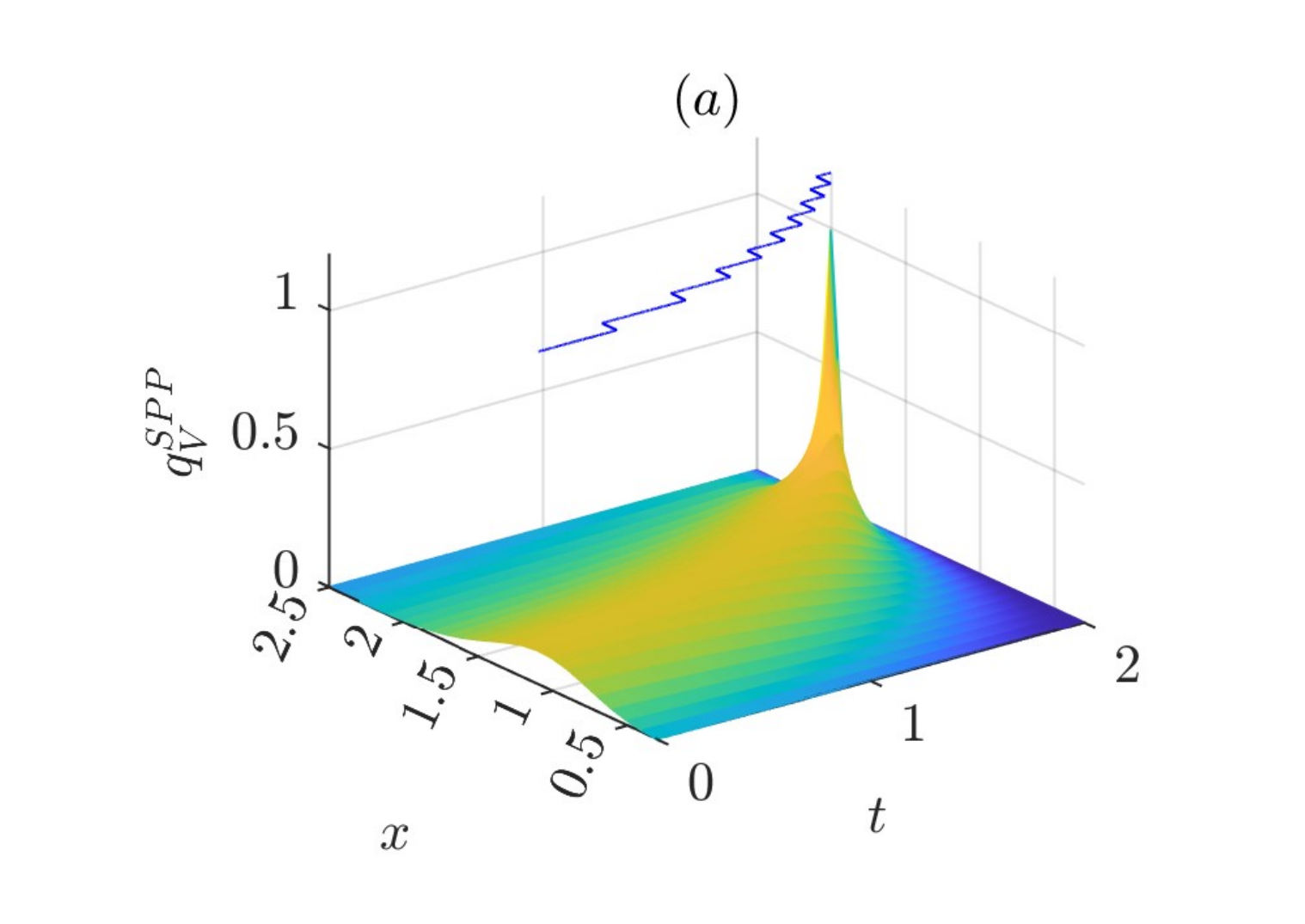}
			\end{minipage}
			\label{fig03a}
		}
		\subfigure{\begin{minipage}[b]{.3\linewidth}
				\centering
				\includegraphics[scale=0.2]{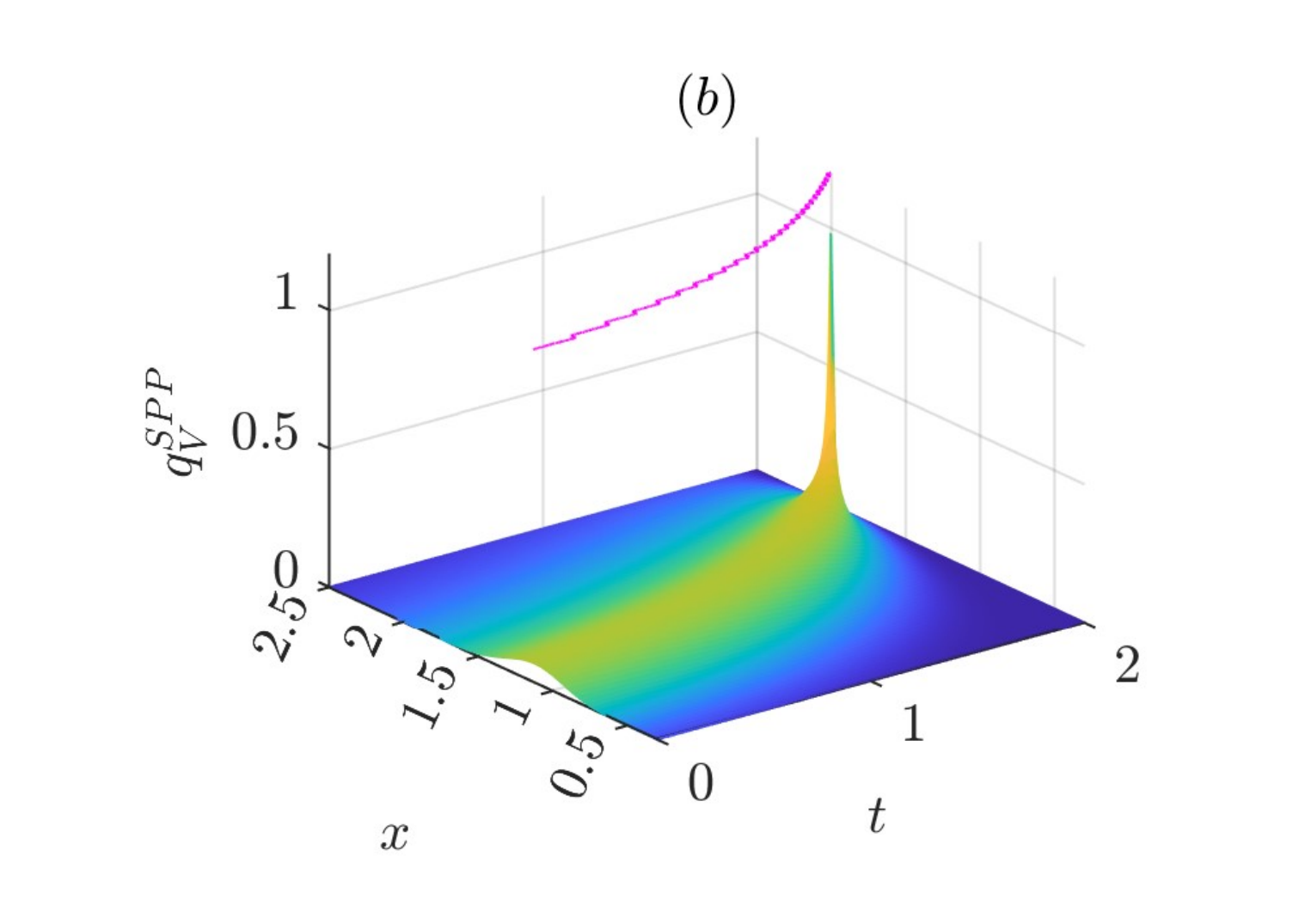}
			\end{minipage}
			\label{fig03b}
		}
		
		\subfigure{\begin{minipage}[b]{.3\linewidth}
				\centering
				\includegraphics[scale=0.2]{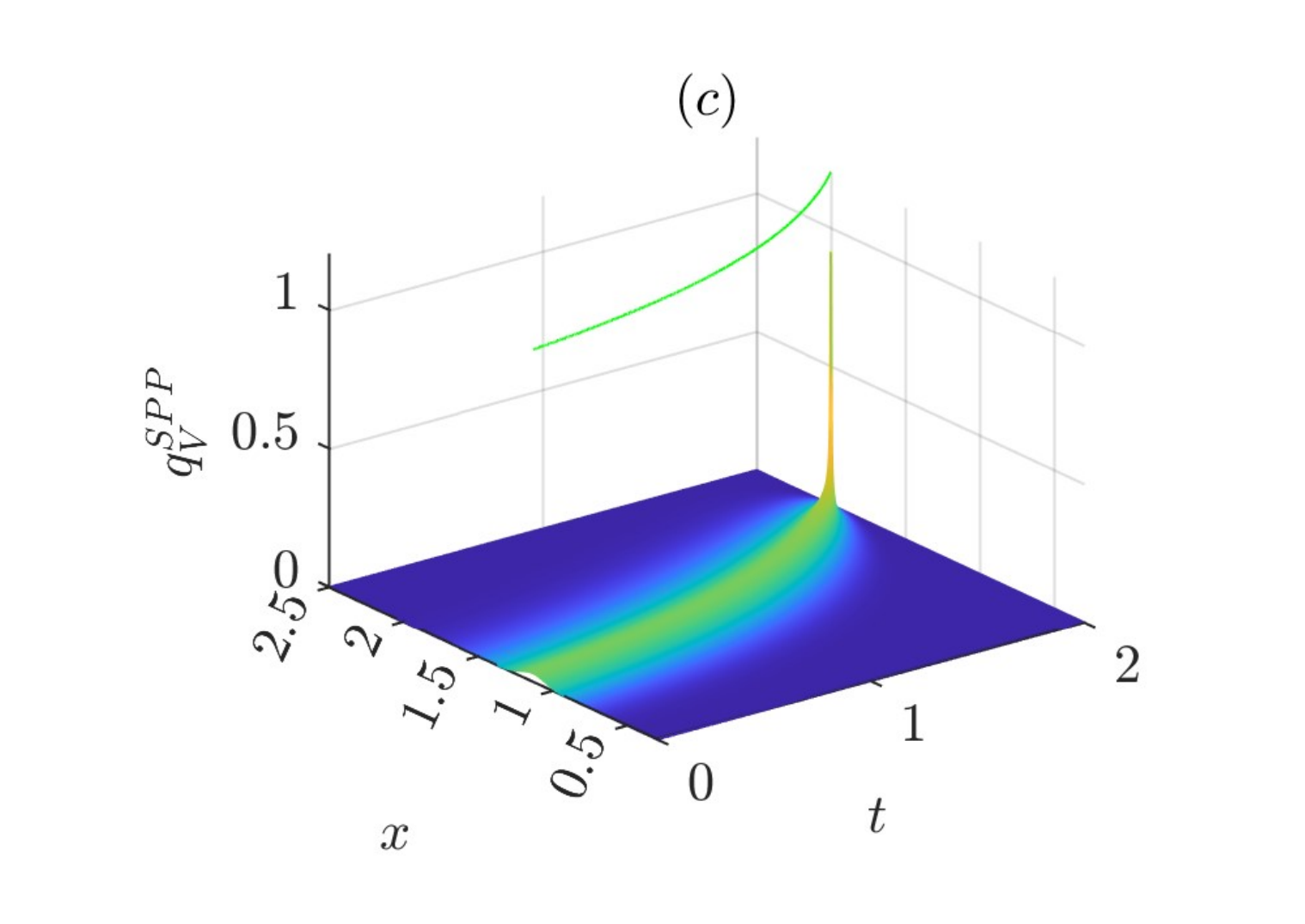}
			\end{minipage}
			\label{fig03c}
		}
		\subfigure{\begin{minipage}[b]{.3\linewidth}
				\centering
				\includegraphics[scale=0.2]{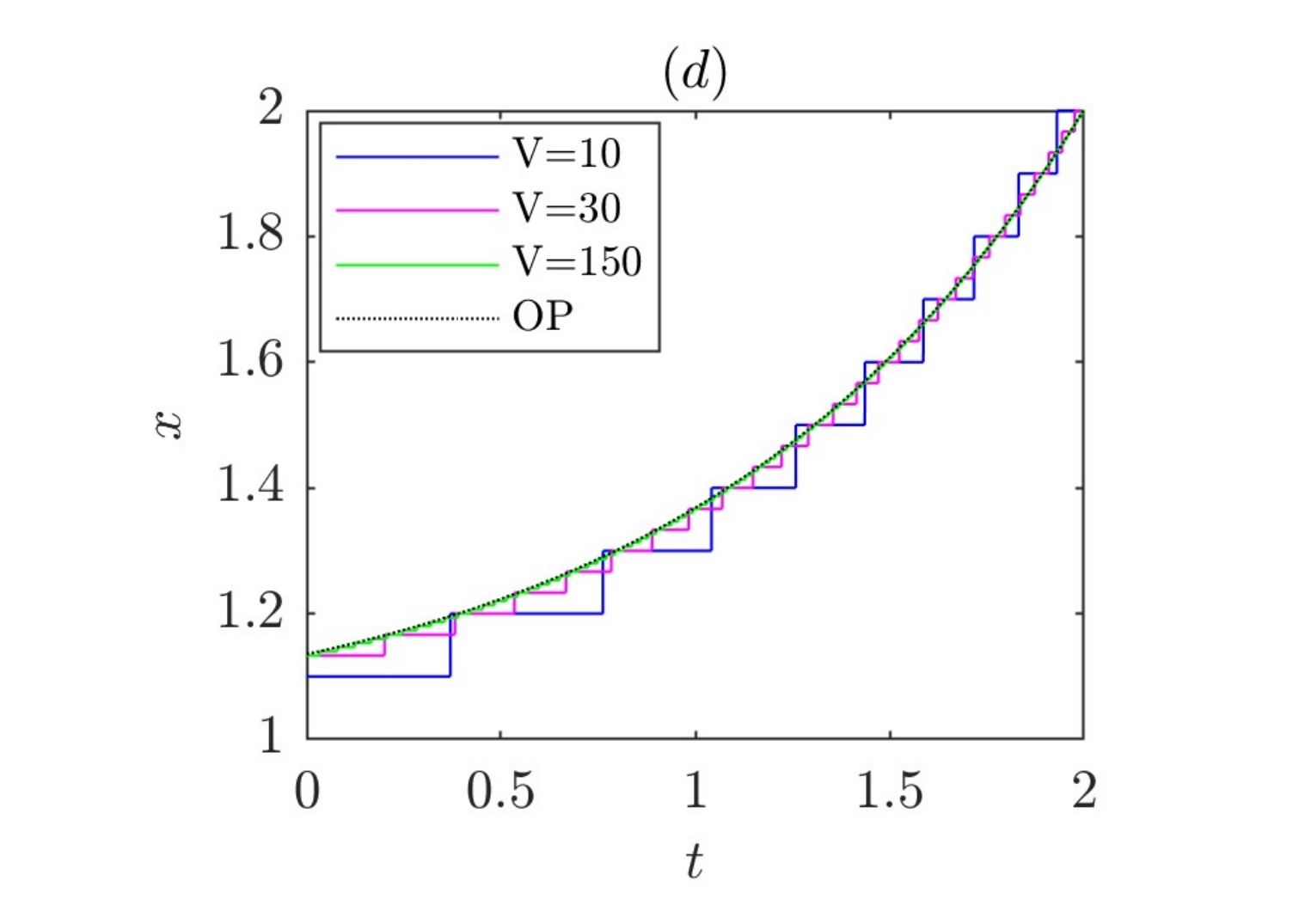}
			\end{minipage}
			\label{fig03d}
		}
		\caption{The stationary prehistory probabilities and their peak trajectories for (a) $V=10$, (b) $V=30$, (c) $V=150$. (d) The convergence of these peak trajectories to the segment $\{\phi_{\text{OP}}(t-T;x_T):t\in[0,T]\}$ of the OP. The parameters are set as $x_T=2$, and $T=2$.}
		\label{fig03}
	\end{figure*}
	
	The Hamiltonian vector field (\ref{eq020303}) is plotted in Fig. \ref{fig01a}. Let $x_0=1$, $x_T=2$. As illustrated in Fig. \ref{fig01b}, the numerical results indicate that for each $T>0$ there is a unique initial momentum $\alpha$ corresponding to a unique solution to the constrained Hamiltonian problem (\ref{eq020303}). In the case of $T=1$, the value of $\alpha$ is $0.382$. The associated solution is displayed in the phase space (Fig. \ref{fig01}) by the magenta solid line, whose projection is exactly the NOP and is exhibited in Fig. \ref{fig02d} by the black dashed line. The non-stationary prehistory probabilities and their peak trajectories for $V=10,\,30,\,150$ are illustrated in Figs. \ref{fig02a}, \ref{fig02b}, \ref{fig02c}, respectively. By comparing these results with the NOP, one can observe a clear focusing effect of the non-stationary prehistory probability on the NOP as $V\to\infty$.
	
	The stationary distribution is given by
	\begin{equation*}
		\pi_V\left(\frac{n}{V}\right)=\exp(-Vk_{+}a/k_{-})\frac{(Vk_{+}a/k_{-})^n}{n!}=e^{-V}\frac{V^n}{n!}.
	\end{equation*}
	Utilizing the Stirling's formula, it can be deduced that as $V\to\infty$,
	\begin{equation*}
		\pi_V(x)\simeq(2\pi{x}V)^{\frac{1}{2V}}e^{-VS(x)},
	\end{equation*}
	where
	\begin{equation*}
		S(x)=\int_{x_{eq}}^{x}\ln\left(\frac{R_{-}(u)}{R_{+}(u)}\right)\mathrm{d}u=x\ln(x)-x+1.
	\end{equation*}
	For each $x_T\in\mathbb{R}_{+}$, the OP connecting $x_{eq}$ with $x_T$ is the unique solution of Eq. (\ref{eq020407}), and is given by
	\begin{equation*}
		\phi_{\text{OP}}(t;x_T)=(x_T-1)e^t+1, \quad t\in(-\infty,0].
	\end{equation*}
	Let $x_T=2$ and $T=2$. The stationary prehistory probabilities and their peak trajectories for $V=10,\,30,\,150$ are presented in Fig. \ref{fig03}, indicating a clear focusing effect of the stationary prehistory probability on the segment $\{\phi_{\text{OP}}(t-T;x_T):t\in[0,T]\}$ of the OP as $V\to\infty$.
	
	\begin{figure*}
		\centering
		\subfigure{\begin{minipage}[b]{.3\linewidth}
				\centering
				\includegraphics[scale=0.2]{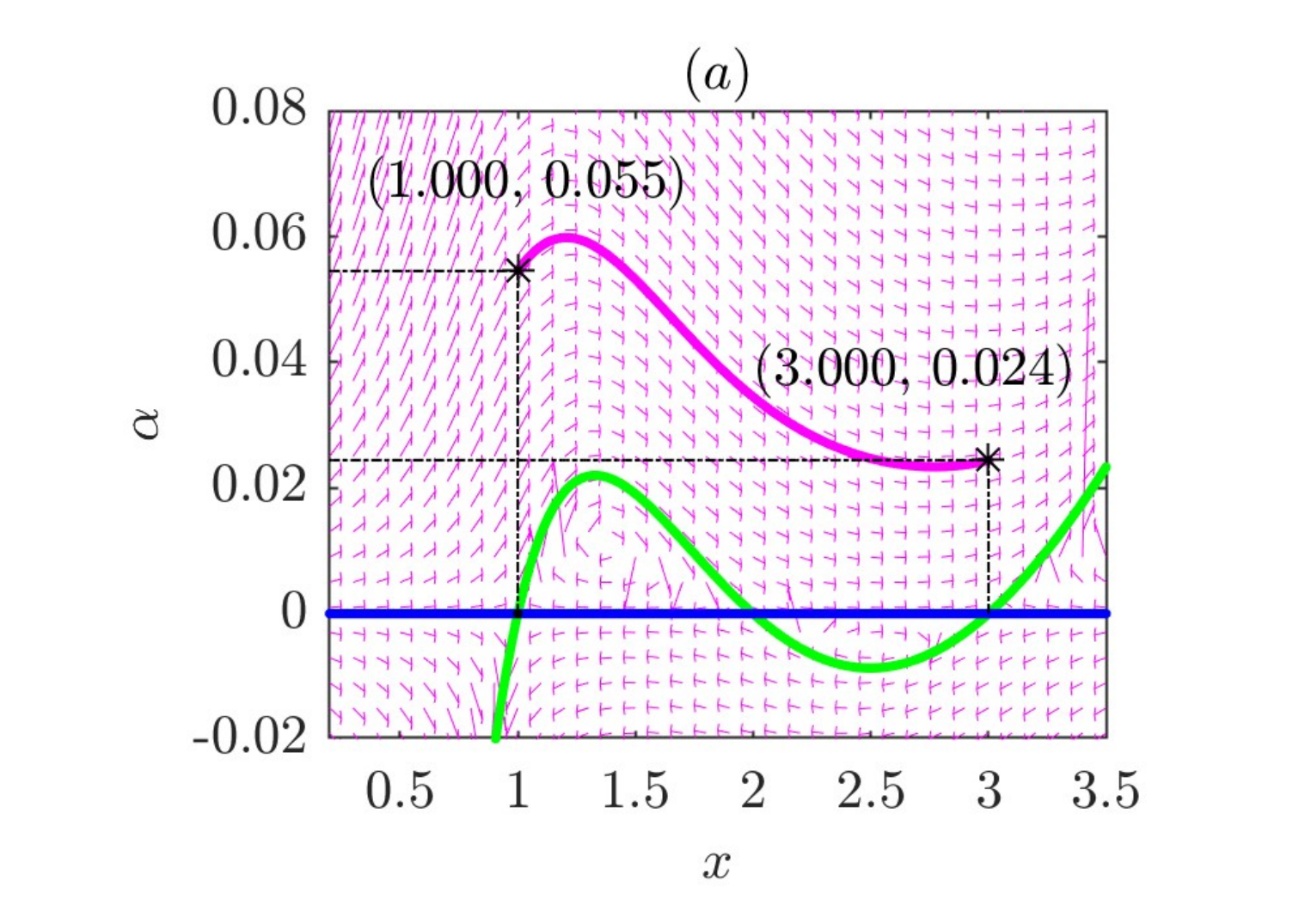}
			\end{minipage}
			\label{fig04a}
		}
		\subfigure{\begin{minipage}[b]{.3\linewidth}
				\centering
				\includegraphics[scale=0.2]{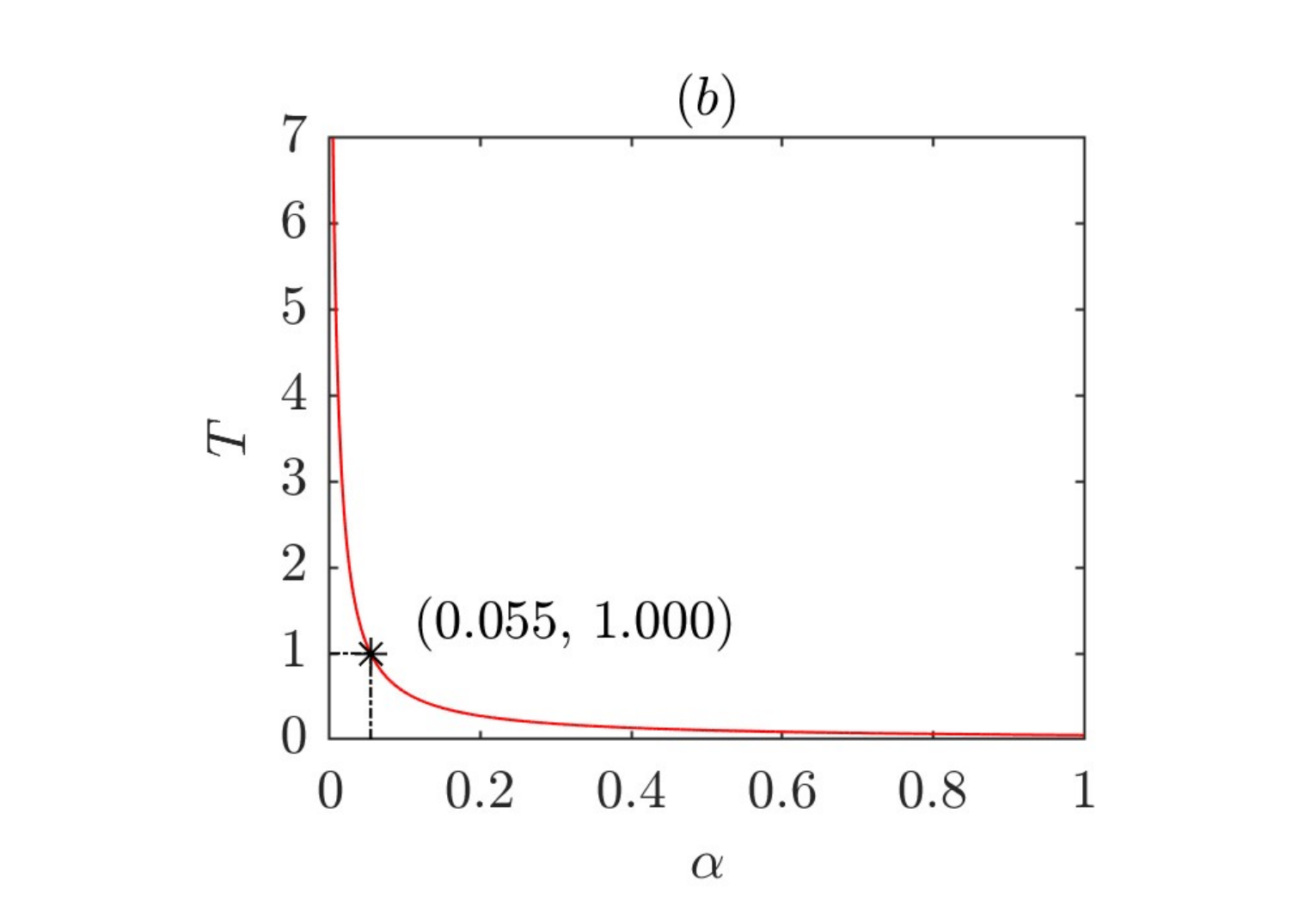}
			\end{minipage}
			\label{fig04b}
		}
		\caption{(a) The Hamiltonian vector field, the invariant (blue and green) manifolds and the unique (magenta) solution to the constrained Hamiltonian problem with $x_0=1$, $x_T=3$, $T=1$. (b) The profiles of $T(\alpha)$ versus $\alpha$ for $x_0=1$, $x_T=3$. The monotonicity indicates that each $T$ corresponds to a unique NOP.}
		\label{fig04}
	\end{figure*}
	
	\begin{figure*}
		\centering
		\subfigure{\begin{minipage}[b]{.3\linewidth}
				\centering
				\includegraphics[scale=0.2]{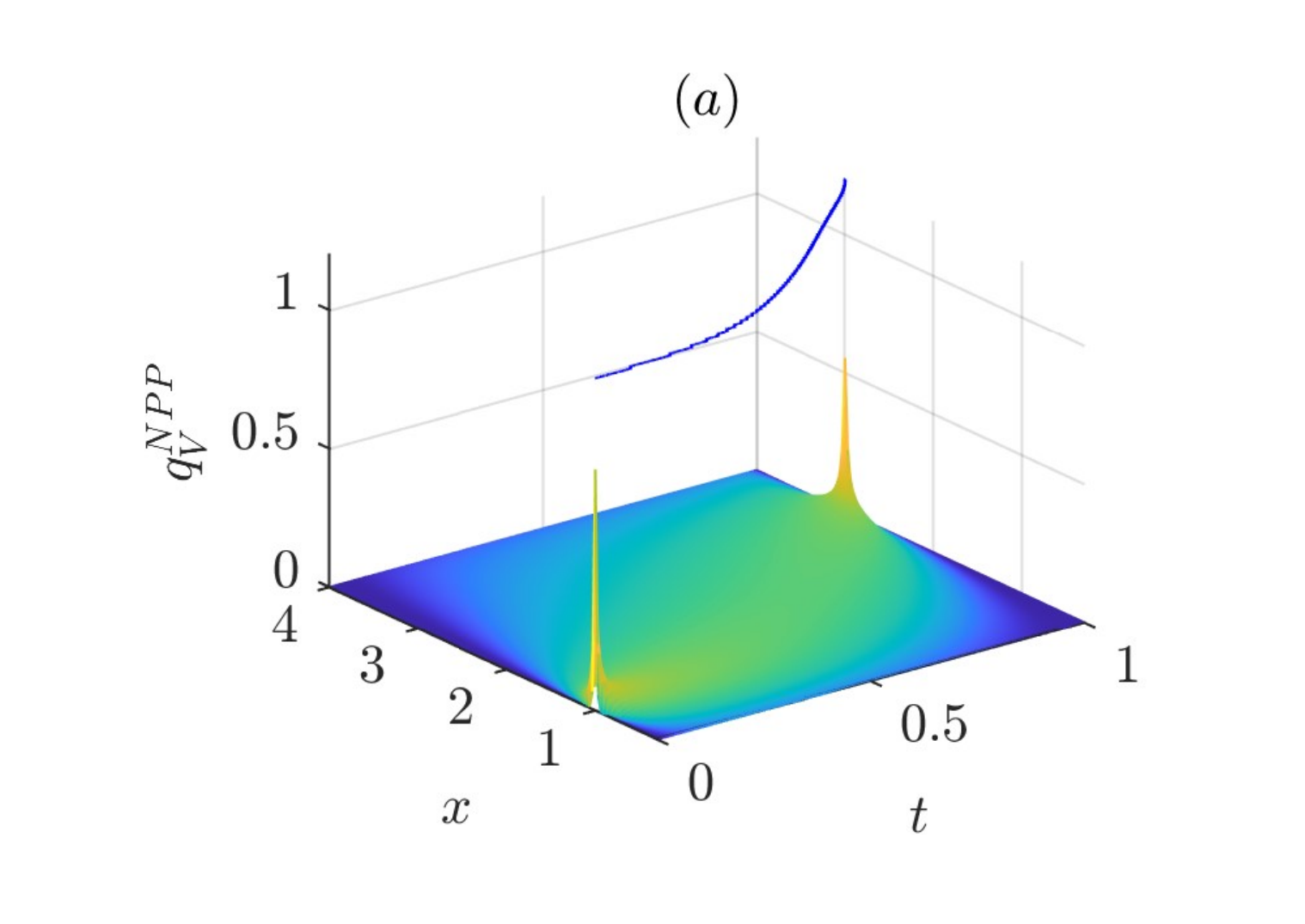}
			\end{minipage}
			\label{fig05a}
		}
		\subfigure{\begin{minipage}[b]{.3\linewidth}
				\centering
				\includegraphics[scale=0.2]{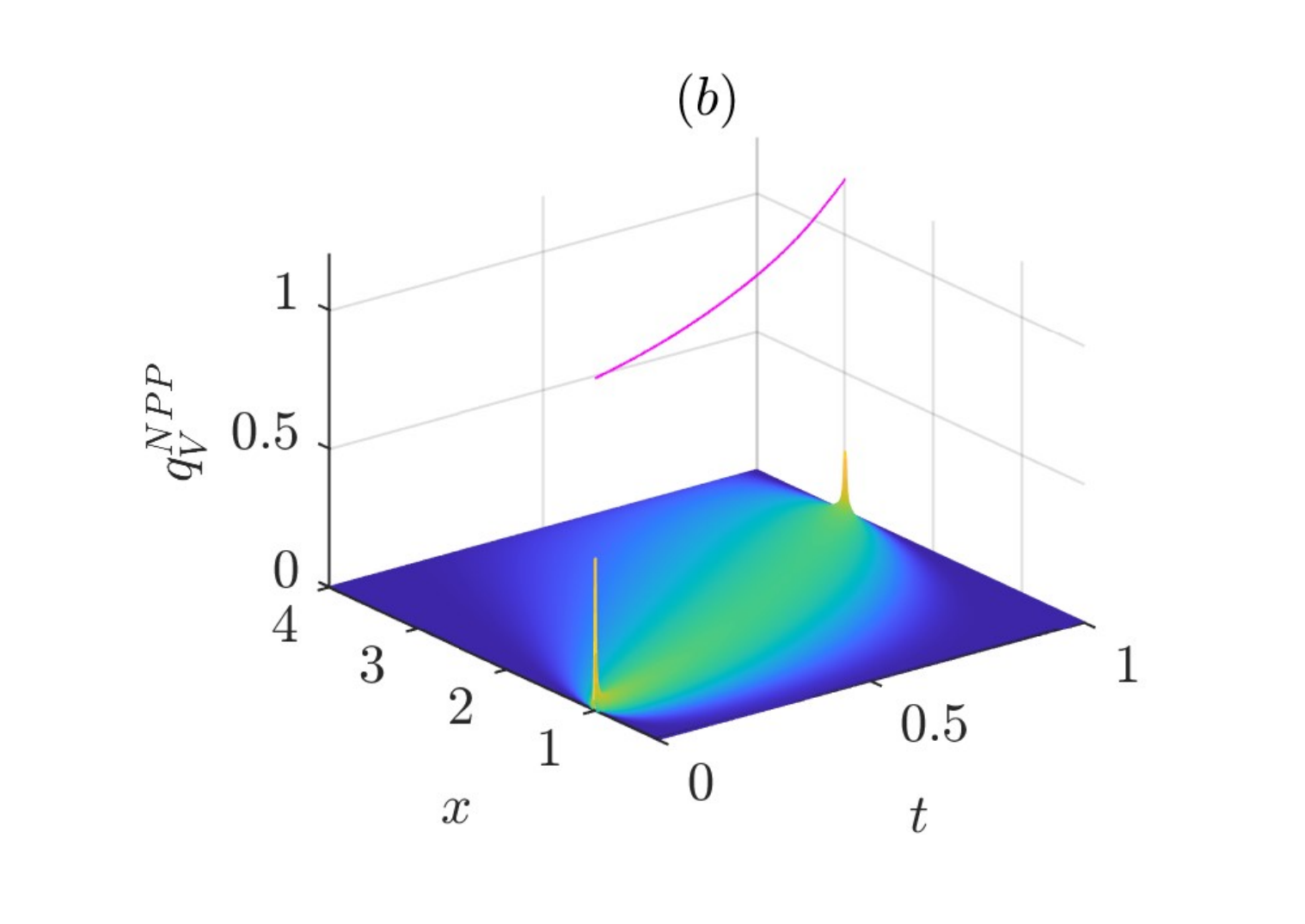}
			\end{minipage}
			\label{fig05b}
		}
		
		\subfigure{\begin{minipage}[b]{.3\linewidth}
				\centering
				\includegraphics[scale=0.2]{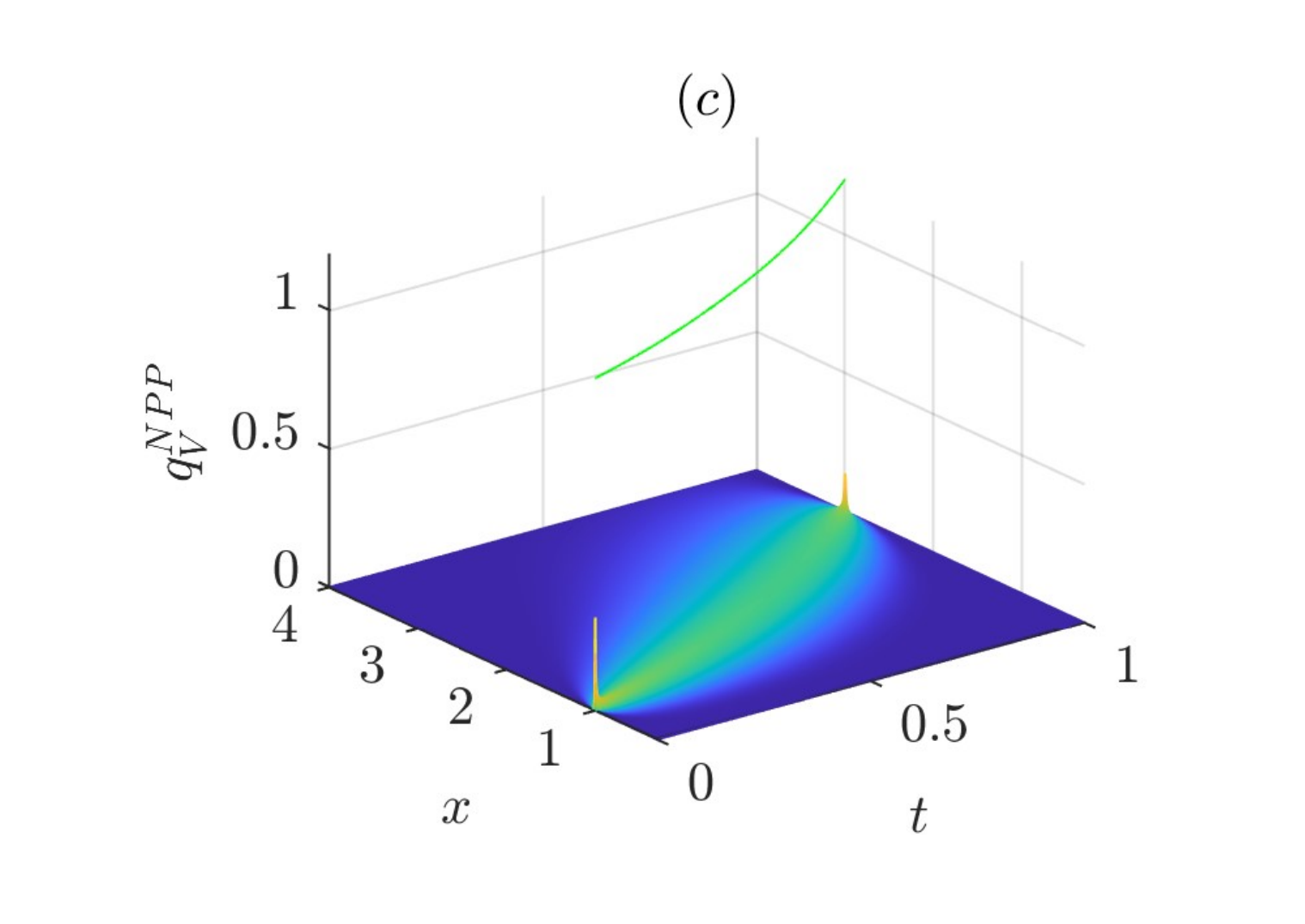}
			\end{minipage}
			\label{fig05c}
		}
		\subfigure{\begin{minipage}[b]{.3\linewidth}
				\centering
				\includegraphics[scale=0.2]{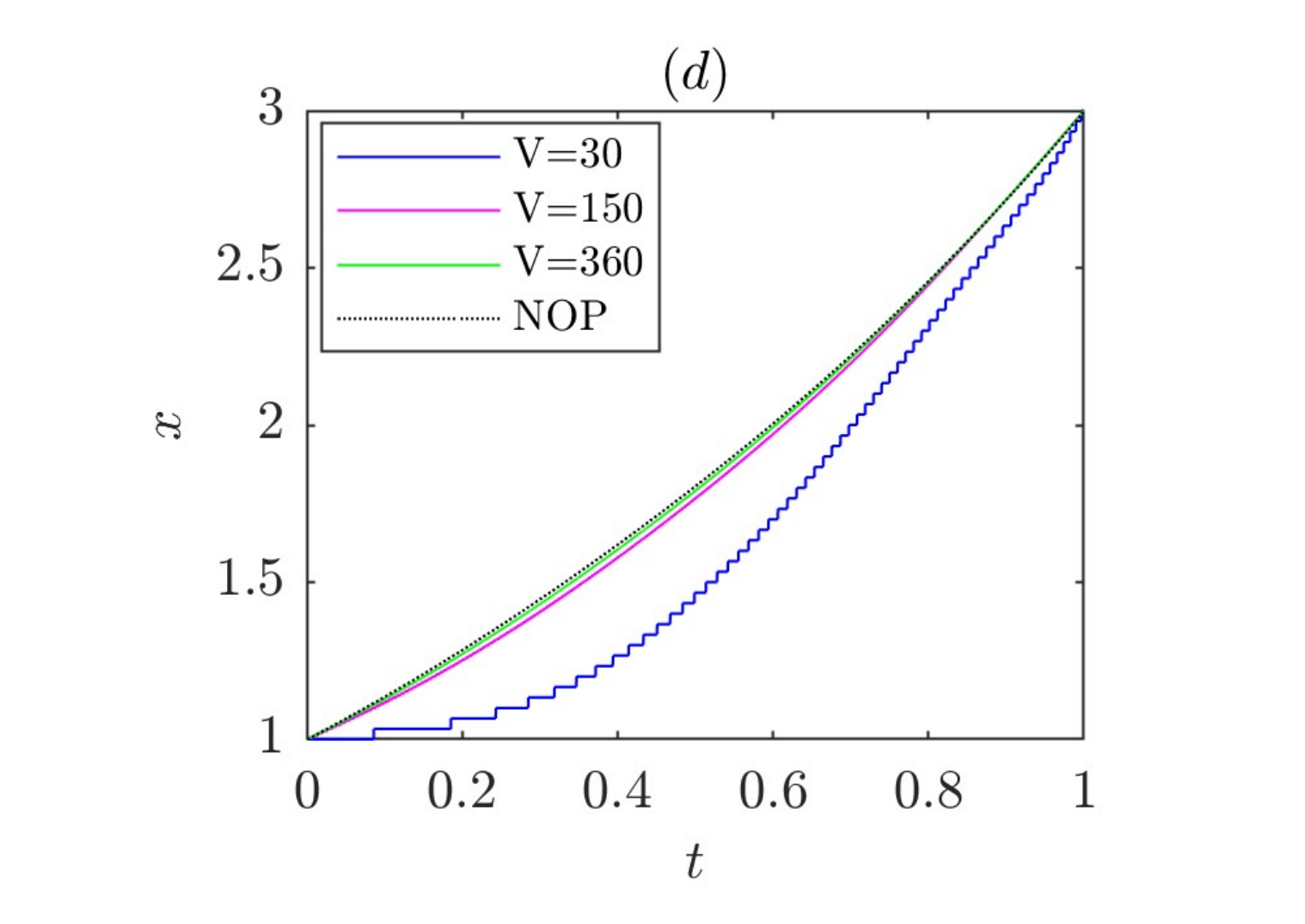}
			\end{minipage}
			\label{fig05d}
		}
		\caption{The non-stationary prehistory probabilities and their peak trajectories for (a) $V=30$, (b) $V=150$, (c) $V=360$. (d) The convergence of these peak trajectories to the NOP. The parameters are set as $x_0=1$, $x_T=3$, and $T=1$.}
		\label{fig05}
	\end{figure*}
	\subsection{A chemical Bistable System}
	We consider the system
	\begin{equation*}
			A+2S\ce{<=>[\mathit{r_{+1}}][\mathit{r_{-1}}]}3S,\quad	A\ce{<=>[\mathit{r_{+2}}][\mathit{r_{-2}}]}S.
	\end{equation*}
	The concentration of $A$ is fixed so that
	\begin{equation*}
		\begin{aligned}
			r_{+1}(n,V)=k_{+1}n_{A}n(n-1)/V^2,\quad &r_{-1}(n,V)=k_{-1}n(n-1)(n-2)/V^2,\quad \nu_1=1,\\
			r_{+2}(n,V)=k_{+2}n_{A},\quad &r_{-2}(n,V)=k_{-2}n, \quad \nu_2=1,
		\end{aligned}
	\end{equation*}
	with 
	\begin{equation*}
		\begin{aligned}
			R_{+1}(x)=k_{+1}ax^2,\quad R_{-1}(x)=k_{-1}x^3,\\
			R_{+2}(x)=k_{+2}a,\quad R_{-2}(x)=k_{-2}x.
		\end{aligned}
	\end{equation*}
	It is important to note that the sum of independent Poisson processes is also a Poisson process. In this particular instance, Eq. (\ref{eq010102}) can be rewritten as
	\begin{equation*}
		{x}_{V}(t)={x}_{V}(0)
		+V^{-1}\left(Y_{+}\left(\int_{0}^{t}{r_{+}(V{x}_{V}(s),V)\mathrm{d}s}\right)-Y_{-}\left(\int_{0}^{t}{r_{-}(V{x}_{V}(s),V)\mathrm{d}s}\right)\right),
	\end{equation*}
	with 
	\begin{equation*}
		\begin{aligned}
			r_{+}(n,V)=r_{+1}(n,V)+r_{+2}(n,V),\quad &r_{-}(n,V)=r_{-1}(n,V)+r_{-2}(n,V),\\
			R_{+}(x)=R_{+1}(x)+R_{+2}(x),\quad &R_{-}(x)=R_{-1}(x)+R_{-2}(x).
		\end{aligned}	
	\end{equation*}
	For simplicity, we assume that $k_{+1}a=6$, $k_{+2}a=6$, $k_{-1}=1$, $k_{-2}=11$. The deterministic system (\ref{eq020102}) has two stable equilibria $x^{sl}_{eq}=1$ and $x^{sr}_{eq}=3$, separated by an unstable equilibrium $x^{u}_{eq}=2$. Let $x_0=1$ and $x_T=2$. As illustrated in Fig. \ref{fig04}, each $T$ corresponds to a unique NOP. In this case, we set $T=1$. Fig. \ref{fig05} demonstrates the analogous focusing effect of the non-stationary prehistory probability on the NOP as $V\to\infty$.
	
	\section{Conclusions}\label{sec07}
	The prehistorical description of the optimal fluctuations for chemical reaction systems in both non-stationary and stationary settings was presented in this paper. Specifically, it was demonstrated that the non-stationary and stationary prehistory probabilities act as the conditional probability of a reversed process in the same setting, respectively. In the macroscopic limit, the stochastic trajectories of the reversed processes focus on a deterministic one as a consequence of the law of large numbers. This is precisely the position where the corresponding optimal path is located. Furthermore, the findings of the type of the central limit theorem show that in the vicinity of the NOP or the OP, the local deviation of the associated reversed process is approximately Gaussian. It is evident that the location of the optimal path (the NOP or the OP) and the statistical characteristics of the nearby trajectories are associated with a specified reversed process. This provides a comprehensive understanding of the optimal fluctuations.

		\begin{acknowledgments}
			The authors acknowledge support from the National Natural Science Foundation of China (Nos. 12172167, 12202195, 12302035), the Jiangsu Funding Program for Excellent Postdoctoral Talent (No. 2023ZB591), and the Natural Science Foundation of Jiangsu Province, China (No. BK20220917).
		\end{acknowledgments}
		
		\section*{Data Availability Statement}
		
		The data that support the findings of this study are available from the corresponding author upon reasonable request.

		\appendix
		\section{Proofs of Proposition \ref{theo0206} and lemma \ref{theo0207}}\label{secA1}
		\begin{proof}[Proof of Proposition \ref{theo0206}]
			Since $\boldsymbol{\nu}^\top(\mathbb{R}^M)\perp\boldsymbol{\nu}^{-1}(\boldsymbol{0})$ and $\text{Dim}(\boldsymbol{\nu}^\top(\mathbb{R}^M))+\text{Dim}(\boldsymbol{\nu}^{-1}(\boldsymbol{0}))=N$, for any $\boldsymbol{\alpha},\boldsymbol{\beta}\in\mathbb{R}^N$, there exist unique decompositions $\boldsymbol{\alpha}=\boldsymbol{\alpha}_{\text{Im}}+\boldsymbol{\alpha}_{\text{Nu}}$ and $\boldsymbol{\beta}=\boldsymbol{\beta}_{\text{Im}}+\boldsymbol{\beta}_{\text{Nu}}$ such that $\boldsymbol{\alpha}_\text{Im},\boldsymbol{\beta}_\text{Im}\in\boldsymbol{\nu}^\top(\mathbb{R}^M)$ and $\boldsymbol{\alpha}_\text{Nu},\boldsymbol{\beta}_\text{Nu}\in\boldsymbol{\nu}^{-1}(\boldsymbol{0})$. Then $L(\boldsymbol{x},\boldsymbol{\beta})$ can be rewritten as 
			\begin{equation*}
				L(\boldsymbol{x},\boldsymbol{\beta})=\sup_{\boldsymbol{\alpha}_{\text{Im}}\in\boldsymbol{\nu}^\top(\mathbb{R}^M),\,\boldsymbol{\alpha}_{\text{Nu}}\in\boldsymbol{\nu}^{-1}(\boldsymbol{0})}\left(\boldsymbol{\alpha}_{\text{Im}}\cdot\boldsymbol{\beta}_{\text{Im}}+\boldsymbol{\alpha}_{\text{Nu}}\cdot\boldsymbol{\beta}_{\text{Nu}}-H(\boldsymbol{x},\boldsymbol{\alpha}_{\text{Im}})\right).
			\end{equation*}
			That is,
			\begin{equation*}
				L(\boldsymbol{x},\boldsymbol{\beta})=\begin{cases}
					\sup_{\boldsymbol{\alpha}\in\boldsymbol{\nu}^\top(\mathbb{R}^M)}\left(\boldsymbol{\alpha}\cdot\boldsymbol{\beta}-H(\boldsymbol{x},\boldsymbol{\alpha})\right),&\text{if $\boldsymbol{\beta}\in\boldsymbol{\nu}^\top(\mathbb{R}^M)$}\\
					\infty,&\text{otherwise}.
				\end{cases}
			\end{equation*}
			It follows from [\onlinecite[Sections 5.1, 5.2]{Freidlin_2012}] (cf. also [\onlinecite[Section 5.2]{Shwartz_1995}], [\onlinecite[p. 231]{Venttsel_1977a}], or [\onlinecite{Rockafellar_1997}]) that, in the increment space $\boldsymbol{\nu}^\top(\mathbb{R}^M)$, (i) $L$ is strictly convex and smooth in the second argument; (ii) $L$ is nonnegative; (iii) $L(\boldsymbol{x},\boldsymbol{F}(\boldsymbol{x}))=0$; (iv) $L(\boldsymbol{x},\boldsymbol{\beta})/\vert\boldsymbol{\beta}\vert\to\infty$ as $\vert\boldsymbol{\beta}\vert\to\infty$, uniformly for all $\boldsymbol{x}\in\mathbb{R}^N_+$; (v) $L$ is bounded for $\boldsymbol{\beta}$ in any bounded subset of $\boldsymbol{\nu}^\top(\mathbb{R}^M)$, uniformly in $\boldsymbol{x}$. We now use these properties to complete the proof.
			
			(a) For each $T>0$, and $\boldsymbol{x}_0, \boldsymbol{x}_T\in\mathbb{R}_+^N$ satisfying $\boldsymbol{x}_T-\boldsymbol{x}_0\in\boldsymbol{\nu}^\top(\mathbb{R}^M)$, we put $\boldsymbol{\phi}(t)\triangleq\boldsymbol{x}_0\frac{T-t}{T}+\boldsymbol{x}_T\frac{t}{T}$. Let $R=\vert\frac{\boldsymbol{x}_T-\boldsymbol{x}_0}{T}\vert+1$. By the property (v) above, there exists a constant $\Gamma_1=\Gamma_1(R)$ such that $L(\boldsymbol{\phi}(t),\dot{\boldsymbol{\phi}}(t))<\Gamma_1$ for $t\in[0,T]$.
			Then $S(\boldsymbol{x}_T,T\vert\boldsymbol{x}_0)\leq I_{[0,T]}\left(\left\{\boldsymbol{\phi}(t):t\in[0,T]\right\}\right)\leq \Gamma_1T<\infty$.
			
			If $\boldsymbol{x}_T-\boldsymbol{x}_0\notin\boldsymbol{\nu}^\top(\mathbb{R}^M)$, for any $\{\boldsymbol{\phi}(t):t\in[0,T],\boldsymbol{\phi}(0)=\boldsymbol{x}_0,\boldsymbol{\phi}(T)=\boldsymbol{x}_T\}$ that is absolutely continuous, there must exist a subset $A_1$ of $[0,T]$ with positive measure such that $\dot{\boldsymbol{\phi}}(t)\notin\boldsymbol{\nu}^\top(\mathbb{R}^M)$ for $t\in{A_1}$. Therefore, $I_{[0,T]}\left(\left\{\boldsymbol{\phi}(t):t\in[0,T]\right\}\right)=\infty$, which proves part (a).
			
			(b) Denote $C_{1}\triangleq{S}(\boldsymbol{x}_T,T\vert\boldsymbol{x}_0)<\infty$. Then for each $m\geq{1}$, there exists a function $\{\boldsymbol{\phi}_{m}(t):t\in[0,T]\}$ so that $\boldsymbol{\phi}_{m}(0)=\boldsymbol{x}_0$, $\boldsymbol{\phi}_{m}(T)=\boldsymbol{x}_T$ and
			\begin{equation*}
				C_{1}\leq{I}_{[0,T]}(\{\boldsymbol{\phi}_{m}(t):t\in[0,T]\})\leq{C_{1}}+\frac{1}{m}.
			\end{equation*}
			It follows that $\{\boldsymbol{\phi}_{m}(t):t\in[0,T]\}\in\Phi_{\boldsymbol{x}_{0},[0,T]}(C_{1}+1)$ for all $m\geq{1}$. Since $\Phi_{\boldsymbol{x}_{0},[0,T]}(C_{1}+1)$ is compact in $\left(\mathcal{D}([0,T];\mathbb{R}^N), \rho\right)$, it is also sequentially compact. As a result, there exists a subsequence $\{\boldsymbol{\phi}_{m_i}(t):t\in[0,T]\}$ that converges to some function $\{\boldsymbol{\phi}^{*}(t):t\in[0,T]\}$ with $\boldsymbol{\phi}^{*}(0)=\boldsymbol{x}_0$, $\boldsymbol{\phi}^{*}(T)=\boldsymbol{x}_T$. The lower semi-continuity of $I_{[0,T]}$ implies
			\begin{equation*}
				{I}_{[0,T]}(\{\boldsymbol{\phi}^{*}(t):t\in[0,T]\})\leq\liminf_{i\to\infty}{I}_{[0,T]}(\{\boldsymbol{\phi}_{m_i}(t):t\in[0,T]\})\leq{C_{1}}.
			\end{equation*}
			Therefore, Equation (\ref{eq020301}) attains its minimum value at the function $\{\boldsymbol{\phi}^{*}(t):t\in[0,T]\}$, which is exactly a NOP connecting $\boldsymbol{x}_0$ with $\boldsymbol{x}_T$ in the time span $T$.
			
			According to the assumptions, $L(\boldsymbol{x},\boldsymbol{\beta})/\vert\boldsymbol{\beta}\vert\to\infty$ as $\vert\boldsymbol{\beta}\vert\to\infty$, uniformly for all $\boldsymbol{x}\in\mathbb{R}^N_+$ (cf. [\onlinecite[Lemmas 5.17, 5.32]{Shwartz_1995}]), i.e., for each $\gamma>0$, there exists a constant $R=R(\gamma)>0$ such that $L(\boldsymbol{x},\boldsymbol{\beta})\geq\frac{1}{\gamma}\vert\boldsymbol{\beta}\vert$ for all $\boldsymbol{x}\in\mathbb{R}^N_+,\,\boldsymbol{\beta}\in\mathbb{R}^N$ with $\vert\boldsymbol{\beta}\vert>R$. For each $T>0$ and $\boldsymbol{x}_0,\,\boldsymbol{x}_T\in\mathbb{R}^N_+$ satisfying $\boldsymbol{x}_0\neq\boldsymbol{x}_T$, let $\{\boldsymbol{\phi}^{*}(t):t\in[0,T]\}$ denote a NOP that connects $\boldsymbol{x}_0$ with $\boldsymbol{x}_T$ in the time span $T$. Define $A_2\triangleq\{t\in[0,T]:\vert\dot{\boldsymbol{\phi}^{*}}(t)\vert\leq{R}\}$ and $A_3\triangleq[0,T]\setminus{A_2}$. Then
			\begin{equation*}
				\begin{aligned}
					\left\vert\frac{\boldsymbol{x}_T-\boldsymbol{x}_0}{T}\right\vert&=\left\vert\frac{1}{T}\int_{0}^{T}\dot{\boldsymbol{\phi}^{*}}(t)\mathrm{d}t\right\vert\\
					&\leq\frac{1}{T}\int_{0}^{T}\left\vert\dot{\boldsymbol{\phi}^{*}}(t)\right\vert\mathrm{d}t\\
					&=\frac{1}{T}\left(\int_{A_2}\left\vert\dot{\boldsymbol{\phi}^{*}}(t)\right\vert\mathrm{d}t+\int_{A_3}\left\vert\dot{\boldsymbol{\phi}^{*}}(t)\right\vert\mathrm{d}t\right)\\
					&\leq\frac{1}{T}\left(\text{Measure}(A_2)R+\int_{A_3}\left\vert\dot{\boldsymbol{\phi}^{*}}(t)\right\vert\mathrm{d}t\right).
				\end{aligned}
			\end{equation*}
			It follows that
			\begin{equation*}
				\frac{1}{T}\int_{A_3}\left\vert\dot{\boldsymbol{\phi}^{*}}(t)\right\vert\mathrm{d}t\geq\frac{1}{T}\int_{0}^{T}\left\vert\dot{\boldsymbol{\phi}^{*}}(t)\right\vert\mathrm{d}t-\frac{\text{Measure}(A_2)R}{T}\geq\left\vert\frac{\boldsymbol{x}_T-\boldsymbol{x}_0}{T}\right\vert-R.
			\end{equation*}
			Consequently,
			\begin{equation*}
				\begin{aligned}
					S(\boldsymbol{x}_T,T\vert\boldsymbol{x}_0)&=I_{[0,T]}(\{\boldsymbol{\phi}^{*}(t):t\in[0,T]\})\\
					&=\int_{0}^{T}L(\boldsymbol{\phi}^{*}(t),\dot{\boldsymbol{\phi}^{*}}(t))\mathrm{d}t\\
					&=\int_{A_2}L(\boldsymbol{\phi}^{*}(t),\dot{\boldsymbol{\phi}^{*}}(t))\mathrm{d}t+\int_{A_3}L(\boldsymbol{\phi}^{*}(t),\dot{\boldsymbol{\phi}^{*}}(t))\mathrm{d}t\\
					&\geq\frac{1}{\gamma}\int_{A_3}\left\vert\dot{\boldsymbol{\phi}^{*}}(t)\right\vert\mathrm{d}t\\
					&\geq\frac{1}{\gamma}\left(\left\vert{\boldsymbol{x}_T-\boldsymbol{x}_0}\right\vert-TR\right).
				\end{aligned}
			\end{equation*}
			If we select the parameters $\boldsymbol{x}_0,\boldsymbol{x}_T$ and $T$ so that $\left\vert{\boldsymbol{x}_T-\boldsymbol{x}_0}\right\vert>TR+1$ or $T<\left\vert{\boldsymbol{x}_T-\boldsymbol{x}_0}\right\vert/2R$, part (b) of this theorem then follows.
			
			(c) Assume that $\{\boldsymbol{\phi}(t+t_1):t\in[0,t_2-t_1]\}$ is not a NOP. Let $\{\boldsymbol{\phi}^*(t):t\in[0,t_2-t_1]\}$ be a NOP connecting $\boldsymbol{\phi}(t_1)$ with $\boldsymbol{\phi}(t_2)$ in the time span $t_2-t_1$. We construct a new path connecting $\boldsymbol{x}_0$ with $\boldsymbol{x}_T$ in the time span $T$ as $\{\bar{\boldsymbol{\phi}}(t):\bar{\boldsymbol{\phi}}(t)=\boldsymbol{\phi}(t)\;\text{for}\;t\in[0,t_1]\cup[t_2,T],\bar{\boldsymbol{\phi}}(t)=\boldsymbol{\phi}^*(t-t_1)\;\text{for}\;t\in[t_1,t_2]\}$. It follows that $I_{[0,T]}\left(\left\{\bar{\boldsymbol{\phi}}(t):t\in[0,T]\right\}\right)<I_{[0,T]}\left(\left\{\boldsymbol{\phi}(t):t\in[0,T]\right\}\right)=S(\boldsymbol{x}_T,T\vert\boldsymbol{x}_0)$. This contradicts the definition of the NOP.
			
			If the NOP that links $\boldsymbol{\phi}(t_1)$ with $\boldsymbol{\phi}(t_2)$ in the time span $t_2-t_1$ is not unique, then there exists another function $\{\boldsymbol{\phi}^{*}(t):t\in[0,t_2-t_1]\}$ that satisfies $\boldsymbol{\phi}^{*}(0)=\boldsymbol{\phi}(t_1)$, $\boldsymbol{\phi}^{*}(t_2-t_1)=\boldsymbol{\phi}(t_2)$ and
			\begin{equation*}
				{I}_{[0,t_2-t_1]}(\{\boldsymbol{\phi}^{*}(t):t\in[0,t_2-t_1]\})={I}_{[0,t_2-t_1]}(\{\boldsymbol{\phi}(t+t_1):t\in[0,t_2-t_1]\}).
			\end{equation*}
			It follows that the function $\{\bar{\boldsymbol{\phi}}(t):t\in[0,T]\}$ defined above shares the same value of the action functional $I_{[0,T]}$ as that of $\{\boldsymbol{\phi}(t):t\in[0,T]\}$, which contradicts the uniqueness of $\{\boldsymbol{\phi}_{\text{NOP}}(t;\boldsymbol{x}_T,T;\boldsymbol{x}_0):t\in[0,T]\}$.
			
			(d) The $\mathcal{C}^{k+1}$-smoothness of $(\boldsymbol{x}(t),\boldsymbol{\alpha}(t))$ was proved in [\onlinecite[Theorem 1]{Day_1985}] (cf. also [\onlinecite[Section 2.6]{Cesari_1983}]), which naturally satisfies Eq. (\ref{eq020303}) as a classical result of calculus of variations (cf. [\onlinecite[Theorem 2.2.i]{Cesari_1983}]). Equations (\ref{eq020304}) and (\ref{eq020305}) follow from part (c) and the fact
			\begin{equation*}
				L(\boldsymbol{x},\boldsymbol{\beta})=\boldsymbol{\alpha}^{*}(\boldsymbol{x},\boldsymbol{\beta})\cdot\boldsymbol{\beta}-H(\boldsymbol{x},\boldsymbol{\alpha}^{*}(\boldsymbol{x},\boldsymbol{\beta})),
			\end{equation*}
			where $\boldsymbol{\alpha}^{*}=\boldsymbol{\alpha}^{*}(\boldsymbol{x},\boldsymbol{\beta})=\nabla_{\boldsymbol{\beta}}L(\boldsymbol{x},\boldsymbol{\beta})$ solves $\boldsymbol{\beta}=\nabla_{\boldsymbol{\alpha}}H(\boldsymbol{x},\boldsymbol{\alpha}^{*})$.
			
			(e) Let $\boldsymbol{x}(t;\boldsymbol{y},\boldsymbol{q}),\,\boldsymbol{\alpha}(t;\boldsymbol{y},\boldsymbol{q})$ denote the solution maps of Eq. (\ref{eq020303}) with the initial conditions $\boldsymbol{x}(0;\boldsymbol{y},\boldsymbol{q})=\boldsymbol{y},\,\boldsymbol{\alpha}(0;\boldsymbol{y},\boldsymbol{q})=\boldsymbol{q}$. (By the continuous dependence theorem of solutions of ordinary differential equations on initial values, they are $\mathcal{C}^k$ with respect to the arguments $(\boldsymbol{y},\boldsymbol{q})$ since $H(\boldsymbol{x},\boldsymbol{\alpha})$ is $\mathcal{C}^{k+1}$.) Then $\boldsymbol{\phi}_{\text{NOP}}(t;\boldsymbol{x}_T,T;\boldsymbol{x}_0)$ can be embedded in the $2N$ parameter family $\boldsymbol{x}(t;\boldsymbol{y},\boldsymbol{q})$ in the sense that
			\begin{equation*}
				\boldsymbol{\phi}_{\text{NOP}}(t;\boldsymbol{x}_T,T;\boldsymbol{x}_0)=\boldsymbol{x}(t;\boldsymbol{y},\boldsymbol{q})\vert_{(\boldsymbol{y},\boldsymbol{q})=(\boldsymbol{x}_0,\boldsymbol{\alpha}_0)},\;\;t\in[0,T],
			\end{equation*}
			with
			\begin{equation*}
				\boldsymbol{\alpha}_0\triangleq\nabla_{\boldsymbol{\beta}}L(\boldsymbol{x}_0,\dot{\boldsymbol{\phi}}(0;\boldsymbol{x}_T,T;\boldsymbol{x}_0)).
			\end{equation*}
			In addition, it is also clear that   $\{\boldsymbol{x}(t-\gamma_0;\boldsymbol{x}_0,\boldsymbol{\alpha}_0):t\in[0,T+2\gamma_0]\}$ is the NOP involved in the hypothesis.
			
			First, we claim that for each $t\in(0,T+\gamma_0)$, $\{\boldsymbol{x}(u;\boldsymbol{x}_0,\boldsymbol{\alpha}_0):u\in[0,t]\}$ is a NOP that connects $\boldsymbol{x}_0$ with $\boldsymbol{x}(t;\boldsymbol{x}_0,\boldsymbol{\alpha}_0)$ in the time span $t$, and moreover is the only one. Indeed if there were a second one, it would be possible to represent it by $\{{\boldsymbol{x}}(u;\boldsymbol{x}_0,\bar{\boldsymbol{\alpha}}_0):u\in[0,t]\}$ with some $\bar{\boldsymbol{\alpha}}_0\in\mathbb{R}^N$ subject to  ${\boldsymbol{x}}(t;\boldsymbol{x}_0,\bar{\boldsymbol{\alpha}}_0)={\boldsymbol{x}}(t;\boldsymbol{x}_0,\boldsymbol{\alpha}_0)$ but $\bar{\boldsymbol{\alpha}}_0\neq{\boldsymbol{\alpha}}_0$. Then we must have $\boldsymbol{\alpha}(t;\boldsymbol{x}_0,\bar{\boldsymbol{\alpha}}_0)\neq\boldsymbol{\alpha}(t;\boldsymbol{x}_0,\boldsymbol{\alpha}_0)$, and thus $\dot{{\boldsymbol{x}}}(t;\boldsymbol{x}_0,\bar{\boldsymbol{\alpha}}_0)\neq\dot{\boldsymbol{x}}(t;\boldsymbol{x}_0,\boldsymbol{\alpha}_0)$. Concatenating it with $\{\boldsymbol{x}(u;\boldsymbol{x}_0,\boldsymbol{\alpha}_0):u\in[t,T+\gamma_0]\}$ would produce a function achieving the minimum for $S(\boldsymbol{x}(T+\gamma_0;\boldsymbol{x}_0,\boldsymbol{\alpha}_0),T+\gamma_0\vert\boldsymbol{x}_0)$, but which would violate the $\mathcal{C}^{k+1}$-smoothness of the NOP $\{\boldsymbol{\phi}_{\text{NOP}}(u;\boldsymbol{x}(T+\gamma_0;\boldsymbol{x}_0,\boldsymbol{\alpha}_0),T+\gamma_0;\boldsymbol{x}_0):u\in[0,T+\gamma_0]\}$ because of the corner at $t$.
			
			The second thing we want to show is that for each $t\in(0,T+\gamma_0)$,
			\begin{equation*}
				\text{Det}\left(\frac{\partial}{\partial{\boldsymbol{q}}}\boldsymbol{x}(t;\boldsymbol{x}_0,\boldsymbol{q})\left\vert_{\boldsymbol{q}=\boldsymbol{\alpha}_0}\right.\right)\neq{0}.
			\end{equation*}
			Suppose it is not true. Then there exist $t\in(0,T+\gamma_0)$ and $\boldsymbol{c}\in\mathbb{R}^N$, $\boldsymbol{c}\neq\boldsymbol{0}$ so that
			\begin{equation*}
				\frac{\partial}{\partial{\boldsymbol{q}}}\boldsymbol{x}(t;\boldsymbol{x}_0,\boldsymbol{q})\left\vert_{\boldsymbol{q}=\boldsymbol{\alpha}_0}\right.\cdot\boldsymbol{c}=\boldsymbol{0}.
			\end{equation*}
			Define $\boldsymbol{\eta}(u)\triangleq\frac{\partial}{\partial{\boldsymbol{q}}}\boldsymbol{x}(u;\boldsymbol{x}_0,\boldsymbol{q})\left\vert_{\boldsymbol{q}=\boldsymbol{\alpha}_0}\right.\cdot\boldsymbol{c},\,u\in[0,t]$. It follows from [\onlinecite[Theorem 2.5.ii]{Cesari_1983}] and the fact
			\begin{equation*}
				\dot{\boldsymbol{\eta}}(0)=\frac{\partial}{\partial{\boldsymbol{q}}}\dot{\boldsymbol{x}}(0;\boldsymbol{x}_0,\boldsymbol{q})\left\vert_{\boldsymbol{q}=\boldsymbol{\alpha}_0}\right.\cdot\boldsymbol{c}=\nabla_{\boldsymbol{\alpha}}\otimes\nabla_{\boldsymbol{\alpha}}H(\boldsymbol{x}_0,\boldsymbol{\alpha}_0)\cdot\boldsymbol{c}\neq\boldsymbol{0},
			\end{equation*}
			that $\{\boldsymbol{\eta}(u):t\in[0,t]\}$ is a non-trivial solution of the accessory system relative to the NOP $\{\boldsymbol{x}(u;\boldsymbol{x}_0,\boldsymbol{\alpha}_0):u\in[0,T+\gamma_0]\}$ (cf. [\onlinecite[Section 2.5]{Cesari_1983}]). Consequently, $(t, \boldsymbol{x}(t;\boldsymbol{x}_0,\boldsymbol{\alpha}_0))$ is conjugate to $(0,\boldsymbol{x}_0)$ on the NOP, which contradicts the Jacobi's conjugate necessary condition formulated in [\onlinecite[Theorem 2.5.i]{Cesari_1983}]. 
			
			For each $T^{*}\in(0,T)$, let $\gamma_{1}=\frac{\gamma_0}{3}\wedge\frac{T-T^{*}}{3}$. Define
			\begin{equation*}
				\Omega_{\delta_{0},(T-T^*-\gamma_{1},T+\gamma_{1})}\triangleq\{(\boldsymbol{z},t)\in\mathbb{R}^N_+\times(T-T^*-\gamma_{1},T+\gamma_{1}):\boldsymbol{z}\in{B}_{\delta_0}(\boldsymbol{x}(t;\boldsymbol{x}_0,\boldsymbol{\alpha}_0))\}.
			\end{equation*}
			Note that $\boldsymbol{x}(t;\boldsymbol{x}_0,\boldsymbol{q})$ is $\mathcal{C}^k$ with respect to the argument $\boldsymbol{q}$, and the matrix  $\frac{\partial}{\partial{\boldsymbol{q}}}\boldsymbol{x}(t;\boldsymbol{x}_0,\boldsymbol{q})\left\vert_{\boldsymbol{q}=\boldsymbol{\alpha}_0}\right.$ is uniformly non-degenerate on the time interval $[T-T^*-\gamma_{1},T+\gamma_{1}]$. By the inverse function theorem, if $\delta_0$ is chosen to be small enough, then for every $(\boldsymbol{z},t)\in\Omega_{\delta_{0},(T-T^*-\gamma_{1},T+\gamma_{1})}$, there exists a unique $\mathcal{C}^{k}$ function $\boldsymbol{q}=\boldsymbol{q}(\boldsymbol{z},t)$ such that $\boldsymbol{x}(t;\boldsymbol{x}_0,\boldsymbol{q}(\boldsymbol{z},t))=\boldsymbol{z}$ and $\vert\boldsymbol{q}(\boldsymbol{z},t)-\boldsymbol{\alpha}_0\vert\to{0}$ uniformly with respect to $t\in(T-T^*-\gamma_{1},T+\gamma_{1})$ as $\vert\boldsymbol{z}-\boldsymbol{x}(t;\boldsymbol{x}_0,\boldsymbol{\alpha}_0)\vert\to{0}$. 
			
			Now, let us demonstrate that $\{\boldsymbol{x}(u;\boldsymbol{x}_0,\boldsymbol{q}(\boldsymbol{z},t)):u\in[0,t]\}$ is the unique NOP that connects $\boldsymbol{x}_0$ with $\boldsymbol{z}$ in the time span $t$. 
			
			Suppose it is not true. Then there is $T^*\in(0,T)$ so that for each $m\geq{1}$, there always exist $(\boldsymbol{z}_m,t_m)\in\Omega_{1/m,(T-T^*-\gamma_{1},T+\gamma_{1})}$ and $\bar{\boldsymbol{q}}_m$ such that (i) $\vert\bar{\boldsymbol{q}}_m-\boldsymbol{\alpha}_0\vert>\eta_0$ (with $\eta_0$ being a positive constant independent of the parameters $m$, $\boldsymbol{z}_m$ and $t_m$); (ii) $\{\boldsymbol{x}(u;\boldsymbol{x}_0,\bar{\boldsymbol{q}}_m):u\in[0,t_m]\}$ is a NOP connecting $\boldsymbol{x}_0$ with $\boldsymbol{z}_m$ in the time span $t_m$, i.e.,
			\begin{equation*}
				\boldsymbol{x}(t_m;\boldsymbol{x}_0,\bar{\boldsymbol{q}}_m)=\boldsymbol{z}_m,
			\end{equation*}
			and 
			\begin{equation*}
				\begin{aligned}
					S(\boldsymbol{z}_m,t_m\vert\boldsymbol{x}_0)&=I_{[0,t_m]}(\{\boldsymbol{x}(u;\boldsymbol{x}_0,\bar{\boldsymbol{q}}_m):u\in[0,t_m]\})\\
					&\leq{I}_{[0,t_m]}(\{\boldsymbol{x}(u;\boldsymbol{x}_0,\boldsymbol{q}(\boldsymbol{z}_m,t_m)):u\in[0,t_m]\}).
				\end{aligned}
			\end{equation*}
			Since $\Omega_{\delta_{0},(T-T^*-\gamma_{1},T+\gamma_{1})}$ is precompact in $\mathbb{R}^N_+\times\mathbb{R}$, the sequence $(\boldsymbol{z}_m,t_m)$ has a subsequence $(\boldsymbol{z}_{m_i},t_{m_i})$ that converges to some point $(\boldsymbol{z}^{*},t^{*})$ with $\boldsymbol{z}^{*}=\boldsymbol{x}(t^{*};\boldsymbol{x}_0,\boldsymbol{\alpha}_0),\,t^{*}\in[T-T^{*}-\gamma_{1},T+\gamma_{1}]$. Thus, for sufficiently large $m_i$ such that $\vert{t_{m_i}-t^{*}}\vert<\gamma_{1}$, we can define a family of functions $\{\boldsymbol{\phi}_{m_i}(u):u\in[0,t^{*}+\gamma_{1}]\}$ by
			\begin{equation*}
				\boldsymbol{\phi}_{m_i}(u)\triangleq\begin{cases}
					\boldsymbol{x}(u;\boldsymbol{x}_0,\bar{\boldsymbol{q}}_{m_i}), &u\in[0,t_{m_i}],\\
					\boldsymbol{x}_{\infty}(u-t_{m_i};\boldsymbol{z}_{m_i}), &u\in[t_{m_i},t^{*}+\gamma_{1}].
				\end{cases}
			\end{equation*}
			where $\boldsymbol{x}_{\infty}(s;\boldsymbol{z}_{m_i})$ means the solution of Eq. (\ref{eq020102}) with the initial condition $\boldsymbol{x}_{\infty}(0;\boldsymbol{z}_{m_i})=\boldsymbol{z}_{m_i}$. Clearly, $\boldsymbol{\phi}_{m_i}(0)=\boldsymbol{x}_0$ and 
			\begin{equation*}
				\begin{aligned}
					I_{[0,t^{*}+\gamma_{1}]}(\{\boldsymbol{\phi}_{m_i}(u):u\in[0,t^{*}+\gamma_{1}]\})&=I_{[0,t_{m_i}]}(\{\boldsymbol{x}(u;\boldsymbol{x}_0,\bar{\boldsymbol{q}}_{m_i}):u\in[0,t_{m_i}]\})\\
					&=S(\boldsymbol{z}_{m_i},t_{m_i}\vert\boldsymbol{x}_0).
				\end{aligned}
			\end{equation*}
			Since $(\boldsymbol{z}_{m_i},t_{m_i})\in\Omega_{1/{m_i},(T-T^*-\gamma_{1},T+\gamma_{1})}$, it is easy to check that 
			\begin{equation*}
				\left\vert\frac{\boldsymbol{z}_{m_i}-\boldsymbol{x}_0}{t_{m_i}}\right\vert\leq\frac{\max_{u\in(T-T^{*}-\gamma_{1},T+\gamma_{1})}\vert\boldsymbol{x}(u;\boldsymbol{x}_0,\boldsymbol{\alpha}_0)-\boldsymbol{x}_0\vert+1}{T-T^{*}-\gamma_{1}}.
			\end{equation*}
			Let $R$ denote the constant on the right. By part (a) of this theorem, there exists $\Gamma_1=\Gamma_1(R)$ such that 
			\begin{equation*}
				S(\boldsymbol{z}_{m_i},t_{m_i}\vert\boldsymbol{x}_0)\leq{\Gamma_1}t_{m_i}\leq{\Gamma_1}(T+\gamma_{1}).
			\end{equation*}
			Define $C_2\triangleq{\Gamma_1(T+\gamma_{1})}<\infty$. The fact $\{\boldsymbol{\phi}_{m_i}(u):u\in[0,t^{*}+\gamma_{1}]\}\in\Phi_{\boldsymbol{x}_{0},[0,t^{*}+\gamma_{1}]}(C_2)$ implies that the sequence $\{\boldsymbol{\phi}_{m_i}(u):u\in[0,t^{*}+\gamma_{1}]\}$ is uniformly bounded and equicontinuous. (The equicontinuity follows from [\onlinecite[Theorem 5.1]{Venttsel_1977b}]. The uniform boundedness is a consequence of the equicontinuity and the fact $\boldsymbol{\phi}_{m_i}(0)\equiv \boldsymbol{x}_0$.) As a result, it has a further subsequence $\{\boldsymbol{\phi}_{m_{i_j}}(u):u\in[0,t^{*}+\gamma_{1}]\}$ that converges to some function $\{\boldsymbol{\phi}^{*}(u):u\in[0,t^{*}+\gamma_{1}]\}$ with $\boldsymbol{\phi}^{*}(0)=\boldsymbol{x}_0$. Note that
			\begin{equation*}
				\vert\boldsymbol{\phi}^{*}(t^{*})-\boldsymbol{z}^{*}\vert\leq\vert\boldsymbol{\phi}^{*}(t^{*})-\boldsymbol{\phi}_{m_{i_j}}(t^{*})\vert+\vert\boldsymbol{\phi}_{m_{i_j}}(t^{*})-\boldsymbol{\phi}_{m_{i_j}}(t_{m_{i_j}})\vert+\vert\boldsymbol{z}_{m_{i_j}}-\boldsymbol{z}^{*}\vert.
			\end{equation*}
			The term on the right goes to zero due to the convergence of $\boldsymbol{\phi}_{m_{i_j}}(t^{*})$ to $\boldsymbol{\phi}^{*}(t^{*})$, the equicontinuity of $\boldsymbol{\phi}_{m_{i_j}}(u)$ and the convergence of $(\boldsymbol{z}_{m_{i_j}},t_{m_{i_j}})$ to $(\boldsymbol{z}^{*},t^{*})$. Hence, 
			\begin{equation*}
				\boldsymbol{\phi}^{*}(t^{*})=\boldsymbol{z}^{*}=\boldsymbol{x}(t^{*};\boldsymbol{x}_0,\boldsymbol{\alpha}_0).
			\end{equation*}
			By the lower semi-continuity of $I_{[0,t^{*}+\gamma_{1}]}$ and the continuity theorem of solutions of Eq. (\ref{eq020303}) on initial conditions, we have
			\begin{equation*}
				\begin{aligned}
					I_{[0,t^{*}]}(\{\boldsymbol{\phi}^{*}(u):u\in[0,t^{*}]\})&\leq{I}_{[0,t^{*}+\gamma_{1}]}(\{\boldsymbol{\phi}^{*}(u):u\in[0,t^{*}+\gamma_{1}]\})\\
					&\leq\liminf_{j\to\infty}{I}_{[0,t^{*}+\gamma_{1}]}(\{\boldsymbol{\phi}_{m_{i_j}}(u):u\in[0,t^{*}+\gamma_{1}]\})\\
					&=\liminf_{j\to\infty}{I}_{[0,t_{m_{i_j}}]}(\{\boldsymbol{x}(u;\boldsymbol{x}_0,\bar{\boldsymbol{q}}_{m_{i_j}}):u\in[0,t_{m_{i_j}}]\})\\
					&\leq\liminf_{j\to\infty}{I}_{[0,t_{m_{i_j}}]}(\{\boldsymbol{x}(u;\boldsymbol{x}_0,\boldsymbol{q}(\boldsymbol{z}_{m_{i_j}},t_{m_{i_j}})):u\in[0,t_{m_{i_j}}]\})\\
					&={I}_{[0,t^{*}]}(\{\boldsymbol{x}(u;\boldsymbol{x}_0,\boldsymbol{\alpha}_0):u\in[0,t^{*}]\}).
				\end{aligned}
			\end{equation*}
			It follows that $\{\boldsymbol{\phi}^{*}(u):u\in[0,t^{*}]\}$ is a NOP that connects $\boldsymbol{x}_0$ with $\boldsymbol{z}^{*}=\boldsymbol{x}(t^{*};\boldsymbol{x}_0,\boldsymbol{\alpha}_0)$ in the time span $t^{*}$. Consequently, we can rewritten it as 
			\begin{equation*}
				\boldsymbol{\phi}^{*}(u)=\boldsymbol{x}(u;\boldsymbol{x}_0,\bar{\boldsymbol{q}}^{*}),\;\;u\in[0,t^{*}],
			\end{equation*}
			where $\bar{\boldsymbol{q}}^{*}$ is a limit point of the sequence $\boldsymbol{q}_{m_{i_j}}$, and thus satisfies $\vert\bar{\boldsymbol{q}}^{*}-\boldsymbol{\alpha}_0\vert\geq\eta_0$. This means that $\{\boldsymbol{\phi}^{*}(u):u\in[0,t^{*}]\}$ is different from $\{\boldsymbol{x}(u;\boldsymbol{x}_0,\boldsymbol{\alpha}_0):u\in[0,t^{*}]\}$, contradicting the uniqueness of $\{\boldsymbol{\phi}_{\text{NOP}}(u;\boldsymbol{z}^{*},t^{*};\boldsymbol{x}_0):u\in[0,t^{*}]\}$.
			
			If $(\boldsymbol{z},t)\in\Omega_{\delta_{0},[T-T^*,T]}\subset\Omega_{\delta_{0},(T-T^*-\gamma_{1},T+\gamma_{1})}$, then
			\begin{equation*}
				S(\boldsymbol{z},t\vert\boldsymbol{x}_0)=\int_{0}^{t}{L(\boldsymbol{x}(u;\boldsymbol{x}_0,\boldsymbol{q}(\boldsymbol{z},t)),\dot{\boldsymbol{x}}(u;\boldsymbol{x}_0,\boldsymbol{q}(\boldsymbol{z},t)))}\mathrm{d}u.
			\end{equation*}
			Differentiating it with respect to $\boldsymbol{z}$ and $t$ respectively, and using the usual identity in [\onlinecite[Lemma 2.9.i]{Cesari_1983}], we have
			\begin{equation}\label{eqA10101}
				\frac{\partial{S}(\boldsymbol{z},t\vert\boldsymbol{x}_0)}{\partial\boldsymbol{z}}=\nabla_{\boldsymbol{\beta}}L(\boldsymbol{x}(t;\boldsymbol{x}_0,\boldsymbol{q}(\boldsymbol{z},t)),\dot{\boldsymbol{x}}(t;\boldsymbol{x}_0,\boldsymbol{q}(\boldsymbol{z},t)))=\boldsymbol{\alpha}(t;\boldsymbol{x}_0,\boldsymbol{q}(\boldsymbol{z},t)),
			\end{equation}
			and
			\begin{equation}\label{eqA10102}
				\frac{\partial{S}(\boldsymbol{z},t\vert\boldsymbol{x}_0)}{\partial{t}}=
				-H(\boldsymbol{x}(t;\boldsymbol{x}_0,\boldsymbol{q}(\boldsymbol{z},t)),\boldsymbol{\alpha}(t;\boldsymbol{x}_0,\boldsymbol{q}(\boldsymbol{z},t))).
			\end{equation}
			Obviously, both of the terms on the right are $\mathcal{C}^k$. So, $S(\boldsymbol{z},t\vert\boldsymbol{x}_0)$ is $\mathcal{C}^{k+1}$ for each $(\boldsymbol{z},t)\in\Omega_{\delta_{0},[T-T^*,T]}$. Substituting Eq. (\ref{eqA10101}) into Eq. (\ref{eqA10102}) yields
			\begin{equation*}
				\frac{\partial{S}(\boldsymbol{z},t\vert\boldsymbol{x}_0)}{\partial{t}}=-H\left(\boldsymbol{z},\frac{\partial{S}(\boldsymbol{z},t\vert\boldsymbol{x}_0)}{\partial\boldsymbol{z}}\right)
			\end{equation*}
			which proves Eq. (\ref{eq020306}). In particular, let $(\boldsymbol{z},t)=(\boldsymbol{x}(t;\boldsymbol{x}_0,\boldsymbol{\alpha}_0),t)$, it follows from Eq. (\ref{eq020303}) that for each $t\in[T-T^{*},T]$,
			\begin{equation*}
				\begin{aligned}
					\dot{\boldsymbol{x}}(t;\boldsymbol{x}_0,\boldsymbol{\alpha}_0)&=\nabla_{\boldsymbol{\alpha}}H(\boldsymbol{x}(t;\boldsymbol{x}_0,\boldsymbol{\alpha}_0),\boldsymbol{\alpha}(t;\boldsymbol{x}_0,\boldsymbol{\alpha}_0))\\
					&=\nabla_{\boldsymbol{\alpha}}H(\boldsymbol{x}(t;\boldsymbol{x}_0,\boldsymbol{\alpha}_0),\nabla_{\boldsymbol{x}}S(\boldsymbol{x}(t;\boldsymbol{x}_0,\boldsymbol{\alpha}_0),t\vert\boldsymbol{x}_0))
				\end{aligned}
			\end{equation*}
			This verifies Eq. (\ref{eq020308}).
			
			If we fix $\boldsymbol{x}_T$ and let $(\boldsymbol{x},t)$ vary in the domain $\Omega_{\delta_{0},[0,T^*]}$, the $\mathcal{C}^{k+1}$-smoothness of $S(\boldsymbol{x}_T,T-t\vert\boldsymbol{x})$, Equations (\ref{eq020307}) and (\ref{eq020309}) can be proved in a similar way. We will not repeat them here.
			
			The rest of this theorem will follow immediately if we can prove Eq. (\ref{eq020310}). Since $S(\boldsymbol{x},t\vert\boldsymbol{x}_0)$ is $\mathcal{C}^{k+1}$ in $\Omega_{\delta_0,[T-T^*,T]}$, we know that for each $(\boldsymbol{x},t)\in\Omega_{\delta_0,[T-T^*,T]}$ and sufficiently small $\varepsilon$, $D\triangleq B_{\varepsilon}(\boldsymbol{x})$ satisfies the condition required in Theorem \ref{theo0205}, and for each $\gamma>0$ there exists a constant $V_0=V_0(\boldsymbol{x},t,\boldsymbol{x}_0,\varepsilon,\gamma)$ such that once $V>V_0$,
			\begin{equation*}
				P_{\boldsymbol{x}_0}\left(\boldsymbol{x}_V(t)\in B_{\varepsilon}(\boldsymbol{x})\right)\geq\exp\left(-V\left(\inf_{\boldsymbol{y}\in B_{\varepsilon}(\boldsymbol{x})}S(\boldsymbol{y},t\vert \boldsymbol{x}_0)+\gamma\right)\right),
			\end{equation*}
			\begin{equation*}
				P_{\boldsymbol{x}_0}\left(\boldsymbol{x}_V(t)\in B_{\varepsilon}(\boldsymbol{x})\right)\leq\exp\left(-V\left(\inf_{\boldsymbol{y}\in B_{\varepsilon}(\boldsymbol{x})}S(\boldsymbol{y},t\vert \boldsymbol{x}_0)-\gamma\right)\right).
			\end{equation*}
			These inequalities are still valid if we replace $\boldsymbol{x}_0$ and $V>V_0$ by $\boldsymbol{x}^*(\boldsymbol{x}_0,V)$ and $V>V_0(\boldsymbol{x},t,\boldsymbol{x}^*(\boldsymbol{x}_0,V),\varepsilon,\gamma)$ respectively. Notice that for $(\boldsymbol{x},t)\in\Omega_{\delta_0,[T-T^*,T]}$ and sufficiently small $\varepsilon$, $S(\boldsymbol{y},T\vert \boldsymbol{x}_0)$ is also  $\mathcal{C}^{k+1}$ with respect to $\boldsymbol{x}_0$. By choosing sufficiently large $V$ such that 
			\begin{equation*}
				\left\vert\inf_{\boldsymbol{y}\in B_{\varepsilon}(\boldsymbol{x})}S(\boldsymbol{y},T\vert \boldsymbol{x}^*(\boldsymbol{x}_0,V))-\inf_{\boldsymbol{y}\in B_{\varepsilon}(\boldsymbol{x})}S(\boldsymbol{y},T\vert \boldsymbol{x}_0)\right\vert<\frac{\gamma}{2},
			\end{equation*}
			then we obtain Eq. (\ref{eq020310}).
		\end{proof}
		\begin{proof}[Proof of Lemma \ref{theo0207}]
			According to the assumptions, we know that for any $\eta>0$, there exist $V_0$ and $\varepsilon_0$ such that for $V>V_0$ and $\varepsilon<\varepsilon_0$,
			\begin{equation*}
				\left\vert\frac{k^{\varepsilon,V}(\boldsymbol{x},t\vert\boldsymbol{x}_0)}{f^{\varepsilon,V}}-K(\boldsymbol{x},t\vert\boldsymbol{x}_0)\right\vert<\eta,
			\end{equation*}
			\begin{equation*}
				\left\vert\frac{k^{\varepsilon,V}(\boldsymbol{x}\pm V^{-1}\boldsymbol{\nu}_i,t\vert\boldsymbol{x}_0)}{f^{\varepsilon,V}}-K(\boldsymbol{x}\pm V^{-1}\boldsymbol{\nu}_i,t\vert\boldsymbol{x}_0)\right\vert<\eta,
			\end{equation*}
			and
			\begin{equation*}
				\left\vert K(\boldsymbol{x}\pm V^{-1}\boldsymbol{\nu}_i,t\vert\boldsymbol{x}_0)-K(\boldsymbol{x},t\vert\boldsymbol{x}_0)\right\vert<\eta.
			\end{equation*}
			Putting all these facts together, we get
			\begin{equation*}
				\lim_{V\to\infty,\,\varepsilon\to 0}\frac{k^{\varepsilon,V}(\boldsymbol{x}\pm V^{-1}\boldsymbol{\nu}_i,t\vert\boldsymbol{x}_0)}{k^{\varepsilon,V}(\boldsymbol{x},t\vert\boldsymbol{x}_0)}=1.
			\end{equation*}
			
			Note that 
			\begin{equation*}
				P_{\boldsymbol{x}^*(\boldsymbol{x}_0,V)}\left(\boldsymbol{x}_V(t)\in B_{\bar{\varepsilon}(V)}(\boldsymbol{x})\right)=k^{\bar{\varepsilon}(V),V}(\boldsymbol{x},t\vert\boldsymbol{x}_0)\exp\left(-V\inf_{\boldsymbol{y}\in B_{\bar{\varepsilon}(V)}(\boldsymbol{x})}S(\boldsymbol{y},t\vert \boldsymbol{x}_0)\right),
			\end{equation*}
			\begin{equation*}
				\begin{aligned}
					P_{\boldsymbol{x}^*(\boldsymbol{x}_0,V)}&\left(\boldsymbol{x}_V(t)\in B_{\bar{\varepsilon}(V)}(\boldsymbol{x}\pm V^{-1}\boldsymbol{\nu}_i)\right)\\
					=&k^{\bar{\varepsilon}(V),V}(\boldsymbol{x}\pm V^{-1}\boldsymbol{\nu}_i,t\vert\boldsymbol{x}_0)\exp\left(-V\inf_{\boldsymbol{y}\in B_{\bar{\varepsilon}(V)}(\boldsymbol{x}\pm V^{-1}\boldsymbol{\nu}_i))}S(\boldsymbol{y},t\vert \boldsymbol{x}_0)\right).
				\end{aligned}
			\end{equation*}
			Therefore, in order to show Eq. (\ref{eq020311}), it suffices to show that
			\begin{equation*}
				\lim_{V\to\infty}V\left(\inf_{\boldsymbol{y}\in B_{\bar{\varepsilon}(V)}(\boldsymbol{x})}S(\boldsymbol{y},t\vert \boldsymbol{x}_0)-\inf_{\boldsymbol{y}\in B_{\bar{\varepsilon}(V)}(\boldsymbol{x}\pm V^{-1}\boldsymbol{\nu}_i)}S(\boldsymbol{y},t\vert \boldsymbol{x}_0)\right)=\mp\boldsymbol{\nu}_i\cdot\nabla_{\boldsymbol{x}}S(\boldsymbol{x},t\vert \boldsymbol{x}_0).
			\end{equation*}
			
			Since $S(\boldsymbol{x},t\vert \boldsymbol{x}_0)$ is $\mathcal{C}^{2}$ in $\Sigma_{\delta_0,[T-T^*,T]}\times[T-T^*,T]$, we have
			\begin{equation*}
				\inf_{\boldsymbol{y}\in B_{\bar{\varepsilon}(V)}(\boldsymbol{x})}S(\boldsymbol{y},t\vert \boldsymbol{x}_0)-S(\boldsymbol{x},t\vert \boldsymbol{x}_0)=-\left\vert\nabla_{\boldsymbol{x}}S(\boldsymbol{x},t\vert \boldsymbol{x}_0)\right\vert\bar{\varepsilon}(V)+{O}(\bar{\varepsilon}^2(V)),
			\end{equation*}
			\begin{equation*}
				\inf_{\boldsymbol{y}\in B_{\bar{\varepsilon}(V)}(\boldsymbol{x}\pm V^{-1}\boldsymbol{\nu}_i))}S(\boldsymbol{y},t\vert \boldsymbol{x}_0)-S(\boldsymbol{x}\pm V^{-1}\boldsymbol{\nu}_i,t\vert \boldsymbol{x}_0)=-\left\vert\nabla_{\boldsymbol{x}}S(\boldsymbol{x}\pm V^{-1}\boldsymbol{\nu}_i,t\vert \boldsymbol{x}_0)\right\vert\bar{\varepsilon}(V)+{O}(\bar{\varepsilon}^2(V)),
			\end{equation*}
			and 
			\begin{equation*}
				\lim_{V\to\infty}V\left(S(\boldsymbol{x},t\vert \boldsymbol{x}_0)-S(\boldsymbol{x}\pm V^{-1}\boldsymbol{\nu}_i,t\vert \boldsymbol{x}_0)\right)=\mp\boldsymbol{\nu}_i\cdot\nabla_{\boldsymbol{x}}S(\boldsymbol{x},t\vert \boldsymbol{x}_0).
			\end{equation*}
			Regrouping the terms, we obtain Eq. (\ref{eq020311}). The convergence is uniform in $\Sigma_{\delta_0,[T-T^*,T]}\times[T-T^*,T]$ since the assumptions hold uniformly in the same domain.
			
			Notice that $\vert\boldsymbol{x}^{**}(\boldsymbol{x},V)-\boldsymbol{x}\vert\leq{O}(1/V)$. If we choose an appropriate parameter $\bar{\varepsilon}(V)$ so that there is a unique point $\boldsymbol{x}^{**}(\boldsymbol{x},V)$ in $B_{\bar{\varepsilon}(V)}(\boldsymbol{x})$, then 
			\begin{equation*}
				P_{\boldsymbol{x}^*(\boldsymbol{x}_0,V)}\left(\boldsymbol{x}_V(t)\in B_{\bar{\varepsilon}(V)}(\boldsymbol{x})\right)=p_V(\boldsymbol{x}^{**}(\boldsymbol{x},V),t\vert\boldsymbol{x}^*(\boldsymbol{x}_0,V)),
			\end{equation*}
			and 
			\begin{equation*}
				P_{\boldsymbol{x}^*(\boldsymbol{x}_0,V)}\left(\boldsymbol{x}_V(t)\in B_{\bar{\varepsilon}(V)}(\boldsymbol{x}\pm V^{-1}\boldsymbol{\nu}_i)\right)=p_V(\boldsymbol{x}^{**}(\boldsymbol{x},V)\pm V^{-1}\boldsymbol{\nu}_i,t\vert\boldsymbol{x}^*(\boldsymbol{x}_0,V)),
			\end{equation*}
			which implies Eq. (\ref{eq020312}).
		\end{proof} 
		
		\section{Proof of Proposition \ref{theo0209}}\label{secA2}
		\begin{proof}
			(a) Let $\boldsymbol{x}_0=\boldsymbol{x}_{eq}$, $\boldsymbol{x}_T=\boldsymbol{x}$. Then for any $T>0$, $S(\boldsymbol{x})\leq S(\boldsymbol{x}_T,T\vert\boldsymbol{x}_0)<\infty$.
			
			(b) For any $\boldsymbol{x},\,\boldsymbol{y}\in\mathbb{R}^N_{+}$, we set $T=\vert\boldsymbol{x}-\boldsymbol{y}\vert$ and $\boldsymbol{\phi}(t)=\boldsymbol{x}+\frac{\boldsymbol{y}-\boldsymbol{x}}{\vert\boldsymbol{y}-\boldsymbol{x}\vert}t$. Since $L$ is bounded for $\boldsymbol{\beta}$ in any bounded subset of $\boldsymbol{\nu}^\top(\mathbb{R}^M)$, uniformly in $\boldsymbol{x}$, there exists a constant $C_3$ such that $L(\boldsymbol{\phi}(t),\dot{\boldsymbol{\phi}}(t))<C_3$ for $t\in[0,T]$. Hence, $S(\boldsymbol{y})-S(\boldsymbol{x})\leq{I}_{[0,T]}(\{\boldsymbol{\phi}(t):t\in[0,T]\})\leq C_3T=C_3\vert\boldsymbol{x}-\boldsymbol{y}\vert$. Swapping $\boldsymbol{x}$ and $\boldsymbol{y}$, we finally prove part (b).
			
			(c) Suppose $\left\{\boldsymbol{\phi}(t):t\in[T_1,T^*]\right\}$ is not an (the unique) OP connecting $\boldsymbol{x}_{eq}$ with $\boldsymbol{\phi}(T^*)$. Then there must be another OP connecting $\boldsymbol{x}_{eq}$ with $\boldsymbol{\phi}(T^*)$, so that its value of the rate function is no more than that of $\left\{\boldsymbol{\phi}(t):t\in[T_1,T^*]\right\}$. Concatenating it with $\left\{\boldsymbol{\phi}(t):t\in[T^*,T_2]\right\}$ would produce a new path achieving the minimum for $S(\boldsymbol{x})$, but which would violate the fact that $\left\{\boldsymbol{\phi}(t):t\in[T_1,T_2],\boldsymbol{\phi}(T_1)=\boldsymbol{x}_{eq},\boldsymbol{\phi}(T_2)=\boldsymbol{x}\right\}$ ($-\infty\leq T_1<T_2\leq\infty$) is an (the unique) OP linking $\boldsymbol{x}_{eq}$ and $\boldsymbol{x}$.
			
			Suppose $\left\{\boldsymbol{\phi}(t+t_1):t\in[0,t_2-t_1]\right\}$ is not a (the unique) NOP that connects $\boldsymbol{\phi}(t_1)$ with $\boldsymbol{\phi}(t_2)$ in all possible time span. Then there must exist another NOP connecting $\boldsymbol{\phi}(t_1)$ with $\boldsymbol{\phi}(t_2)$ in some time span, so that its value of the rate function is no more than that of $\left\{\boldsymbol{\phi}(t+t_1):t\in[0,t_2-t_1]\right\}$. Concatenating it with 
			$\left\{\boldsymbol{\phi}(t):t\in[T_1,t_1]\right\}$ and $\left\{\boldsymbol{\phi}(t):t\in[t_2,T_2]\right\}$ would produce a new path achieving the minimum for $S(\boldsymbol{x})$, but which would violate the fact that $\{\boldsymbol{\phi}(t):t\in[T_1,T_2],\boldsymbol{\phi}(T_1)=\boldsymbol{x}_{eq},\boldsymbol{\phi}(T_2)=\boldsymbol{x}\}$ ($-\infty\leq T_1<T_2\leq\infty$) is an (the unique) OP linking $\boldsymbol{x}_{eq}$ and $\boldsymbol{x}$.
			
			(d) The existence of the optimal path follows from [\onlinecite[Corollary 1]{Day_1985}]. The OP can be represented by $\left\{\boldsymbol{\phi}_{\text{OP}}(t;\boldsymbol{x}):t\in(-\infty,0],\lim_{t\to-\infty}\boldsymbol{\phi}_{\text{OP}}(t;\boldsymbol{x})=\boldsymbol{x}_{eq},\boldsymbol{\phi}_{\text{OP}}(0;\boldsymbol{x})=\boldsymbol{x}\right\}$ as a result of [\onlinecite[Lemma 2]{Day_1985}].
			
			(e) The $\mathcal{C}^{k+1}$-smoothness of $(\boldsymbol{x}(t),\boldsymbol{\alpha}(t))$ and Eq. (\ref{eq020404}) were proved in [\onlinecite[Theorem 1]{Day_1985}]. Eq. (\ref{eq020405}) can be derived from [\onlinecite[Corollary 3]{Day_1985}].
			
			(f) The $\mathcal{C}^{k+1}$-smoothness of $S(\boldsymbol{x})$ can be found in [\onlinecite[Theorems 2; 6]{Day_1985}]. Referring to [\onlinecite[Corollary 5]{Day_1985}], it can be seen that $S(\boldsymbol{x})$ satisfies the stationary Hamilton-Jacobi equation (Eq. (\ref{eq020406})). In this case, the OP is naturally the unique solution of Eq. (\ref{eq020407}) as a result of part (e) and the $\mathcal{C}^{k+1}$-smoothness of $S(\boldsymbol{x})$.
		\end{proof}
		\section{Proof of Proposition \ref{theo0401}}\label{secA3}
		\begin{lemma}\label{theoA301}
			Assume the conditions of Proposition \ref{theo0401} hold. Let $\tau\triangleq\inf\{t>0:\hat{\boldsymbol{x}}_{V}(t)\notin\hat{\Sigma}_{\delta_0,[0,T^*]}\}$ be the first time for $\hat{\boldsymbol{x}}_{V}(t)$ to escape from $\hat{\Sigma}_{\delta_0,[0,T^*]}$. Then in order to show that
			\begin{equation*}
				\lim_{V\to\infty}P\left\{\sup_{t\in[0,T^*]}\left\vert\hat{\boldsymbol{x}}_{V}(t\wedge\tau)-\hat{\boldsymbol{x}}_{\infty}(t\wedge\tau)\right\vert>\delta\right\}=0,
			\end{equation*}
			for every sufficiently small $\delta>0$, it suffices to show that
			\begin{equation*}
				\lim_{V\to\infty}P\left\{\sup_{t\in[0,T^*]}\left\vert\hat{\boldsymbol{x}}_{V}(t\wedge\tau)-\boldsymbol{x}^{**}(\boldsymbol{x}_T,V)-\int_0^{t\wedge\tau}\boldsymbol{G}_V(\hat{\boldsymbol{x}}_{V}(s),s)\mathrm{d}s\right\vert>\delta\right\}=0,
			\end{equation*}
			for every sufficiently small $\delta>0$.
		\end{lemma}
		\begin{proof}
			Since $\hat{\boldsymbol{x}}_{V}(t)$ is a pure jump process with right continuous piecewise constant trajectories, we know that $\hat{\boldsymbol{x}}_{V}(\tau)\notin\hat{\Sigma}_{\delta_0,[0,T^*]}$. Let $\check{\boldsymbol{x}}_{V}(t)=\hat{\boldsymbol{x}}_{V}(t\wedge\tau)$. Then $\check{\boldsymbol{x}}_{V}(t)$ is a pure jump Markov process with the rate functions
			\begin{equation*}
				\check{r}_{\pm i}(V\boldsymbol{x},V,t)=\begin{cases}
					\hat{r}_{\pm i}(V\boldsymbol{x},V,t),\quad &\boldsymbol{x}\in\hat{\Sigma}_{\delta_0,[0,T^*]},\\
					0,\quad &\text{otherwise}.
				\end{cases}
			\end{equation*}
			Let
			\begin{equation*}
				\check{\boldsymbol{G}}_V(\boldsymbol{x},t)\triangleq{V}^{-1}\sum_{i=1}^{M}\boldsymbol{\nu}_i\left(\check{r}_{+i}(V\boldsymbol{x},V,t)-\check{r}_{-i}(V\boldsymbol{x},V,t)\right),
			\end{equation*}
			and 
			\begin{equation*}
				\check{\boldsymbol{G}}(\boldsymbol{x},t)\triangleq\begin{cases}
					\boldsymbol{G}(\boldsymbol{x},t),\quad &\boldsymbol{x}\in\hat{\Sigma}_{\delta_0,[0,T^*]},\\
					0,\quad &\text{otherwise},
				\end{cases}
			\end{equation*}
			we can obtain
			\begin{equation*}
				\begin{aligned}
					\vert\check{\boldsymbol{x}}_{V}(t)-\hat{\boldsymbol{x}}_{\infty}(t\wedge\tau)\vert=&\vert\hat{\boldsymbol{x}}_{V}(t\wedge\tau)-\hat{\boldsymbol{x}}_{\infty}(t\wedge\tau)\vert\\
					=&\left\vert\hat{\boldsymbol{x}}_{V}(t\wedge\tau)-\boldsymbol{x}_T-\int_0^{t\wedge\tau}\boldsymbol{G}(\hat{\boldsymbol{x}}_{\infty}(s),s)\mathrm{d}s\right\vert\\
					\leq&\left\vert\hat{\boldsymbol{x}}_{V}(t\wedge\tau)-\boldsymbol{x}^{**}(\boldsymbol{x}_T,V)-\int_0^{t\wedge\tau}\boldsymbol{G}_V(\hat{\boldsymbol{x}}_{V}(s),s)\mathrm{d}s\right\vert\\
					&+\left\vert\boldsymbol{x}^{**}(\boldsymbol{x}_T,V)-\boldsymbol{x}_T\right\vert\\
					&+\int_0^{t\wedge\tau}\left\vert\boldsymbol{G}_V(\hat{\boldsymbol{x}}_{V}(s),s)-\boldsymbol{G}(\hat{\boldsymbol{x}}_{V}(s),s)\right\vert\mathrm{d}s\\
					&+\int_0^{t\wedge\tau}\left\vert\boldsymbol{G}(\hat{\boldsymbol{x}}_{V}(s),s)-\boldsymbol{G}(\hat{\boldsymbol{x}}_{\infty}(s),s)\right\vert\mathrm{d}s\\
					=&\left\vert\check{\boldsymbol{x}}_{V}(t)-\boldsymbol{x}^{**}(\boldsymbol{x}_T,V)-\int_0^{t}\check{\boldsymbol{G}}_V(\check{\boldsymbol{x}}_{V}(s),s)\mathrm{d}s\right\vert\\
					&+\left\vert\boldsymbol{x}^{**}(\boldsymbol{x}_T,V)-\boldsymbol{x}_T\right\vert\\
					&+\int_0^{t}\left\vert\check{\boldsymbol{G}}_V(\check{\boldsymbol{x}}_{V}(s),s)-\check{\boldsymbol{G}}(\check{\boldsymbol{x}}_{V}(s),s)\right\vert\mathrm{d}s\\
					&+\int_0^{t\wedge\tau}\left\vert\check{\boldsymbol{G}}(\hat{\boldsymbol{x}}_{V}(s),s)-\check{\boldsymbol{G}}(\hat{\boldsymbol{x}}_{\infty}(s),s)\right\vert\mathrm{d}s.
				\end{aligned}
			\end{equation*}
			
			Notice that 
			\begin{equation*}
				\begin{aligned}
					\big\vert V^{-1}\hat{r}_{\pm i}&(V\boldsymbol{x},V,t)-R_{\mp i}(\boldsymbol{x})e^{\mp\boldsymbol{\nu}_i\cdot\nabla_{\boldsymbol{x}}S(\boldsymbol{x},T-t\vert\boldsymbol{x}_0)}\big\vert\\
					\leq&\left\vert\frac{p_V(\boldsymbol{x}\pm V^{-1}\boldsymbol{\nu}_i,T-t\vert\boldsymbol{x}^*(\boldsymbol{x}_0,V))}{p_V(\boldsymbol{x},T-t\vert\boldsymbol{x}^*(\boldsymbol{x}_0,V))}-e^{\mp\boldsymbol{\nu}_i\cdot\nabla_{\boldsymbol{x}}S(\boldsymbol{x},T-t\vert\boldsymbol{x}_0)}\right\vert V^{-1} r_{\mp i}(V\boldsymbol{x}\pm\boldsymbol{\nu}_i,V)\\
					&+\left\vert V^{-1} r_{\mp i}(V\boldsymbol{x}\pm\boldsymbol{\nu}_i,V)-R_{\mp i}(\boldsymbol{x})\right\vert e^{\mp\boldsymbol{\nu}_i\cdot\nabla_{\boldsymbol{x}}S(\boldsymbol{x},T-t\vert\boldsymbol{x}_0)}.
				\end{aligned}
			\end{equation*}
			The second term on the right goes to zero by the uniform convergence of $V^{-1}r_{\pm i}(V\boldsymbol{x},V)$ to $R_{\pm i}(\boldsymbol{x})$, the smoothness of $R_{\pm i}(\boldsymbol{x})$, and the boundedness of ${\nabla_{\boldsymbol{x}}S(\boldsymbol{x},T-t\vert\boldsymbol{x}_0)}$ for $\boldsymbol{x}\in\hat{\Sigma}_{\delta_0,[0,T^*]}$ and $t\in[0,T^*]$. 
			
			According to Lemma \ref{theo0207}, for each $\boldsymbol{x}\in\left\{\boldsymbol{y}\in\mathbb{R}_+^N:\boldsymbol{y}=\boldsymbol{x}^{*}(\boldsymbol{x}_0,V)+V^{-1}\boldsymbol{\nu}^\top\boldsymbol{k},\boldsymbol{k}\in\mathbb{Z}^M\right\}\cap\hat{\Sigma}_{\delta_0,[0,T^*]}$ and $t\in[0,T^{*}]$, there exists a point $\boldsymbol{z}\in\hat{\Sigma}_{\delta_0,[0,T^*]}$ satisfying $\boldsymbol{x}=\boldsymbol{x}^{**}(\boldsymbol{z},V)$, such that for sufficiently large $V$,
			\begin{equation*}
				\begin{aligned}
					&\bigg\vert\frac{p_V(\boldsymbol{x}\pm V^{-1}\boldsymbol{\nu}_i,T-t\vert\boldsymbol{x}^*(\boldsymbol{x}_0,V))}{p_V(\boldsymbol{x},T-t\vert\boldsymbol{x}^*(\boldsymbol{x}_0,V))}-e^{\mp\boldsymbol{\nu}_i\cdot\nabla_{\boldsymbol{x}}S(\boldsymbol{x},T-t\vert\boldsymbol{x}_0)}\bigg\vert\\
					&\quad\quad\quad\quad\quad\quad\leq\bigg\vert\frac{p_V(\boldsymbol{x}\pm V^{-1}\boldsymbol{\nu}_i,T-t\vert\boldsymbol{x}^*(\boldsymbol{x}_0,V))}{p_V(\boldsymbol{x},T-t\vert\boldsymbol{x}^*(\boldsymbol{x}_0,V))}-e^{\mp\boldsymbol{\nu}_i\cdot\nabla_{\boldsymbol{x}}S(\boldsymbol{z},T-t\vert\boldsymbol{x}_0)}\bigg\vert\\
					&\quad\quad\quad\quad\quad\quad\quad+\big\vert e^{\mp\boldsymbol{\nu}_i\cdot\nabla_{\boldsymbol{x}}S(\boldsymbol{x},T-t\vert\boldsymbol{x}_0)}-e^{\mp\boldsymbol{\nu}_i\cdot\nabla_{\boldsymbol{x}}S(\boldsymbol{z},T-t\vert\boldsymbol{x}_0)}\big\vert\\
					&\quad\quad\quad\quad\quad\quad={o}(1)+{O}({1}/{V}).
				\end{aligned}
			\end{equation*}
			Combined with the boundedness of $V^{-1}r_{\pm i}(V\boldsymbol{x},V)$, we get
			\begin{equation*}
				\lim_{V\to\infty}V^{-1}\hat{r}_{\pm i}(V\boldsymbol{x},V,t)=R_{\mp i}(\boldsymbol{x})e^{\mp\boldsymbol{\nu}_i\cdot\nabla_{\boldsymbol{x}}S(\boldsymbol{x},T-t\vert\boldsymbol{x}_0)},
			\end{equation*}
			uniformly for $\boldsymbol{x}\in\left\{\boldsymbol{y}\in\mathbb{R}_+^N:\boldsymbol{y}=\boldsymbol{x}^{*}(\boldsymbol{x}_0,V)+V^{-1}\boldsymbol{\nu}^\top\boldsymbol{k},\boldsymbol{k}\in\mathbb{Z}^M\right\}\cap\hat{\Sigma}_{\delta_0,[0,T^*]}$ and $t\in[0,T^*]$. Hence, for any $\eta>0$, there exists a constant $V_0$ such that once $V>V_0$,
			\begin{equation*}
				\left\vert\check{\boldsymbol{G}}_V(\boldsymbol{x},t)-\check{\boldsymbol{G}}(\boldsymbol{x},t)\right\vert<\eta.
			\end{equation*}
			
			Furthermore, by the $\mathcal{C}^{1}$-smoothness of $G(\boldsymbol{x},t)$ in $\hat{\Sigma}_{\delta_0,[0,T^*]}\times[0,T^*]$ and the precompactness of $\hat{\Sigma}_{\delta_0,[0,T^*]}\times[0,T^*]$ in $\mathbb{R}_+^N\times[0,T]$, we know that for sufficiently small $\delta_0$,
			\begin{equation*}
				L_1\triangleq\sup_{(\boldsymbol{x},t)\in\hat{\Sigma}_{\delta_0,[0,T^*]}\times[0,T^*]}\vert\nabla_{\boldsymbol{x}}G(\boldsymbol{x},t)\vert<\infty.
			\end{equation*}
			Consequently,
			\begin{equation*}
				\begin{aligned}
					\int_0^{t\wedge\tau}\left\vert\check{\boldsymbol{G}}(\hat{\boldsymbol{x}}_{V}(s),s)-\check{\boldsymbol{G}}(\hat{\boldsymbol{x}}_{\infty}(s),s)\right\vert\mathrm{d}s\leq&L_1\int_0^{t\wedge\tau}\left\vert\hat{\boldsymbol{x}}_{V}(s)-\hat{\boldsymbol{x}}_{\infty}(s)\right\vert\mathrm{d}s\\
					=&L_1\int_0^{t\wedge\tau}\left\vert\check{\boldsymbol{x}}_{V}(s)-\hat{\boldsymbol{x}}_{\infty}(s\wedge\tau)\right\vert\mathrm{d}s\\
					\leq&L_1\int_0^{t}\left\vert\check{\boldsymbol{x}}_{V}(s)-\hat{\boldsymbol{x}}_{\infty}(s\wedge\tau)\right\vert\mathrm{d}s.
				\end{aligned}
			\end{equation*}
			
			Combined with the fact that for sufficiently large $V$,
			\begin{equation*}
				\vert\boldsymbol{x}^{**}(\boldsymbol{x}_T,V)-\boldsymbol{x}_T\vert<\eta,
			\end{equation*}
			and the Gronwall's inequality, we have 
			\begin{equation*}
				\sup_{t\in[0,T^*]}\left\vert\check{\boldsymbol{x}}_{V}(t)-\hat{\boldsymbol{x}}_{\infty}(t\wedge\tau)\right\vert\leq\left(\Delta_{V}+\eta(1+T^*)\right)e^{L_1T^*},
			\end{equation*}
			in which 
			$\Delta_{V}\triangleq\sup_{t\in[0,T^*]}\left\vert\check{\boldsymbol{x}}_{V}(t)-\boldsymbol{x}^{**}(\boldsymbol{x}_T,V)-\int_0^{t}\check{\boldsymbol{G}}_V(\check{\boldsymbol{x}}_{V}(s),s)\mathrm{d}s\right\vert$. This completes the proof.
		\end{proof}
		\begin{lemma}\label{theoA302}
			Assume the conditions of Proposition \ref{theo0401} hold. Let $\varphi:\mathbb{R}\to\mathbb{R}$ be a nonnegative, even, convex function such that both $\varphi$ and $\varphi{'}$ are absolutely continuous. If $\varphi(0)=\varphi{'}(0)=0$ and $\varphi{''}$ is non-negative and non-increasing on $(0,+\infty)$ (for example, $\varphi(u)=\vert{u}\vert^{\theta},\,1<\theta\leq{2}$), then for each $\delta>0$,
			\begin{equation*}
				P\left\{\Delta_{V}>\delta\right\}\leq\frac{8T^{*}\Gamma_2V}{N\varphi(\delta)}\sum_{i=1}^{M}\sum_{j=1}^{N}\varphi\left(\frac{C(N)\nu_{ij}}{2V}\right),
			\end{equation*}
			where
			\begin{equation*}
				\Gamma_2\triangleq\sup_{(\boldsymbol{x},t)\in\hat{\Sigma}_{\delta_0,[0,T^*]}\times[0,T^*],\,1\leq i\leq M}V^{-1}\hat{r}_{\pm i}(V\boldsymbol{x},V,t)<\infty.
			\end{equation*}
			
		\end{lemma}
		\begin{proof}
			Define $\boldsymbol{\xi}_{V}(t)\triangleq\check{\boldsymbol{x}}_{V}(t)-\boldsymbol{x}^{**}(\boldsymbol{x}_T,V)-\int_0^{t}\check{\boldsymbol{G}}_V(\check{\boldsymbol{x}}_{V}(s),s)\mathrm{d}s$. Then $\boldsymbol{\xi}_{V}(t)$ forms a martingale. The martingale inequality implies
			\begin{equation*}
				P\left\{\sup_{t\in[0,T^*]}\left\vert\boldsymbol{\xi}_{V}(t)\right\vert>\delta\right\}\leq[\varphi(\delta)]^{-1}E[\varphi(\vert\boldsymbol{\xi}_{V}(T^*)\vert)].
			\end{equation*}
			Therefore, the problem is to estimate the expectation on the right-hand side.
			
			Let $\boldsymbol{\zeta}_V(t)=(\check{\boldsymbol{x}}_{V}(t),\boldsymbol{\xi}_{V}(t))^{\top}$, we know that $\boldsymbol{\zeta}_V(t)$ is a Markov process. For any bounded continuously differentiable function $f(\boldsymbol{x},\boldsymbol{\xi})$, the infinitesimal generator of $\boldsymbol{\zeta}_V(t)$ is given by
			\begin{equation*}
				\begin{aligned}
					\mathscr{A}_{\boldsymbol{x},\boldsymbol{\xi},t}f(&\boldsymbol{x},\boldsymbol{\xi})\triangleq\lim_{h\downarrow 0}\frac{E_{(\check{\boldsymbol{x}}_{V}(t),\boldsymbol{\xi}_{V}(t))=(\boldsymbol{x},\boldsymbol{\xi})}f(\check{\boldsymbol{x}}_{V}(t+h),\boldsymbol{\xi}_{V}(t+h))-f(\boldsymbol{x},\boldsymbol{\xi})}{h}\\
					=&\sum_{i=1}^{M}\left(f(\boldsymbol{x}+V^{-1}\boldsymbol{\nu}_i,\boldsymbol{\xi}+V^{-1}\boldsymbol{\nu}_i)-f(\boldsymbol{x},\boldsymbol{\xi})-V^{-1}\boldsymbol{\nu}_{i}\cdot\nabla_{\boldsymbol{\xi}}f(\boldsymbol{x},\boldsymbol{\xi})\right)\check{r}_{+i}(V\boldsymbol{x},V,t)\\
					&+\sum_{i=1}^{M}\left(f(\boldsymbol{x}-V^{-1}\boldsymbol{\nu}_i,\boldsymbol{\xi}-V^{-1}\boldsymbol{\nu}_i)-f(\boldsymbol{x},\boldsymbol{\xi})+V^{-1}\boldsymbol{\nu}_{i}\cdot\nabla_{\boldsymbol{\xi}}f(\boldsymbol{x},\boldsymbol{\xi})\right)\check{r}_{-i}(V\boldsymbol{x},V,t).
				\end{aligned}
			\end{equation*}
			Notice that there exists a constant $C(N)$ so that 
			\begin{equation*}
				\varphi(\vert\boldsymbol{\xi}\vert)\leq\varphi\left(C(N)\frac{\vert{\xi}_1\vert+\cdots+\vert{\xi}_N\vert}{N}\right)\leq N^{-1}\sum_{j=1}^{N}\varphi(C(N)\vert{\xi}_{j}\vert)=N^{-1}\sum_{j=1}^{N}\varphi(C(N){\xi}_{j}).
			\end{equation*}
			Let $f(\boldsymbol{x},\boldsymbol{\xi})=N^{-1}\sum_{j=1}^{N}\varphi(C(N){\xi}_{j})$. Lemma 2.9 of [\onlinecite{Kurtz_1971}] gives
			\begin{equation*}
				E[\varphi(\vert\boldsymbol{\xi}_{V}(T^*)\vert)]\leq E[f(\boldsymbol{\xi}_{V}(T^*))]\leq f(0)+\int_{0}^{T^*}E[\mathscr{A}_{\boldsymbol{x},\boldsymbol{\xi},t}f(\check{\boldsymbol{x}}_V(t),\boldsymbol{\xi}_V(t))]\mathrm{d}t.
			\end{equation*}
			Then, the fact \cite{Kurtz_1972b}
			\begin{equation*}
				\begin{aligned}
					\varphi(z+u)-\varphi(z)-u\varphi'(z)=&\int_{0}^{u}\int_{0}^{v}\varphi{''}(z+w)\mathrm{d}w\mathrm{d}v\\
					\leq&2\int_{0}^{\vert u\vert}\int_{0}^{v/2}\varphi{''}(w)\mathrm{d}w\mathrm{d}v\\
					=&4\varphi(\frac{1}{2}\vert u\vert),
				\end{aligned}
			\end{equation*}
			implies
			\begin{equation*}
				\begin{aligned}
					E[\varphi(\vert\boldsymbol{\xi}_{V}(T^*)\vert)]\leq& 4N^{-1}\sum_{i=1}^{M}\sum_{j=1}^{N}\varphi(\frac{C(N)\nu_{ij}}{2V})\int_{0}^{T^*} E[\check{r}_{+i}(V\check{\boldsymbol{x}}_V(t),V,t)+\check{r}_{-i}(V\check{\boldsymbol{x}}_V(t),V,t)]\mathrm{d}t\\
					\leq&8T^{*}\Gamma_2VN^{-1}\sum_{i=1}^{M}\sum_{j=1}^{N}\varphi(\frac{C(N)\nu_{ij}}{2V}).
				\end{aligned}
			\end{equation*}		
			This completes the lemma.
		\end{proof}
		\begin{proof}[Proof of proposition \ref{theo0401}]
			Let $\varphi(u)=u^2$, then Lemmas \ref{theoA301} and \ref{theoA302} imply that for any $T^*\in(0,T)$, there is a constant $\delta_0$ so that for each $\eta>0$,
			\begin{equation*}
				\lim_{V\to\infty}P\left\{\sup_{t\in[0,T^*]}\left\vert\hat{\boldsymbol{x}}_{V}(t\wedge\tau)-\hat{\boldsymbol{x}}_{\infty}(t\wedge\tau)\right\vert>\eta\right\}=0.
			\end{equation*}
			If $\sup_{t\in[0,T^*]}\left\vert\hat{\boldsymbol{x}}_{V}(t\wedge\tau)-\hat{\boldsymbol{x}}_{\infty}(t\wedge\tau)\right\vert\leq\eta\leq\delta<\delta_0$, then we know that $\hat{\boldsymbol{x}}_{V}(t\wedge\tau)\in\hat{\Sigma}_{\delta,[0,T^*]}\subset\hat{\Sigma}_{\delta_0,[0,T^*]}$ for $t\in[0,T^*]$, and hence $T^*<\tau$. Consequently, $\sup_{t\in[0,T^*]}\left\vert\hat{\boldsymbol{x}}_{V}(t)-\hat{\boldsymbol{x}}_{\infty}(t)\right\vert\leq\delta$, and this theorem follows. 
		\end{proof}
		
		\section{Proof of Proposition \ref{theo0402}}\label{secA4}
		\begin{lemma}\label{theoA401}
			Assume the conditions of Proposition \ref{theo0402} hold. Then $\boldsymbol{\mu}_V(t)$ is tight (more precisely, $\mathcal{C}$-tight) in $\mathcal{D}([0,T^*];\mathbb{R}^N)$.
		\end{lemma}
		\begin{proof}
			According to [\onlinecite[Chapter VI, Proposition 3.26]{Jacod_2013}], in order to prove the tightness (more precisely, the $\mathcal{C}$-tightness) of the sequence $\boldsymbol{\mu}_V(t)$, it suffices to show that for any $\varepsilon>0$, there are $K>0$ and $V_{1}>0$ with
			\begin{equation}\label{eq64}
				\sup_{V>V_{1}}P\left\{\sup_{t\in[0,T^*]}\left\vert\boldsymbol{\mu}_{V}(t)\right\vert>K\right\}<\varepsilon,
			\end{equation}
			and for any $\eta>0$, $\varepsilon>0$, there are $\gamma>0$ and $V_{2}>0$ with
			\begin{equation}\label{eq65}
				\sup_{V>V_{2}}P\left\{\sup_{0\leq t_1\leq t_2\leq T^*,\,t_2-t_1\leq \gamma}\left\vert\boldsymbol{\mu}_{V}(t_2)-\boldsymbol{\mu}_{V}(t_1)\right\vert>\eta\right\}<\varepsilon.
			\end{equation}
			
			Using the assumptions in this lemma and the results in Proposition \ref{theo0401}, we know that if $V>\frac{K^2}{\delta^2_0}$, 
			\begin{equation*}
				\begin{aligned}
					P\left\{\sup_{t\in[0,T^*]}\left\vert\boldsymbol{\mu}_{V}(t)\right\vert\leq K\right\}=&P\left\{\sup_{t\in[0,T^*]}\left\vert\hat{\boldsymbol{x}}_V(t)-\hat{\boldsymbol{x}}_{\infty}(t)\right\vert\leq \frac{K}{\sqrt{V}}<\delta_0\right\}\\
					=&P\left\{\sup_{t\in[0,T^*]}\left\vert\hat{\boldsymbol{x}}_V(t\wedge\tau)-\hat{\boldsymbol{x}}_{\infty}(t\wedge\tau)\right\vert\leq \frac{K}{\sqrt{V}}<\delta_0\right\}\\
					=&P\left\{\sup_{t\in[0,T^*]}\left\vert\boldsymbol{\mu}_{V}(t\wedge\tau)\right\vert\leq K\right\},
				\end{aligned}
			\end{equation*}
			and for sufficiently large $V$ (for example, $V>V^*(K)$),
			\begin{equation*}
				\sqrt{V}\vert\boldsymbol{x}^{**}(\boldsymbol{x}_T,V)-\boldsymbol{x}_T\vert<\frac{K}{3e^{L_1T^*}},
			\end{equation*} 
			\begin{equation*}
				\sup_{(\boldsymbol{x},t)\in\hat{\Sigma}_{\delta_0,[0,T^*]}\times[0,T^*]}\sqrt{V}\vert{\boldsymbol{G}}_V(\boldsymbol{x},t)-{\boldsymbol{G}}(\boldsymbol{x},t)\vert\leq\frac{K}{3T^*e^{L_1T^*}}.
			\end{equation*}
			Combined with the fact
			\begin{equation*}
				\sup_{t\in[0,T^*]}\left\vert\boldsymbol{\mu}_{V}(t\wedge\tau)\right\vert\leq\sqrt{V}\left(\Delta_{V}+\vert\boldsymbol{x}^{**}(\boldsymbol{x}_T,V)-\boldsymbol{x}_T\vert+\int_0^{T^*}\left\vert\check{\boldsymbol{G}}_V(\check{\boldsymbol{x}}_{V}(s),s)-\check{\boldsymbol{G}}(\check{\boldsymbol{x}}_{V}(s),s)\right\vert\mathrm{d}s\right)e^{L_1T^*},
			\end{equation*}
			we obtain
			\begin{equation*}
				P\left\{\sup_{t\in[0,T^*]}\left\vert\boldsymbol{\mu}_{V}(t)\right\vert>K\right\}\leq P\left\{\Delta_{V}>\frac{K}{3\sqrt{V}e^{L_1T^*}}\right\}\leq\frac{8T^{*}\Gamma_2V}{N\varphi(\frac{K}{3\sqrt{V}e^{L_1T^*}})}\sum_{i=1}^{M}\sum_{j=1}^{N}\varphi(\frac{C(N)\nu_{ij}}{2V}).
			\end{equation*}
			Set $\varphi(u)=u^2$, let $K$ be selected so that
			\begin{equation*}
				\frac{8T^{*}\Gamma_2V}{N\varphi(\frac{K}{3\sqrt{V}e^{L_1T^*}})}\sum_{i=1}^{M}\sum_{j=1}^{N}\varphi(\frac{C(N)\nu_{ij}}{2V})<\varepsilon,
			\end{equation*}
			and define $V_{1}=\max({K^2}/{\delta^2_0},V^*(K))$, then we can get Eq. (\ref{eq64}).
			
			Utilizing the Markov property, it can be shown easily that for sufficiently large $V$
			\begin{equation*}
				P\left\{\sup_{0\leq t_1\leq t_2\leq T^*}\left\vert\boldsymbol{\mu}_{V}(t_2)-\boldsymbol{\mu}_{V}(t_1)\right\vert>\eta\right\}\leq\frac{8(t_2-t_1)\Gamma_2V}{N\varphi(\frac{\eta}{3\sqrt{V}e^{L_1T^*}})}\sum_{i=1}^{M}\sum_{j=1}^{N}\varphi(\frac{C(N)\nu_{ij}}{2V}).
			\end{equation*}
			Set $\varphi(u)=u^2$, choose sufficiently small $\gamma$ such that $t_2-t_1\leq\gamma$,
			\begin{equation*}
				\frac{8\gamma\Gamma_2V}{N\varphi(\frac{\eta}{3\sqrt{V}e^{L_1T^*}})}\sum_{i=1}^{M}\sum_{j=1}^{N}\varphi(\frac{C(N)\nu_{ij}}{2V})<\varepsilon,
			\end{equation*}
			and let $V_{2}=\max({\gamma^2}/{\delta^2_0},V^*(\gamma))$, then we can achieve Eq. (\ref{eq65}).
		\end{proof}
		\begin{lemma}\label{theoA402}
			Under the hypotheses of Proposition \ref{theo0402}, the finite-dimensional distributions of $\boldsymbol{\mu}_V(t)$ on the interval $[0,T^*]$ converge to the finite-dimensional distributions of the diffusion $\boldsymbol{\mu}_{\infty}(t)$.
		\end{lemma}
		\begin{proof}
			Denote $\check{\boldsymbol{J}}_{V}(\boldsymbol{x},t)\triangleq{\boldsymbol{J}}_{V}(\boldsymbol{x},t)1_{\hat{\Sigma}_{\delta_0,[0,T^*]}}(\boldsymbol{x})$ and $\check{\boldsymbol{J}}(\boldsymbol{x},t)\triangleq{\boldsymbol{J}}(\boldsymbol{x},t)1_{\hat{\Sigma}_{\delta_0,[0,T^*]}}(\boldsymbol{x})$, where $1_{\hat{\Sigma}_{\delta_0,[0,T^*]}}(\boldsymbol{x})$ is the indicator function of $\hat{\Sigma}_{\delta_0,[0,T^*]}$. Clearly, ${\boldsymbol{J}}_{V}(\boldsymbol{x},t)$ converges to ${\boldsymbol{J}}(\boldsymbol{x},t)$ uniformly for $(\boldsymbol{x},t)\in\hat{\Sigma}_{\delta_0,[0,T^*]}\times[0,T^*]$. 
			
			Let $\Phi_{V}(\boldsymbol{\theta},t)\triangleq E\exp\{i\boldsymbol{\theta}\cdot\boldsymbol{\mu}_{V}(t)\}$. For each $\delta<\delta_0$, let $A_4=\{\sup_{t\in[0,T^*]}\vert\boldsymbol{\mu}_{V}(t\wedge\tau)\vert\leq\sqrt{V}\delta\}=\{\sup_{t\in[0,T^*]}\vert\hat{\boldsymbol{x}}_V(t\wedge\tau)-\hat{\boldsymbol{x}}_{\infty}(t\wedge\tau)\vert\leq\delta\}$. Then we have
			\begin{equation*}
				\begin{aligned}
					\Phi_{V}(\boldsymbol{\theta},t\wedge\tau)=&E\exp\{i\boldsymbol{\theta}\cdot\boldsymbol{\mu}_{V}(t\wedge\tau)\}1_{A_4}+E\exp\{i\boldsymbol{\theta}\cdot\boldsymbol{\mu}_{V}(t\wedge\tau)\}1_{A_4^{\text{C}}}\\
					=&E\exp\{i\boldsymbol{\theta}\cdot\boldsymbol{\mu}_{V}(t)\}1_{A_4}+E\exp\{i\boldsymbol{\theta}\cdot\boldsymbol{\mu}_{V}(t\wedge\tau)\}1_{A_4^{\text{C}}},
				\end{aligned}
			\end{equation*}
			followed by
			\begin{equation*}
				\Phi_{V}(\boldsymbol{\theta},t)-\Phi_{V}(\boldsymbol{\theta},t\wedge\tau)=E\exp\{i\boldsymbol{\theta}\cdot\boldsymbol{\mu}_{V}(t)\}1_{A_4^{\text{C}}}-E\exp\{i\boldsymbol{\theta}\cdot\boldsymbol{\mu}_{V}(t\wedge\tau)\}1_{A_4^{\text{C}}},
			\end{equation*}
			where $1_{A_4}$ is the indicator function of $A_4$ and $A_4^{\text{C}}$ stands for the complement of $A_4$.  Let $\Xi_{1,V}(\boldsymbol{\theta},t)$ be the term on the right, then
			\begin{equation*}
				\sup_{t\in[0,T^*]}\vert\Xi_{1,V}(\boldsymbol{\theta},t)\vert\leq 2P\left\{\sup_{t\in[0,T^*]}\vert\hat{\boldsymbol{x}}_{V}(t\wedge\tau)-\hat{\boldsymbol{x}}_{\infty}(t\wedge\tau)\vert>\delta\right\}\to {0},\quad \text{as}\;V\to\infty.
			\end{equation*}
			
			Denote
			\begin{equation*}
				\Lambda_{1,V}(\boldsymbol{\theta},t)\triangleq E\exp\big\{i\boldsymbol{\theta}\cdot\sqrt{V}\big[\boldsymbol{\xi}_V(t)+\int_0^{t\wedge\tau}(\boldsymbol{G}(\hat{\boldsymbol{x}}_{V}(s\wedge\tau),s)-\boldsymbol{G}(\hat{\boldsymbol{x}}_{\infty}(s\wedge\tau),s))\mathrm{d}s\big]\big\},
			\end{equation*}
			we know that
			\begin{equation*}
				\begin{aligned}
					\Xi_{2,V}(\boldsymbol{\theta}&,t)\triangleq\Phi_{V}(\boldsymbol{\theta},t\wedge\tau)-\Lambda_{1,V}(\boldsymbol{\theta},t)\\
					=&E\exp\bigg\{i\boldsymbol{\theta}\cdot\sqrt{V}\bigg[\boldsymbol{\xi}_V(t)+\int_0^{t\wedge\tau}(\boldsymbol{G}(\hat{\boldsymbol{x}}_{V}(s\wedge\tau),s)-\boldsymbol{G}(\hat{\boldsymbol{x}}_{\infty}(s\wedge\tau),s))\mathrm{d}s\bigg]\bigg\}\times\\
					&\bigg[\exp\bigg\{i\boldsymbol{\theta}\cdot\sqrt{V}\bigg[\boldsymbol{x}^{**}(\boldsymbol{x}_T,V)-\boldsymbol{x}_T+\int_0^{t\wedge\tau}(\boldsymbol{G}_V(\hat{\boldsymbol{x}}_{V}(s\wedge\tau),s)-\boldsymbol{G}(\hat{\boldsymbol{x}}_{V}(s\wedge\tau),s))\mathrm{d}s\bigg]\bigg\}-1\bigg],
				\end{aligned}
			\end{equation*}
			goes to zero uniformly for $t\in[0,T^*]$ due to Eq. (\ref{eq0406}) and the fact $\sqrt{V}\vert\boldsymbol{x}^{**}(\boldsymbol{x}_T,V)-\boldsymbol{x}_T\vert\to 0$.
			
			Define
			\begin{equation*}
				\begin{aligned}
					\Lambda_{2,V}(\boldsymbol{\theta},t)\triangleq E\exp\big\{i\boldsymbol{\theta}\cdot\sqrt{V}\big[&\boldsymbol{\xi}_V(t)+\int_0^{t}(\check{\boldsymbol{G}}(\check{\boldsymbol{x}}_{V}(s),s)-\check{\boldsymbol{G}}(\hat{\boldsymbol{x}}_{\infty}(s\wedge\tau),s))\mathrm{d}s\big]\big\},
				\end{aligned}
			\end{equation*}
			and 
			\begin{equation*}
				\Xi_{3,V}(\boldsymbol{\theta},t)\triangleq\Lambda_{1,V}(\boldsymbol{\theta},t)-\Lambda_{2,V}(\boldsymbol{\theta},t).
			\end{equation*}
			It can be verified easily that 
			\begin{equation*}
				\sup_{t\in[0,T^*]}\vert\Xi_{3,V}(\boldsymbol{\theta},t)\vert\leq 2P\left\{\sup_{t\in[0,T^*]}\vert\hat{\boldsymbol{x}}_{V}(t\wedge\tau)-\hat{\boldsymbol{x}}_{\infty}(t\wedge\tau)\vert>\delta\right\}.
			\end{equation*}
			
			Let $\boldsymbol{\varsigma}_{V}(t)\triangleq\boldsymbol{\xi}_{V}(t)+\int_0^{t}(\check{\boldsymbol{G}}(\check{\boldsymbol{x}}_{V}(s),s)-\check{\boldsymbol{G}}(\hat{\boldsymbol{x}}_{\infty}(s\wedge\tau),s))\mathrm{d}s$, then $(\check{\boldsymbol{x}}_{V}(t),\boldsymbol{\varsigma}_{V}(t))^{\top}$ is a Markov process. According to Lemma 2.6 of [\onlinecite{Kurtz_1971}], the function $\exp\{i\boldsymbol{\theta}\cdot\sqrt{V}\boldsymbol{\varsigma}\}$ is in the domain of the weak infinitesimal operator of $(\check{\boldsymbol{x}}_{V}(t),\boldsymbol{\varsigma}_{V}(t))^{\top}$. Hence, 
			\begin{equation*}
				\begin{aligned}
					\Lambda_{2,V}(\boldsymbol{\theta},t)-1=&\int_{0}^{t}E\exp\{i\sqrt{V}\boldsymbol{\theta}\cdot\boldsymbol{\varsigma}_V(s)\}\bigg\{\sum_{j=1}^{M}\bigg[\exp\{i{V}^{-1/2}\boldsymbol{\theta}\cdot\boldsymbol{\nu}_j\}-1-iV^{-1/2}\boldsymbol{\theta}\cdot\boldsymbol{\nu}_j\bigg]\\
					&\times\check{r}_{+j}(V\check{\boldsymbol{x}}_{V}(s),V,s)+\sum_{j=1}^{M}\bigg[\exp\{-i{V}^{-1/2}\boldsymbol{\theta}\cdot\boldsymbol{\nu}_j\}-1+iV^{-1/2}\boldsymbol{\theta}\cdot\boldsymbol{\nu}_j\bigg]\\
					&\times\check{r}_{-j}(V\check{\boldsymbol{x}}_{V}(s),V,s)+i\sqrt{V}\boldsymbol{\theta}\cdot\bigg[\check{\boldsymbol{G}}(\check{\boldsymbol{x}}_{V}(s),s)-\check{\boldsymbol{G}}(\hat{\boldsymbol{x}}_{\infty}(s\wedge\tau),s)\bigg]\bigg\}\mathrm{d}s.
				\end{aligned}
			\end{equation*}
			If $\sup_{t\in[0,T^*]}\vert\hat{\boldsymbol{x}}_V(t\wedge\tau)-\hat{\boldsymbol{x}}_{\infty}(t\wedge\tau)\vert\leq\delta$, then for any $\varepsilon$, there exist a parameter $\delta_{1}>0$ and a point $\check{\boldsymbol{x}}^*_V(t)$ between $\check{\boldsymbol{x}}_{V}(t)$ and $\hat{\boldsymbol{x}}_{\infty}(t\wedge\tau)$ for each $t\in[0,T^*]$, such that once $\delta<\delta_{1}$,
			\begin{equation*}
				\check{\boldsymbol{G}}(\check{\boldsymbol{x}}_{V}(t),t)-\check{\boldsymbol{G}}(\hat{\boldsymbol{x}}_{\infty}(t\wedge\tau),t)=\nabla_{\boldsymbol{x}}\check{\boldsymbol{G}}(\check{\boldsymbol{x}}^*_V(t),t)\cdot(\check{\boldsymbol{x}}_{V}(t)-\hat{\boldsymbol{x}}_{\infty}(t\wedge\tau)),
			\end{equation*}
			with $\boldsymbol{\Xi}_{4,V}(t)\triangleq\nabla_{\boldsymbol{x}}\check{\boldsymbol{G}}(\check{\boldsymbol{x}}^*_V(t),t)-\nabla_{\boldsymbol{x}}\check{\boldsymbol{G}}(\hat{\boldsymbol{x}}_{\infty}(t\wedge\tau),t)$ satisfying
			\begin{equation*}
				\sup_{t\in[0,T^*]}\vert\boldsymbol{\Xi}_{4,V}(t)\vert<\varepsilon.
			\end{equation*}
			Consequently,
			\begin{equation*}
				\begin{aligned}
					\Lambda_{2,V}^{(1)}(\boldsymbol{\theta},t)\triangleq&\int_{0}^{t}E\bigg\{i\sqrt{V}\boldsymbol{\theta}\cdot\bigg[\check{\boldsymbol{G}}(\check{\boldsymbol{x}}_{V}(s),s)-\check{\boldsymbol{G}}(\hat{\boldsymbol{x}}_{\infty}(s\wedge\tau),s)\bigg]\exp\{i\sqrt{V}\boldsymbol{\theta}\cdot\boldsymbol{\varsigma}_V(s)\}\bigg\}\mathrm{d}s\\
					=&\int_{0}^{t}E\bigg\{i\sqrt{V}\boldsymbol{\theta}\cdot\bigg(\nabla_{\boldsymbol{x}}\check{\boldsymbol{G}}(\hat{\boldsymbol{x}}_{\infty}(s\wedge\tau),s)+\boldsymbol{\Xi}_{4,V}(s)\bigg)\cdot(\check{\boldsymbol{x}}_{V}(s)-\hat{\boldsymbol{x}}_{\infty}(s\wedge\tau))\exp\{i\sqrt{V}\boldsymbol{\theta}\cdot\boldsymbol{\varsigma}_V(s)\}\bigg\}\mathrm{d}s\\
					&+\Xi_{5,V}(\boldsymbol{\theta},t)\\
					=&\int_{0}^{t}E\bigg\{i\sqrt{V}\boldsymbol{\theta}\cdot\bigg(\nabla_{\boldsymbol{x}}\check{\boldsymbol{G}}(\hat{\boldsymbol{x}}_{\infty}(s),s)+\boldsymbol{\Xi}_{4,V}(s)\bigg)\cdot(\check{\boldsymbol{x}}_{V}(s)-\hat{\boldsymbol{x}}_{\infty}(s\wedge\tau))\exp\{i\sqrt{V}\boldsymbol{\theta}\cdot\boldsymbol{\varsigma}_V(s)\}\bigg\}\mathrm{d}s\\
					&+\Xi_{5,V}(\boldsymbol{\theta},t)+\Xi_{6,V}(\boldsymbol{\theta},t)\\
					=&\int_{0}^{t}i\sqrt{V}\boldsymbol{\theta}\cdot\nabla_{\boldsymbol{x}}\check{\boldsymbol{G}}(\hat{\boldsymbol{x}}_{\infty}(s),s)\cdot{E}\left((\check{\boldsymbol{x}}_{V}(s)-\hat{\boldsymbol{x}}_{\infty}(s\wedge\tau))\exp\{i\sqrt{V}\boldsymbol{\theta}\cdot\boldsymbol{\varsigma}_V(s)\}\right)\mathrm{d}s\\
					&+\Xi_{5,V}(\boldsymbol{\theta},t)+\Xi_{6,V}(\boldsymbol{\theta},t)+\Xi_{7,V}(\boldsymbol{\theta},t)\\
					=&\int_{0}^{t}\boldsymbol{\theta}\cdot\nabla_{\boldsymbol{x}}\check{\boldsymbol{G}}(\hat{\boldsymbol{x}}_{\infty}(s),s)\cdot\nabla_{\boldsymbol{\theta}}\Phi_{V}(\boldsymbol{\theta},s\wedge\tau)\mathrm{d}s\\
					&+\Xi_{5,V}(\boldsymbol{\theta},t)+\Xi_{6,V}(\boldsymbol{\theta},t)+\Xi_{7,V}(\boldsymbol{\theta},t)+\Xi_{8,V}(\boldsymbol{\theta},t)
				\end{aligned}
			\end{equation*}
			in which $\Xi_{5,V}(\boldsymbol{\theta},t)$, $\Xi_{6,V}(\boldsymbol{\theta},t)$, $\Xi_{7,V}(\boldsymbol{\theta},t)$ and $\Xi_{8,V}(\boldsymbol{\theta},t)$ converge to zero uniformly for $t\in[0,T^*]$. This leaves
			\begin{equation*}
				\begin{aligned}
					\Lambda_{2,V}^{(2)}(\boldsymbol{\theta},t)\triangleq&\Lambda_{2,V}(\boldsymbol{\theta},t)-1-\Lambda_{2,V}^{(1)}(\boldsymbol{\theta},t)\\
					=&-\int_{0}^{t}\bigg(\frac{1}{2}\boldsymbol{\theta}\cdot\check{\boldsymbol{J}}(\hat{\boldsymbol{x}}_{\infty}(t),t)\cdot\boldsymbol{\theta}\bigg)\Lambda_{2,V}(\boldsymbol{\theta},s)\mathrm{d}s\\
					&+\int_{0}^{t}E\bigg(\frac{1}{2}\boldsymbol{\theta}\cdot\left(\check{\boldsymbol{J}}(\hat{\boldsymbol{x}}_{\infty}(t),t)-\check{\boldsymbol{J}}_{V}(\check{\boldsymbol{x}}_{V}(t),t)\right)\cdot\boldsymbol{\theta}\bigg)\exp\{i\sqrt{V}\boldsymbol{\theta}\cdot\boldsymbol{\varsigma}_V(s)\}\mathrm{d}s\\
					&+\int_{0}^{t}E\exp\{i\sqrt{V}\boldsymbol{\theta}\cdot\boldsymbol{\varsigma}_V(s)\}\bigg\{\sum_{j=1}^{M}\varrho({V}^{-1/2}\boldsymbol{\theta}\cdot\boldsymbol{\nu}_j)(\boldsymbol{\theta}\cdot\boldsymbol{\nu}_j)^2V^{-1}\check{r}_{+j}(V\check{\boldsymbol{x}}_{V}(s),V,s)\\
					&\qquad\qquad\qquad+\sum_{j=1}^{M}\varrho(-{V}^{-1/2}\boldsymbol{\theta}\cdot\boldsymbol{\nu}_j)(\boldsymbol{\theta}\cdot\boldsymbol{\nu}_j)^2V^{-1}\check{r}_{+j}(V\check{\boldsymbol{x}}_{V}(s),V,s)\bigg\}\mathrm{d}s,
				\end{aligned}
			\end{equation*}
			where $\varrho(u)=(e^{iu}-1-iu+u^2/2)/u^2$.
			The second term on the right, call it $\Xi_{9,V}(\boldsymbol{\theta},t)$, converges to zero by the uniform convergence of $\check{\boldsymbol{J}}_V$ to $\check{\boldsymbol{J}}$, the uniform continuity of $\check{\boldsymbol{J}}$ and Lemmas \ref{theoA301}, \ref{theoA302}. Moreover, $\lim_{u\to 0}\varrho(u)=0$ implies the third term on the right, denote it by $\Xi_{10,V}(\boldsymbol{\theta},t)$, vanishes as $V\to\infty$.
			
			In summary, we have
			\begin{equation*}
				\begin{aligned}
					\Phi_{V}(\boldsymbol{\theta},t\wedge\tau)-1=&-\int_{0}^{t}\bigg(\frac{1}{2}\boldsymbol{\theta}\cdot\check{\boldsymbol{J}}(\hat{\boldsymbol{x}}_{\infty}(s),s)\cdot\boldsymbol{\theta}\bigg)\Phi_{V}(\boldsymbol{\theta},s\wedge\tau)\mathrm{d}s\\
					&+\int_{0}^{t}\bigg\{\boldsymbol{\theta}\cdot\nabla_{\boldsymbol{x}}\check{\boldsymbol{G}}(\hat{\boldsymbol{x}}_{\infty}(s),s)\cdot\nabla_{\boldsymbol{\theta}}\Phi_{V}(\boldsymbol{\theta},s\wedge\tau)\bigg\}\mathrm{d}s\\
					&+\Xi_{11,V}(\boldsymbol{\theta},t),
				\end{aligned}
			\end{equation*}
			followed by
			\begin{equation*}
				\begin{aligned}
					\Phi_{V}(\boldsymbol{\theta},t\wedge\tau)-\Phi(\boldsymbol{\theta},t)=&-\int_{0}^{t}\bigg(\frac{1}{2}\boldsymbol{\theta}\cdot\check{\boldsymbol{J}}(\hat{\boldsymbol{x}}_{\infty}(s),s)\cdot\boldsymbol{\theta}\bigg)(\Phi_{V}(\boldsymbol{\theta},s\wedge\tau)-\Phi(\boldsymbol{\theta},s))\mathrm{d}s\\
					&+\int_{0}^{t}\bigg\{\boldsymbol{\theta}\cdot\nabla_{\boldsymbol{x}}\check{\boldsymbol{G}}(\hat{\boldsymbol{x}}_{\infty}(s),s)\cdot\nabla_{\boldsymbol{\theta}}(\Phi_{V}(\boldsymbol{\theta},s\wedge\tau)-\Phi(\boldsymbol{\theta},s))\bigg\}\mathrm{d}s\\
					&+\Xi_{11,V}(\boldsymbol{\theta},t),
				\end{aligned}
			\end{equation*}
			where
			\begin{equation*}
				\begin{aligned}
					\Xi_{11,V}(\boldsymbol{\theta},t)\triangleq&\int_{0}^{t}\bigg(\frac{1}{2}\boldsymbol{\theta}\cdot\check{\boldsymbol{J}}(\hat{\boldsymbol{x}}_{\infty}(s),s)\cdot\boldsymbol{\theta}\bigg)(\Xi_{2,V}(\boldsymbol{\theta},s)+\Xi_{3,V}(\boldsymbol{\theta},s))\mathrm{d}s\\
					&+\Xi_{2,V}(\boldsymbol{\theta},s)+\Xi_{3,V}(\boldsymbol{\theta},s)+\Xi_{5,V}(\boldsymbol{\theta},s)+\Xi_{6,V}(\boldsymbol{\theta},s)\\
					&+\Xi_{7,V}(\boldsymbol{\theta},s)+\Xi_{8,V}(\boldsymbol{\theta},s)+\Xi_{9,V}(\boldsymbol{\theta},s)+\Xi_{10,V}(\boldsymbol{\theta},s).
				\end{aligned}
			\end{equation*}
			By the method of characteristics for first-order partial differential equations, one can show that
			\begin{equation*}
				\lim_{V\to\infty}\sup_{t\in[0,T^*]}\vert\Phi_{V}(\boldsymbol{\theta},t\wedge\tau)-\Phi(\boldsymbol{\theta},t)\vert=0,
			\end{equation*}
			and hence
			\begin{equation*}
				\lim_{V\to\infty}\sup_{t\in[0,T^*]}\vert\Phi_{V}(\boldsymbol{\theta},t)-\Phi(\boldsymbol{\theta},t)\vert=0.
			\end{equation*}
			Consequently, the convergence of the finite-dimensional distributions of $\boldsymbol{\mu}_V(t)$ to the finite-dimensional distributions of $\boldsymbol{\mu}_{\infty}(t)$ follows from this fact and the Markov property.
		\end{proof}
		\begin{proof}[Proof of Proposition \ref{theo0402}]
			Proposition \ref{theo0402} follows from Lemmas \ref{theoA401}, \ref{theoA402} and [\onlinecite[Chapter 9, Theorem 5.2]{Gikhman_1969}] (cf. also [\onlinecite[Chapter VI, Theorem 3.20]{Jacod_2013}]).
		\end{proof}

		\section{Algorithms for Calculating the Non-stationary and Stationary Prehistory Probabilities}\label{secA6}
		In this section, the focus is on the case of $N=1$ and $M=1$. The stochastic model is 
		\begin{equation*}
			{x}_{V}(t)={x}^{*}({x}_0,V)
			+\frac{\nu}{V}\left(Y_{+}\left(\int_{0}^{t}{r_{+}(V{x}_{V}(s),V)\mathrm{d}s}\right)-Y_{-}\left(\int_{0}^{t}{r_{-}(V{x}_{V}(s),V)\mathrm{d}s}\right)\right),
		\end{equation*}
		Assume ${x}^{*}({x}_0,V)\in V^{-1}\nu\mathbb{N}$, we know that the state space is $V^{-1}\nu\mathbb{N}$.
		\subsection{An Algorithm for Calculating the Non-stationary Prehistory Probability}
		\textbf{Initialize:}
		\begin{itemize}
			\item[(1)] Choose a domain $D=[x_l,x_r]\subset\mathbb{R}_{+}$ such that $x_0,x_T\in{D}$, the vector field $F(x)$ at the boundary $\partial{D}$ is directed towards the interior, and
			\begin{equation*}
				\min_{x\in\partial{D},0\leq{t}\leq{T}}S(x,t\vert x_0)>S(x_T,T\vert x_0).
			\end{equation*}
			In other words, the NOP connecting $x_0$ with $x_T$ in the time span $T$ is entirely contained within the domain $D$, and neglecting all of the trajectories that escape before the moment T has negligible impact on the calculation of the NOP. Define $N_x\triangleq\lfloor(x_l-x_r)V/\nu\rfloor$. Evidently, the number of states in $D$ is equal to $N_x$. 
			\item[(2)] Partition the state space into $N_x+1$ subsets: The $i$th subset, denoted by $D_i$, is defined as a set containing the single point $x(i)=(\lfloor{x}_lV/\nu\rfloor+i)\nu/V$ for $i=1,\cdots,N_x$. The $(N_x+1)$th subset is defined as the complement of $\{x(i):i=1,\cdots,N_x\}$.
			\item[(3)] Discretize the time interval $[0,T]$ into $N_t$ steps with a uniform size of $\Delta{t}=T/N_t$.
			\item[(4)] By the Euler $\tau$-leaping method \cite{Kurtz_2015},
			\begin{equation*}
				{x}_{V}(t+\Delta{t})\simeq{x}_{V}(t)
				+\frac{\nu}{V}\left(\text{Poisson}\left(r_{+}(V{x}_{V}(t),V)\Delta{t}\right)-\text{Poisson}\left(r_{-}(V{x}_{V}(t),V)\Delta{t}\right)\right),
			\end{equation*}
			where $\text{Poisson}(u)$ signifies a Poisson-distributed stochastic variable with parameter $u$, and the different variables are independent of each other. Define a homogeneous Markov chain $x^{\text{stop}}_V(m\Delta{t})$ for $m\in\mathbb{N}$, which is the discrete approximation of the stopped process $x^{\text{stop}}_V(t)$ (a process generated by $x_V(t)$ that stops when it escapes from the domain $D$). The transition probability of this process is given by
			\begin{equation*}
				\begin{gathered}
					\begin{aligned}
						P_{i,j}&=P({x}_{V}(t+\Delta{t})=x(j)\vert{x}_{V}(t)=x(i))\\
						&=P_{\text{Skellam}}\left(\frac{V(x(j)-x(i))}{\nu};r_{+}(V{x}(i),V)\Delta{t},r_{-}(V{x}(i),V)\Delta{t}\right),\quad 1\leq i,j\leq N_x,
					\end{aligned}\\
					P_{i,N_x+1}=1-\sum_{j=1}^{N_x}P_{i,j},\quad P_{N_x+1,j}=0,\quad 1\leq i,j\leq N_x,\\
					P_{N_x+1,N_x+1}=1,
				\end{gathered}
			\end{equation*}
			where $P_{\text{Skellam}}(\cdot;u_1,u_2)$ denotes the Skellam distribution with parameters $u_1$ and $u_2$.
		\end{itemize}
		\textbf{Algorithm:}
		\begin{itemize}
			\item[(1)] The family of probabilities $\{p_V(x,t\vert{x}^{*}({x}_0,V))\}_{t\in[0,T]}$ can be approximated by the formula
			\begin{equation*}
				p_V(x(i),m\Delta{t}\vert{x}^{*}({x}_0,V))\simeq p_{i}(m)\triangleq P(x^{\text{stop}}_V(m\Delta{t})=x(i)),\quad m=0,\cdots,N_t,
			\end{equation*}
			where $p_{i}(m)$ can be calculated by
			\begin{equation*}
				p_{i}(0)=\begin{cases}
					1, \quad x(i)={x}^{*}({x}_0,V),\\
					0, \quad \text{otherwise},
				\end{cases}
			\end{equation*}
			\begin{equation*}
				p_{j}(m+1)=\sum_{i=1}^{N_x+1}p_{i}(m)P_{i,j},\quad j=1,\cdots,N_x+1,\quad m=0,\cdots,N_t-1.
			\end{equation*}
			\item[(2)] The reversed evolution law (\ref{eq030504}) of the process $\bar{x}_V^{\text{NPP}}$ can be described by $\bar{P}^{(m)}_{i,j}$, which is defined for $i,j=1,\cdots,N_x+1$, and $m=0,\cdots,N_t-1$ by
			\begin{equation*}
				\bar{P}^{(m)}_{i,j}=\begin{cases}
					\frac{p_{j}(m)P_{j,i}}{p_{i}(m+1)},\quad &p_{i}(m+1)>0,\\
					0,\quad &p_{i}(m+1)=0.
				\end{cases}
			\end{equation*}
			\item[(3)] The non-stationary prehistory probability ${q}^{\text{NPP}}_{V}(\boldsymbol{x},t;\boldsymbol{x}^{**}(\boldsymbol{x}_T,V),T;\boldsymbol{x}^{*}(\boldsymbol{x}_0,V))$ can be approximated by the formula
			\begin{equation*}
				{q}^{\text{NPP}}_{V}({x}(i),m\Delta{t};{x}^{**}({x}_T,V),T;{x}^{*}({x}_0,V))\simeq \bar{p}_{i}(m),\quad m=0,\cdots,N_t,
			\end{equation*}
			with
			\begin{equation*}
				\bar{p}_{i}(N_t)=\begin{cases}
					1\quad x(i)={x}^{**}({x}_T,V),\\
					0\quad \text{otherwise},
				\end{cases}
			\end{equation*}
			\begin{equation*}
				\bar{p}_{j}(m)=\sum_{i=1}^{N_x+1}\bar{p}_{i}(m+1)\bar{P}^{(m)}_{i,j}\quad j=1,\cdots,N_x+1,\quad m=N_t-1,\cdots,0.
			\end{equation*}
			
		\end{itemize}
		
		\subsection{An Algorithm for Calculating the Stationary Prehistory Probability}
		In comparison with the preceding algorithm, minor adjustments must be made.
		\begin{itemize}
			\item[(1)] Here, we choose a domain $D=[x_l,x_r]\subset\mathbb{R}_{+}$ such that $x_T\in{D}$, the vector field $F(x)$ at the boundary $\partial{D}$ is directed towards the interior, and 
			\begin{equation*}
				\min_{x\in\partial{D}}S(x)>S(x_T).
			\end{equation*}
			In this case, the OP connecting $x_0$ with $x_T$ is entirely contained within $D$. Concentrating on the trajectories that do not leave from the domain $D$ should be sufficiently precise enough.
			\item[(2)] Let $p_i(0)=\pi_V(x(i)),\;i=1,\cdots,N_x$ and $p_{N_x+1}(0)=1-\sum_{i=1}^{N_x}\pi_V(x(i))$, we know that as $V\to\infty$, $p_i(m)$ remains nearly invariant, i.e., $p_i(m)\simeq p_i(0)$ for $i=1,\cdots,N_x+1$.
			\item[(3)] Repeat the step (2) and (3) above, we can obtain the reversed evolution law (\ref{eq030404}) of the process $\bar{x}_{V}^{\text{SPP}}$ and the stationary prehistory probability ${q}^{\text{SPP}}_{V}({x},t;{x}^{**}({x}_T,V),T)$, respectively.
		\end{itemize}
		\begin{remark}
			(a) If $\{\phi(t):t\in[0,T]\}$ is a NOP connecting $x_0$ with $x_T$ in the time span $T$, we know that $\{\phi(t):t\in[m\Delta{t},(m+1)\Delta{t}]\}$ is also a NOP. The preceding algorithm furnishes us with a means to capture the optimal fluctuation from $\phi(m\Delta{t})$ to $\phi((m+1)\Delta{t})$ in the time span $\Delta{t}$. In order to achieve a satisfactory focusing effect of the non-stationary prehistory probability on the NOP, it is necessary that the transition probability $P_{i,j}$ contain all the information of the form $p_V(\phi((m+1)\Delta{t}),\Delta{t}\vert\phi(m\Delta{t}))$. If $S(x_T,T\vert x_0)>0$, then $p_V(\phi((m+1)\Delta{t}),\Delta{t}\vert\phi(m\Delta{t}))$ is exponentially small. For sufficiently large $V$ such that the probability is lower than the machine precision (in MATLAB, this is approximately $10^{-308}$), this quantity will be set to be zero in the computer. It can be deduced that the algorithm is likely to fail. From a computational perspective, in order to enhance the precision of our algorithm, it is imperative to ensure that the noise is not set so weak that the exponentially small probability falls below the minimum limit that the computer can identify. Conversely, it is essential to avoid setting the noise so strong that the incompatibility of the peak trajectory of the non-stationary prehistory probability with the NOP obtained by large deviation theory becomes apparent. This phenomenon also manifests in the stationary setting.
			
			(b) The case of a higher dimension suggests that the optimal fluctuations may be observed less frequently, thereby necessitating an escalation in the computational effort required to analyze such phenomena. Therefore, it is imperative to emphasise that, despite the conclusions being valid in arbitrarily high dimensions, the numerical calculation based on the non-stationary and stationary prehistory probabilities is only useful for systems with lower dimensions due to limitations in computational power. For this reason, we have elected to present examples exclusively with $N = 1$ in this particular paper.
			
			(c) The algorithms presented herein are employed solely to facilitate comprehension of the prehistorical description of the optimal fluctuations. In order to achieve a more precise approximation of the optimal path, reference should be made to the minimum action method \cite{Weinan_2004,Heymann_2008}.
		\end{remark}

		\nocite{*}
		\bibliography{Rerferences}
		
	\end{document}